\documentclass{amsart}
\usepackage[letterpaper,top=3.5cm,bottom=3.5cm,left=3.5cm,right=3.5cm]{geometry}

\usepackage[foot]{amsaddr}
\RequirePackage{amsthm,amsmath,amsfonts,amssymb}
\RequirePackage[numbers]{natbib}
\RequirePackage[colorlinks,linkcolor=,citecolor=blue,urlcolor=blue]{hyperref}
\RequirePackage{graphicx}

\usepackage{mathtools}
\usepackage{xcolor}
\usepackage{hyperref}
\usepackage[capitalize]{cleveref}
\usepackage{physics}
\usepackage{enumerate}

\theoremstyle{plain}
\newtheorem{theorem}{Theorem}
\newtheorem{assumption}{Assumption}
\newtheorem{corollary}{Corollary}
\newtheorem{definition}{Definition}
\newtheorem{proposition}{Proposition}
\newtheorem{itheorem}{Informal Theorem}
\newtheorem{lemma}{Lemma}
\newtheorem*{lemma*}{Lemma}
\newtheorem*{theorem*}{Theorem}
\newtheorem*{proposition*}{Proposition}
\newtheorem{remark}{Remark}
\newtheorem{condition}{Condition}
\crefname{assumption}{assumption}{assumptions}
\crefname{condition}{condition}{conditions}

\usepackage{amsmath,amsfonts,bm}

\def\eqref#1{equation~\ref{#1}}

\def\1{\bm{1}}

\def\rva{{\mathbf{a}}}
\def\rvb{{\mathbf{b}}}

\def\rve{{\mathbf{e}}}

\def\rvg{{\mathbf{g}}}
\def\rvh{{\mathbf{h}}}

\def\rvx{{\mathbf{x}}}
\def\rvy{{\mathbf{y}}}

\def\rmH{{\mathbf{H}}}

\def\rmM{{\mathbf{M}}}

\def\rmX{{\mathbf{X}}}
\def\rmY{{\mathbf{Y}}}
\def\rmZ{{\mathbf{Z}}}

\DeclareMathAlphabet{\mathsfit}{\encodingdefault}{\sfdefault}{m}{sl}
\SetMathAlphabet{\mathsfit}{bold}{\encodingdefault}{\sfdefault}{bx}{n}
\newcommand{\tens}[1]{\bm{\mathsfit{#1}}}

\def\tG{{\tens{G}}}

\def\gC{{\mathcal{C}}}

\def\gE{{\mathcal{E}}}

\def\gJ{{\mathcal{J}}}

\def\gL{{\mathcal{L}}}
\def\gM{{\mathcal{M}}}
\def\gN{{\mathcal{N}}}

\def\gU{{\mathcal{U}}}
\def\gV{{\mathcal{V}}}

\def\gZ{{\mathcal{Z}}}

\newcommand{\R}{\mathbb{R}}

\newcommand{\Var}{\mathrm{Var}}

\DeclareMathOperator{\supp}{supp}
\DeclareMathOperator{\dist}{dist}

\DeclareMathOperator{\diag}{diag}

\begin{document}

\title[Geometric Ergodicity of Affine Invariant Ensemble Langevin]{Geometric Ergodicity of Affine Invariant Ensemble Langevin and its Discrete Time Variants}

\author[Hong Ye Tan]{Hong Ye Tan$^\dagger$}
\address{$^\dagger$ Department of Mathematics, UCLA}
\email{\{hyt35,yifanchen\}@math.ucla.edu}

\author[Yifan Chen]{Yifan Chen$^\dagger$}

\begin{abstract}

Affine-invariant ensemble samplers are widely used in Bayesian applications. However, their quantitative convergence theory, in particular geometric ergodicity, remains a basic open question. We study the affine invariant ensemble Langevin dynamics, an interacting particle system that uses the empirical covariance of the whole ensemble as a preconditioner. While effective in practice, theoretical understanding of this method is not available beyond plain qualitative convergence in total variation, with a central difficulty being that the empirical covariance can become degenerate. For potentials with bounded Hessian that are strongly convex outside a ball, we prove geometric ergodicity using a novel Lyapunov function that combines the inverse covariance with a coercive exponential energy. We then show that directly applying the Euler--Maruyama scheme can diverge with positive probability, even for a one-dimensional Gaussian target. This motivates a covariance-trace regularization in continuous time. We prove geometric ergodicity of the regularized diffusion and, for sufficiently small step size, of its unadjusted Euler--Maruyama discretization. We also show that the invariant distributions of the discretization converge weakly to the product target distribution as the step size tends to zero. Lastly, we provide an explicit discretization that is both affine-invariant and geometrically ergodic, based on an adaptive step size clamping. 
\end{abstract}


\maketitle

\section{Introduction}
Let $\pi$ be a target distribution on $\R^d$ with density 
\begin{align}
    \pi(x) \propto \exp(-f(x))
\end{align}
for a sufficiently regular potential function $f:\R^d \rightarrow \R$. A standard method to sample from $\pi$ is to simulate the overdamped Langevin diffusion
\begin{equation}\label{eq:LangevinDiffusion}
    \dd{\rvx_t} = -\nabla f(\rvx_t) \dd{t} + \sqrt{2} \dd{W_t}
\end{equation}
where $W_t$ is a standard $d$-dimensional Wiener process. Under mild conditions, this diffusion has invariant distribution $\pi$. Langevin dynamics has deep roots
in statistical physics \cite{rossky1978brownian}; we refer to
\cite{pavliotis2014stochastic} for a detailed treatment of the underlying
diffusion theory.

This work studies an ensemble variant of the diffusion \labelcref{eq:LangevinDiffusion}. For $N$ particles $\rvx_1,...,\rvx_N \in \R^d$, let $\rmX \in \R^{N \times d}$ denote the row-stacked collection of all particles. Recall that the empirical mean and covariance corresponding to $\rmX$ are
\begin{equation*}
    \bar \rvx \coloneqq \frac{1}{N}  \sum_{i=1}^N \rvx_i,\qquad C(\rmX) \coloneqq \frac{1}{N} \sum_{i=1}^N (\rvx_i - \bar \rvx)(\rvx_i - \bar \rvx)^\top.
\end{equation*}
The \emph{affine invariant ensemble Langevin dynamics} are given by the following SDEs
\begin{equation} \label{eq:aldiDynamics}
    \dd{\rvx_i} = \left[-C(\rmX) \nabla f(\rvx_i) + \frac{d+1}{N} (\rvx_i  - \bar \rvx)\right] \dd{t} + \sqrt{2 C(\rmX)} \dd{W_i},\quad i=1,...,N,
\end{equation}
where $W_i$ are independent $d$-dimensional Wiener processes. The empirical covariance acts as a preconditioner adapted to the ensemble, and the correction term ensures that the product target distribution $\Pi_* \coloneqq \pi^{\otimes N}$ is invariant for a finite ensemble. Covariance preconditioning of this type was developed in \cite{greengard2015ensemblized,leimkuhler2018ensemble}, and the finite-ensemble correction appears explicitly in \cite{garbuno2020affine}, where the dynamics are termed ALDI.

\textit{Affine invariance} means that transforming the target and initial ensemble by an invertible affine map produces the correspondingly transformed sampling law along the dynamics. Thus, the algorithm's performance is invariant under invertible affine changes of coordinates, making it robust to anisotropy and linear rescaling. The affine-invariant viewpoint was introduced in ensemble MCMC by Goodman and Weare, where the derivative-free stretch and walk moves use ensemble interactions to adapt proposals to the geometry of the target \cite{goodman2010ensemble}. This method has proved to be successful for example in astrophysics, with the open-source Monte Carlo package \texttt{emcee} helping to popularize this approach \cite{foreman2013emcee}. The dynamics in \labelcref{eq:aldiDynamics} incorporate affine invariance into the first-order gradient-based Langevin dynamics \labelcref{eq:LangevinDiffusion}, and they have rich mathematical connections to gradient flows and ensemble Kalman filters via
derivative-free approximations \cite{garbuno2020interacting, garbuno2020affine, chen2026sampling}.

For the ensemble Langevin dynamics, covariance degeneracy presents a new analytical difficulty with no counterpart in the standard overdamped Langevin dynamics. When $\rmX$ approaches the boundary of
the admissible state space $\mathrm{M}\coloneqq\{\rmX\in\R^{N\times d}:\det
C(\rmX)>0\}$, the empirical covariance $C(\rmX)$ becomes singular and the noise
in \labelcref{eq:aldiDynamics} becomes arbitrarily weak in some directions;
controlling this collapse is the central obstacle in the analysis. Under suitable
assumptions, Garbuno-I\~nigo, N\"usken, and Reich proved that the continuous-time
ensemble dynamics preserve nondegeneracy and converge in total variation to
$\Pi_*$, but obtained no quantitative rate \cite{garbuno2020affine}. To the best
of our knowledge, geometric ergodicity has remained a basic open problem for
affine-invariant ensemble samplers in general, including both the finite-particle
dynamics \labelcref{eq:aldiDynamics} and the classical Goodman--Weare stretch and
walk moves.

The goal of this paper is to take a first step toward such a theory by studying
the affine invariant ensemble Langevin dynamics. We begin by establishing
geometric ergodicity of the continuous-time dynamics for a broad class of
potentials. We then turn to the discretizations used in practice and show that
the direct Euler--Maruyama scheme can diverge even for the one-dimensional
standard Gaussian target. The same instability occurs when each particle uses a leave-one-out covariance. To control it, we regularize the continuous-time dynamics so that the preconditioning matrix remains uniformly bounded. We prove geometric ergodicity for both the resulting diffusion and its Euler--Maruyama discretization. Finally, we construct a geometrically ergodic discretization that retains affine invariance by adapting the step size.

\subsection{Main results}
Recall that a Markov process is called \textit{ergodic} if an invariant probability $\pi$ exists, and for any $x$ in its state space,
\begin{equation*}
    \lim_{t \rightarrow \infty} \|P^t(x, \cdot) - \pi\|_{\rm TV}=0,
\end{equation*}
where $P^t$ is the transition semigroup. For a function $f \ge 1$, a positive Harris recurrent Markov process/chain with stationary distribution is $f$-\textit{geometrically ergodic} if there exists a constant $\rho<1$ and a finite function $B(x)$ such that for any $x$ in the state space,
\begin{equation*}
    \|P^t(x, \cdot) - \pi\|_f = \sup_{|g| \le f} \left|P^t_x g - \pi g\right| \le B(x) \rho^t.
\end{equation*}
In particular, $f$-geometric ergodicity implies linear convergence in total variation; for initial probability distributions $\mu$ satisfying $\mu(B) < \infty$, we have $\|\mu P^t  - \pi\|_{\rm TV} \le \mu(B) \rho^t$. In the below statements, geometric ergodicity will refer to $f$-geometric ergodicity with $f= V+1$, where $V$ is the associated Lyapunov function \cite[Sec. 6]{meyn1993stability}. 

We first establish geometric ergodicity for the affine invariant ensemble dynamics \labelcref{eq:aldiDynamics}. Proofs in the main text will be presented in the easier case $N \ge d+3$ for brevity, but all results hold with small modifications for $N \ge d+2$, detailed in \Cref{appsec:dplus2}. 

\begin{itheorem}
 Assume that $f$ has Lipschitz gradient and is strongly convex outside a compact ball, and the number of particles satisfies $N \ge d+2$. Then the ensemble dynamics \labelcref{eq:aldiDynamics} are geometrically ergodic with stationary distribution $\pi^{\otimes N}$. 
\end{itheorem}
Details are given in \Cref{ssec:ctsLyapunov}, and the proof for $N \ge d+3$ is in
\Cref{proof:ctsGeoErgo}. Inspired by \cite{garbuno2020affine}, we employ a Foster--Lyapunov approach. This is done by using an improved Lyapunov function involving the inverse covariance defined by 
\begin{equation*}
    W_a = \exp(\frac{a}{N} \sum_{i=1}^N f(\rvx_i)) \Tr(C(\rmX)^{-1}),
\end{equation*}
for a sufficiently small $a>0$ depending on $f$. The exponential term controls escape to infinity, while the inverse covariance produces a negative drift as the covariance degenerates. Compared with the (non-geometric) ergodicity result of \cite{garbuno2020affine}, assumptions on $f$ and the number of particles are unchanged. For $N = d+1$, the admissible state space is disconnected, and therefore the dynamics can not be ergodic. We note that geometric ergodicity also holds when the function has superquadratic growth outside a compact set; see \Cref{rem:polyTailsBound}. 

Geometric ergodicity gives exponential convergence in total variation. The same Foster--Lyapunov condition can also be used to show a spectral gap and exponential $L^2(\Pi_*)$ convergence, see \Cref{appsec:spectralgap}.

We next turn to discretization. The direct Euler--Maruyama scheme can diverge with positive probability, highlighting the need to control large empirical covariances.

\begin{itheorem}
    Let $d=1$, $f(x)=x^2/2$, and $N\ge2$. For any positive step size $\eta>0$, if the initial empirical variance $C_0$ is sufficiently large, the forward Euler--Maruyama discretization satisfies $C_k\to\infty$ with positive probability. 
\end{itheorem}
Details are given in \Cref{ssec:transient}. The key is that the empirical covariance makes the drift cubic in the ensemble; with positive probability, the noise remains small enough for this drift to drive
the covariance to infinity.

To control this instability, we introduce a time scaled diffusion based on the trace of the covariance, motivated by the tamed unadjusted Langevin algorithm \cite{brosse2019tamed}. For a fixed $\theta>0$, the proposed dynamics are
\begin{subequations} \label{eqs:scalingDiffusionDefinitionIntro}
\begin{gather}
    \dd{\rvx_i} = \left[-\gamma C \nabla f(\rvx_i) + B(C) (\rvx_i - \bar \rvx)\right] \dd{t} + \sqrt{2 \gamma C} \dd{W_i},\quad i=1,...,N,\\
    \gamma \coloneqq\frac{1}{1+\theta  \Tr(C)},\quad B(C) \coloneqq \frac{1}{N} \left[(d+1)\gamma I - 2 \frac{\theta}{(1+\theta \Tr(C))^2} C\right].
\end{gather}
\end{subequations}

The factor $\gamma=\gamma(\Tr C)$ regularizes the covariance-dependent drift, with a modified correction term $B(C)$ so that $\pi^{\otimes N}$ is an invariant distribution of \labelcref{eqs:scalingDiffusionDefinitionIntro}. For small covariance, the dynamics remain close to \labelcref{eq:aldiDynamics}; for large covariance, $\gamma C$ is uniformly bounded. Near covariance collapse, however, the same degeneracy remains, so the boundary still requires separate control. The regularized dynamics are invariant under translations and orthogonal transformations, but not under general affine maps.

The following result shows that these controls suffice to restore geometric ergodicity in both continuous and discrete time.
\begin{itheorem}
\label{thm-informal-trace-reg}
    Assume that $f$ has Lipschitz gradient and is strongly convex outside a compact ball, that the number of particles satisfies $N \ge d+2$, and fix $\theta >0$. Then the regularized diffusion \labelcref{eqs:scalingDiffusionDefinitionIntro} is geometrically ergodic with invariant distribution $\Pi_* = \pi^{\otimes N}$.

    Moreover, for all sufficiently small step sizes $\eta$, the Euler--Maruyama discretizations are also geometrically ergodic, with stationary distributions $\Pi_{*, \eta}$ satisfying
    \[\Pi_{*, \eta} \rightharpoonup \pi^{\otimes N} \quad\text{weakly  as }\eta \rightarrow 0.\]
\end{itheorem}
These results are stated for $N \ge d+3$ in \Cref{ssec:adapTime,ssec:discAdapGeo}, and the
corresponding proofs are given in
\Cref{ssec:CtsTimeScaled,ssec:discreteTimeErgo}.

While regularizing the diffusion stabilizes its discretization, the resulting algorithm loses affine invariance. To retain affine invariance, we propose an alternative affine-invariant step size in the discretization of \labelcref{eq:aldiDynamics}. For a step size $\eta>0$ and fixed regularization parameter $\delta>0$, the proposed dynamics are
\begin{subequations} \label{eqs:introTamedDynamics}
    \begin{gather}
    D(\rmX)
    =\frac1N\sum_{i=1}^N\nabla f(\rvx_i)^\top C(\rmX)\nabla f(\rvx_i), \quad  h_{\eta,\delta}(\rmX)
    =\min\left\{\eta,\frac{\delta}{\sqrt{D(\rmX)}}\right\},\label{eq:introTamedStepSize}\\
    \rvx_i^{(k+1)}
    =\rvx_i^{(k)}+h_{\eta,\delta}
    \left[-C\nabla f(\rvx_i^{(k)})+\frac{d+1}{N}(\rvx_i^{(k)}-\bar\rvx^{(k)})\right]
    +\sqrt{2h_{\eta,\delta}C_k}\xi_i^{(k)},
\end{gather}
\end{subequations}
where we define $h_{\eta,\delta}=\eta$ if $D=0$. 

\begin{itheorem}
    The discretized dynamics \labelcref{eqs:introTamedDynamics} are affine invariant. Assume in addition that $f$ has Lipschitz gradient and is strongly convex outside a compact ball, that $N \ge d+2$. Then for sufficiently small $\delta>0$, the dynamics are geometrically ergodic for all sufficiently small step sizes $\eta$ depending on $\delta$.

    Moreover, for any fixed admissible $\delta$, the stationary distributions $\Pi_{\eta,\delta}$ satisfy
    \[\Pi_{\eta,\delta} \rightharpoonup \pi^{\otimes N} \quad\text{weakly  as }\eta \rightarrow 0.\]
\end{itheorem}
This is stated for $N \ge d+2$ in \Cref{ssec:clampingdiscupdate} and proved in \Cref{ssec:proofTamedALDI}. It relies on an alternative Lyapunov function detailed in \Cref{appsec:dplus2}. We therefore obtain geometric ergodicity for a discrete,  affine invariant sampler.
\color{black}
\subsection{Related work}

Affine-invariant ensemble sampling was introduced by Goodman and Weare through the derivative-free stretch and walk moves \cite{goodman2010ensemble}. Later gradient-based developments include first and second-order ensemble Langevin methods \cite{greengard2015ensemblized,leimkuhler2018ensemble,garbuno2020affine,liu2025second} and ensemble Hamiltonian Monte Carlo \cite{chen2025new}. Affine-invariant sampling has also been formulated at the level of gradient flows \cite{chen2026sampling}.

In the mean field limit, the dynamics \labelcref{eq:aldiDynamics} formally become
the covariance-preconditioned Langevin equation \cite{garbuno2020affine}
\begin{equation}\label{eq:covLangevin}
    \dd{\rvx_t} = - C(\rho_t) \nabla f(\rvx_t) \dd{t} + \sqrt{2 C(\rho_t)} \dd{W_t},\quad \rho_t = \mathrm{Law}(\rvx_t),
\end{equation}
where $C(\rho_t)$ is the covariance of $\rho_t$. A closely related method for Bayesian inverse problems sharing the same mean field limit (in the Gaussian setting) is the ensemble Kalman sampler \cite{garbuno2020interacting}, for which quantitative mean field convergence is available for linear inverse problems \cite{ding2021ensemble}. In this case, the mean-field equation admits explicit solutions and convergence estimates in $L^1$ and Wasserstein-2 distance \cite{garbuno2020interacting,carrillo2021wasserstein,burger2025covariance}. Second-order ensemble Langevin dynamics have likewise been studied at the mean-field level for linear inverse problems, with a characterization of stationary distributions and local convergence \cite{liu2025second}. We refer to \cite{calvello2025ensemble} for a broader overview of mean-field theory for ensemble Kalman methods.

Geometric ergodicity of Langevin dynamics and of their discretizations is also
well studied, under a variety of assumptions on the potential $f$ and by a
variety of techniques. Coupling, functional inequalities, and estimates of the
error between the diffusion and its discretization give nonasymptotic bounds in
total variation, Wasserstein distance, and relative entropy
\cite{dalalyan2017theoretical,cheng2018convergence, eberle2016reflection, vempala2019rapid,chewi2025analysis}.
Foster--Lyapunov theory combines a drift condition in the tails with local mixing
on petite sets, though its rate constants are often less explicit
\cite{meyn1993stability,durmus2022geometric,pereyra2016proximal}; see
\cite{mattingly2002ergodicity} for locally Lipschitz drifts and hypoellipticity
arguments. Our proofs follow the Foster--Lyapunov route, but must in addition
control the boundary where the empirical covariance becomes singular.

\subsection{Notation}
Throughout, $\rmX \in \R^{N \times d}$ will denote a collection of particles $\rvx_1,...,\rvx_N \in \R^d$, and without loss of generality $f:\R^d \rightarrow \R_{\ge 0}$ is non-negative. For a square matrix $A$, $\|A\|$ denotes the operator norm, while $\|A\|_F$ denotes the Frobenius norm. For a test function $\varphi(\rmX)$, we let $\nabla_i \varphi$ denote the gradient with respect to $\rvx_i$, and similarly $\nabla_i^2 \varphi$ the Hessian with respect to $\rvx_i$. For two matrices $A,B$ of the same shape, we let the Frobenius inner product be denoted
\begin{equation*}
    A\!:\!B = \Tr(A^\top B) = \sum_{i,j}A_{ij}B_{ij}.
\end{equation*}
We let the open admissible set be denoted $\mathrm{M} = \{\rmX \in \R^{N \times d} \mid \det C(\rmX) >0\}$, with null complement $\mathrm{M}^c$ with respect to the Lebesgue measure for $N \ge d+1$. We further let $\Pi_* \propto \exp(-\sum_i f(\rvx_i))$ denote the target product density, using the same notation between the full Lebesgue density on $\R^{N \times d}$ and restricted to $\mathrm{M}$. We denote by $\gL$ the generator of the continuous-time ensemble dynamics \labelcref{eq:aldiDynamics}.

The rest of the paper is organized as follows. \Cref{sec:ctsGeoErgodic} develops the continuous-time Lyapunov argument. \Cref{sec:discrete} proves divergence of the direct discretization and introduces the regularized diffusion. The proofs of the main results are collected in \Cref{sec:proofs}, with supporting definitions and auxiliary calculations in the appendix.

\section{Continuous time: geometric ergodicity} \label{sec:ctsGeoErgodic}
We first briefly detail the proof strategy for ergodicity in \cite{garbuno2020affine}. A Foster--Lyapunov condition shows non-explosiveness assuming the initial condition lies within the admissible set $\mathrm{M}$ \cite[Sec. 3]{meyn1993stability}. An ellipticity plus irreducibility argument within $\mathrm{M}$ then shows positive (Harris) recurrence \cite[Thm. 4.1]{kliemann1987recurrence} and therefore ergodicity.  Supporting definitions are recalled in \Cref{appsec:defs}.

For the ensemble dynamics, the state space is $\mathrm{M} \subset \R^{N \times d}$. A nonnegative real-valued function $V:\mathrm{M} \rightarrow \R_{\ge0}$ is \emph{norm-like} if $V \rightarrow \infty$ as either $\rmX \rightarrow \infty$ or $\rmX \rightarrow \partial \mathrm{M} = \{\det C = 0\}$. We now recall the existing ergodicity result.

\begin{proposition}[{\cite[Prop. 4.4]{garbuno2020affine}}] \label{prop:ALDIRegErgo}
    Assume that $f \in \gC^2 \cap L^1(\pi)$, and further that there exists a compact set $\mathrm{K} \subset \R^d$ and constants $0 < c_1 < c_2$ such that for all $x \in \R^d \setminus \mathrm{K}$,
    \begin{subequations}
        \begin{gather}
            c_1 \|x\|^2 \le f(x) \le c_2 \|x\|^2,\label{eq:ALDIAsmp1}\\
            c_1 \|x\| \le \|\nabla f(x)\| \le c_2 \|x\|,\label{eq:ALDIAsmp2}\\
            c_1 I_d \preceq \nabla^2 f(x) \preceq c_2 I_d.\label{eq:HessBddAsmp}
        \end{gather}
    \end{subequations}
    Then for $N \ge d+2$, the ensemble dynamics \labelcref{eq:aldiDynamics} are ergodic, converging in total variation to the stationary distribution $\Pi_*$.
\end{proposition}
The assumptions in \Cref{prop:ALDIRegErgo} are rather weak, and can in fact be reduced to \labelcref{eq:HessBddAsmp} alone, since it implies the other two. For a test function $V(\rmX)$, the generator $\gL$ of \labelcref{eq:aldiDynamics} is given by:
\begin{align}
    \gL V(\rmX) = \sum_i \left(-C \nabla f(\rvx_i) + \frac{d+1}{N} (\rvx_i - \bar \rvx)\right) \cdot \nabla_i V +  \sum_i C \!:\! \nabla_i^2 V.\label{eq:ALDIGeneratorDef}
\end{align}
To show \Cref{prop:ALDIRegErgo}, the Lyapunov function identified in \cite{garbuno2020affine} is
\begin{equation}\label{eq:ALDILogdet}
    \gV(\rmX) = \sum_i f(\rvx_i) - \frac{d+1}{2} \log\det C(\rmX).
\end{equation}
Under some additional conditions satisfied by the diffusion, the norm-like property of $\gV$ and the Foster--Lyapunov condition \cite[Thm. 2.1]{meyn1993stability}
\begin{equation}
    \gL \gV \lesssim 1 + \gV
\end{equation}
give non-explosiveness of the process. The argument then concludes as follows. Since the covariance is full rank in $M$, the diffusion is elliptic in the state space. 
Since $\Pi_*$ is invariant and has positive Lebesgue density on $\mathrm{M}$, an invariant probability measure exists.
Using that the state space $\mathrm{M}$ is an \emph{invariant control set}, \cite[Thm. 4.1]{kliemann1987recurrence} gives that the diffusion is positive recurrent. This plus Harris recurrence, shown in \Cref{appsec:HarrisDiffusion}, gives ergodicity in the sense of Meyn--Tweedie \cite[Thm. 6.1]{meyn1993stability2}.

\subsection{Log determinant is insufficient}\label{ssec:logdet}
To upgrade from ergodicity to geometric ergodicity, a typical route is the following stronger Foster--Lyapunov condition \cite[Thm. 6.1]{meyn1993stability}: for some $c>0$, $b \in \R$, and a compact set $\mathrm{K} \subset \mathrm{M}$,
\begin{equation}\label{eq:geoFosterLyapunov}
    \gL \gV \le -c \gV + b \mathbf{1}_\mathrm{K}.
\end{equation}

However, this condition does not hold for the Lyapunov function \labelcref{eq:ALDILogdet}. We will show this by exhibiting a sequence $\rmX \rightarrow \partial \mathrm{M}$ where $\gL \gV$ does not tend to $-\infty$.

Consider the Gaussian target $f(x) = \|x\|^2/2$ in any dimension $d\ge 1$. Applying $\gL$ to $\log \det C$ yields
\begin{align*}
    \gL \log \det C &= 2d\left(1 - \frac{1}{N}\right) - \frac{2}{N} \sum_i (\rvx_i - \bar \rvx)\cdot \nabla f(\rvx_i)\\
    &=2d\left(1-\frac{1}{N}\right) - 2 \Tr(C)
\end{align*}
The other term gives
\begin{align*}
    \gL \sum_i f(\rvx_i) &= \sum_i \left(-C\rvx_i + \frac{d+1}{N}(\rvx_i - \bar \rvx)\right)\cdot \rvx_i + N \Tr(C)\\
    &= -\sum_i \rvx_i^\top C \rvx_i + (N+d+1) \Tr(C)
\end{align*}
Combining the two identities, the Lyapunov function \labelcref{eq:ALDILogdet} satisfies
\begin{align*}
     &\quad \gL \left(\sum_i f(\rvx_i) - \frac{d+1}{2}\log\det C \right) \\
     &= -\sum_i \rvx_i^\top C \rvx_i- d(d+1)\left(1-\frac{1}{N}\right)+ (N+2d+2) \Tr(C)\\
     &= -N \Tr(C^2) - N \bar \rvx^\top C \bar \rvx - d(d+1)\left(1-\frac{1}{N}\right) + (N+2d+2) \Tr(C).
\end{align*}
Take any sequence of zero mean ensembles with eigenvalues of $C$ all being $\lambda$ for $\lambda \rightarrow 0$. Then, $\Tr(C)$ and $\Tr(C^2)$ tend to 0, hence $\gL \gV$ is bounded below. However, the ensembles converge to $\partial \mathrm{M}$, so $\gV \rightarrow \infty$. Therefore, no $c>0$ exists such that the geometric Foster--Lyapunov criterion \labelcref{eq:geoFosterLyapunov} is satisfied.

A similar argument shows that any Lyapunov function whose drift $\gL \gV$ involves only positive powers of $C$ cannot satisfy \labelcref{eq:geoFosterLyapunov}: such terms vanish as $\rmX \rightarrow \partial \mathrm{M}$, and so do not supply the negative drift needed near the boundary.

\subsection{New Lyapunov function}\label{ssec:ctsLyapunov}
While the log determinant Lyapunov function is insufficient to show geometric ergodicity, it is still possible to find a Lyapunov function that satisfies the Foster--Lyapunov condition \labelcref{eq:geoFosterLyapunov}. This is based on the inverse covariance, with suitable modifications to ensure that the Lyapunov function works in high dimensions.

We use the same assumption as in \cite{garbuno2020affine}, namely bounded Hessian and strong convexity outside a compact ball. Variants of this are standard in the convergence literature for Lyapunov-like analyses of Langevin algorithms \cite{durmus2017nonasymptotic,durmus2022geometric,pereyra2016proximal,eberle2016reflection,gorham2019measuring}; such conditions avoid assuming global strong convexity.
\begin{assumption}[Distant convexity]\label{assmp:distantConvexityND}
    The potential $f: \R^d \rightarrow \R_{\ge 0}$ is $\gC^2$, and there exists a constant $L>0$ such that $\sup_x\|\nabla^2 f(x)\| \le L$. Furthermore, there exists a $\mu>0$ such that outside a compact set, the Hessian is uniformly positive $\nabla^2 f \succeq \mu I_d$.
\end{assumption}
Motivated by the insufficiency of positive powers of covariance in the drift, we can use $\Tr(C^{-1})$ within a Lyapunov function, which will result in inverse covariance terms after applying the generator. Further modifying the coercivity term to deal with covariance collapse, this allows us to show a Foster--Lyapunov condition, i.e. decay of a Lyapunov function outside some compact set.
\begin{theorem}\label{thm:ctsGeoErgo}
    Suppose that $f$ satisfies \Cref{assmp:distantConvexityND} and consider the affine invariant Langevin dynamics \labelcref{eq:aldiDynamics}. If $N \ge d+3$, then for sufficiently small $a>0$, 
    \begin{equation}\label{eq:LyapunovWa}
        W_a = e^{\frac{a}{N} \sum f(\rvx_i)} \Tr(C^{-1})
    \end{equation}
    satisfies the Foster--Lyapunov condition
    \begin{equation}
        \gL W_a \le -c W_a + b \mathbf{1}_\mathrm{K}
    \end{equation}
    for some constants $c>0,\, b \in \R$ and compact $\mathrm{K} \subset \mathrm{M}$ depending on $a$. 
\end{theorem}
The proof is deferred to \Cref{proof:ctsGeoErgo}. Such a Foster--Lyapunov condition almost immediately implies geometric ergodicity \cite[Thm. 6.1]{meyn1993stability}.

\begin{corollary}\label{cor:geoErgodicWa}
    Under \Cref{assmp:distantConvexityND} and for $N \ge d+3$, the ensemble dynamics \labelcref{eq:aldiDynamics} are geometrically ergodic.
\end{corollary}
\begin{proof}
    To apply \cite[Thm. 6.1]{meyn1993stability}, it remains to show that $W_a$ is norm-like. It is positive on $\mathrm{M}$. Since the exponential is at least $1$, we have that $W_a \rightarrow \infty$ as $\rmX \rightarrow \partial \mathrm{M}$.

    We now show that $W_a \rightarrow \infty$ as $\rmX \rightarrow \infty$. By Cauchy--Schwarz, $\Tr(C) \Tr(C^{-1}) \ge d^2$, and by \Cref{prop:DistantConvexity}, there exists $b_1>0$ such that the following bounds hold
    \begin{equation*}
        W_a = e^{\frac{a}{N}\sum_i f(\rvx_i)} \Tr(C^{-1}) \ge \frac{d^2 e^{\frac{a}{N}\sum_i f(\rvx_i)}}{\Tr(C)} \ge \frac{d^2 N e^{\frac{a}{N}\sum_i f(\rvx_i)}}{\|\rmX\|_F^2} \gtrsim\frac{e^{\frac{a}{N}\sum_i f(\rvx_i)}}{\sum_i f(\rvx_i) + b_1}.
    \end{equation*}
    Since $f$ is coercive, $\rmX\rightarrow \infty$ implies $\sum f(\rvx_i) \rightarrow \infty$ and therefore $W_a \rightarrow \infty$.
\end{proof}

\begin{remark}
    In one dimension, an alternative Lyapunov function is $(1+\sum f(\rvx_i)) C^{-\kappa}$ for some sufficiently small $\kappa \in (0,1)$, valid for $N \ge 3$. This motivates a similar modification to the Lyapunov function by taking an appropriate power of $\Tr(C^{-1})$, which we do in \Cref{appsec:dplus2}.
\end{remark}

\begin{remark}\label{rem:polyTailsBound}
    Geometric ergodicity of the continuous dynamics also holds under the following more general tail condition: $f \in \gC^2$ is bounded below by $1$, and there exist constants $\ell \ge 0$ and $c_1, c_2>0$ such that outside a compact set,
    \begin{align*}
        &c_1 \|x\|^{\ell+2} \le f(x) \le c_2 \|x\|^{\ell+2},\\
        &c_1 \|x\|^{\ell+1} \le \|\nabla f(x)\| \le c_2 \|x\|^{\ell+1},\\
        & c_1 \|x\|^\ell I_d \preceq \nabla^2 f(x) \preceq c_2 \|x\|^\ell I_d.
    \end{align*}
    This assumption is used for example in \cite{vaes2024sharp} for ensemble samplers, and can be applied to potentials like $f(x) = \|x\|^4$. The proof of this extension is deferred to \Cref{appsec:polyTailsBound}.
\end{remark}

\section{Discrete time divergence and geometric ergodicity of a regularized scheme}\label{sec:discrete}
The previous section shows that the continuous time flow is geometrically ergodic through a tailored Lyapunov function. To obtain geometric ergodicity for time discretizations, a common technique is to show that the discrete Markov kernel approximates the continuous semigroup, then take a sufficiently small step size such that the resulting Markov chain also satisfies a Foster--Lyapunov condition in the sense of \cite[Sec. 6]{meyn1992stability1}

For the affine invariant ensemble Langevin dynamics, a uniform comparison breaks down because large empirical covariance makes the drift non-globally Lipschitz. In fact, even for a one-dimensional Gaussian target, the direct Euler--Maruyama scheme can diverge for every positive step size.
More precisely, the forward Euler--Maruyama discretization of \labelcref{eq:aldiDynamics} with step size $\eta>0$ is given in \cite{garbuno2020affine} by
\begin{subequations}
    \begin{gather}
        \rvx_i^{(k+1)} = \rvx_i^{(k)} + \eta \left[-C_{k} \nabla f(\rvx_i^{(k)}) + \frac{d+1}{N}(\rvx_i^{(k)} - \bar \rvx^{(k)})\right] + \sqrt{2\eta} S_{k} \xi_i^{(k)} \label{eq:DiscretizedALDI},\\
        S(\rmX) \coloneqq \frac{1}{\sqrt{N}} \begin{bmatrix}
            \rvx_1 - \bar \rvx & ... & \rvx_N - \bar \rvx
        \end{bmatrix} = \frac{1}{\sqrt{N}}(\rmX^\top - \bar \rvx \mathbf{1}_N^\top)\label{eq:nonsymRoot}
    \end{gather}
\end{subequations}
where $S(\rmX)$ is a rectangular non-symmetric square root of the empirical covariance satisfying $C(\rmX) = SS^\top$, and $\xi_i^{(k)}$ are standard Gaussian vectors of appropriate length, in this case $N$-dimensional.

\subsection{Counterexample: direct Euler can diverge for 1D Gaussian}\label{ssec:transient}
For the one-dimensional Gaussian target $f(x) = x^2/2$, the discretized iterations \labelcref{eq:DiscretizedALDI} simplify to 
\begin{equation}
    x_j^{(k+1)} = x_j^{(k)} -\eta C_k x_j^{(k)} + \frac{2\eta}{N} (x_j^{(k)} - \bar x^{(k)}) + \sqrt{2 \eta C_k } \xi_{j}^{(k)}
\end{equation}
for some i.i.d.\ standard 1D Gaussians $\xi_j^{(k)}$. Suppose $N \ge 2$. Define the zero mean vector $\rmY^{(k)} = P \rmX^{(k)}$, where $P = I_N - \frac{1}{N} \mathbf{1}_N\mathbf{1}_N^\top$, so that $C_k = N^{-1}\|\rmY^{(k)}\|_F^2$. Then, the update for $\rmY^{(k)}$ satisfies 
\begin{equation}\label{eq:DiscreteALDIRecursionGaussian}
    \rmY^{(k+1)} = (1 - \eta C_k + \frac{2\eta }{N}) \rmY^{(k)} + \sqrt{2\eta C_k}P \boldsymbol{\xi}^{(k)},
\end{equation}
where $\boldsymbol{\xi}^{(k)}$ are $N$-dimensional standard Gaussians. Observe that the drift component contains a cubic term, hence is not globally Lipschitz. We will show that, for sufficiently large initial
$C_0$, this cubic term can drive the ensemble covariance to infinity with positive probability.

Fix any step size $\eta>0$, and define the function $ \omega(c) = 1 - \eta c + \frac{2\eta}{N}$ for $c>0$, so that the recursion \labelcref{eq:DiscreteALDIRecursionGaussian} is
\begin{equation*}
        \rmY^{(k+1)} = \omega(C_k) \rmY^{(k)} + \sqrt{2\eta C_k} P \boldsymbol{\xi}^{(k)}.
\end{equation*}
Consider the events
\begin{equation*}
    A_k \coloneqq \left\{\|P\boldsymbol{\xi}^{(k)}\|_F \le \frac{|\omega(C_k)|\sqrt{N}}{2\sqrt{2\eta} }\right\}.
\end{equation*}

Along the trajectory we will construct, $C_k$ increases, so $\mathbb{P}(A_k)$ tends to $1$ quickly as $k \rightarrow \infty$. Noting that $\|\rmY^{(k)}\|_F = \sqrt{NC_k}$, and conditioned on previous events $\cap_{j \le k} A_j$, the norm update satisfies
\begin{align*}
    \|\rmY^{(k+1)}\|_F &\ge |\omega(C_k)| \|\rmY^{(k)}\|_F - \sqrt{2\eta C_k} \|P\boldsymbol{\xi}^{(k)}\|_F \ge \frac{1}{2}|\omega(C_k)| \|\rmY^{(k)}\|_F,
\end{align*}
therefore the covariance update satisfies
\begin{equation}\label{eq:CovRecursion}
    C_{k+1} \ge \frac{\omega(C_k)^2}{4} C_k.
\end{equation}

Since $|\omega(c)|/c \rightarrow \eta$ as $c \rightarrow \infty$, choose $C_0$ sufficiently large such that for all $c \ge C_0$, it holds that $\omega(c)^2/4 \ge e^\lambda$ for some constant $\lambda>0$ to be chosen later. The recursion \labelcref{eq:CovRecursion} then gives $C_k \gtrsim e^{\lambda k}$. The probability of the event $A_k$ is then lower bounded by 
\begin{align*}
    \mathbb{P}(A_k \mid \cap_{j<k} A_j) \ge 1 - \exp(-B \lambda^2 k^2)
\end{align*}
for some constant $B>0$ depending only on $N$ and $\eta$. Let $p_0\coloneqq\mathbb P(A_0)>0$. The probability that all events occur conditioned on sufficiently large initial covariance is bounded below by
\begin{align*}
    \mathbb{P}(A_k ,\ \forall k \ge 0) &\ge {p_0}\prod_{k \ge 1}\left[1 - \exp(-B \lambda^2 k^2)\right] \\
    &\ge {p_0\left[1 - \sum_{k \ge 1} \exp(- B \lambda^2 k^2)\right]}.
\end{align*}
Choose $\lambda$ sufficiently large that the sum is less than $1$, so that all the events $A_k$ hold simultaneously with positive probability. On this event $C_k \to \infty$, which proves positive-probability divergence.

\subsection{Counterexample:  leave-one-out Euler can also diverge}\label{ssec:leaveoneout}
A similar approach shows that the discretized \textit{leave-one-out} variant can fail in the same way. Here the covariance preconditioning applied to any particle is the empirical covariance of the remaining particles. This variant is proposed for example in \cite{nusken2019note,leimkuhler2018ensemble} to avoid the need for a correction term, while retaining $\Pi_*$ as an invariant distribution in continuous time. The discretized dynamics are 
\begin{equation}
    \rvx_i^{(k+1)} = \rvx_i^{(k)} - \eta C^{(k)}_{-i} \nabla f(\rvx_i^{(k)}) + \sqrt{2\eta C^{(k)}_{-i}}  \xi_i^{(k)} ,\label{eq:leaveoneout}
\end{equation}
where $C_{-i}$ denotes the empirical covariance of the particles $\{\rvx_j \mid j \ne i\}$. We now show that for $N \ge 3$ and the 1D Gaussian target $f(x) = x^2/2$, the covariance explodes with positive probability because of the cubic interaction in the drift. This is based on a two-step recurrence in order to couple the effect of any given particle back to itself. We defer the proof of the following proposition to \Cref{appsec:leaveOneOutIneqs}. 

\begin{proposition}\label{prop:leaveoneout}
    Let $d=1$, $f(x) = x^2/2$, and $N \ge 3$. For any positive step size $\eta>0$, if the initial empirical covariance is sufficiently large, then the discretized leave-one-out scheme \labelcref{eq:leaveoneout} diverges with positive probability.
\end{proposition}

\subsection{Modified diffusion through covariance taming}\label{ssec:adapTime}
To address the divergence in discrete time, we introduce a regularization on both the drift and diffusion to cap the large-covariance regime. For a scalar function $\gamma:\R_+ \rightarrow \R_+$ evaluated at $T = \Tr(C(\rmX))$, we propose the following modification:
\begin{subequations} \label{eqs:scalingDiffusionDefinition}
\begin{gather}
    \dd{\rvx_i} = \left[-\gamma C \nabla f(\rvx_i) + B(C) (\rvx_i - \bar \rvx)\right] \dd{t} + \sqrt{2 \gamma C} \dd{W_i},\quad i=1,...,N,\\
    B(C) \coloneqq \frac{1}{N} \left[(d+1)\gamma I_d + 2 \gamma' C\right]. \label{eq:BDef}
\end{gather}
\end{subequations}
When $\gamma \equiv 1$ and $\gamma'=0$, \labelcref{eqs:scalingDiffusionDefinition} recovers the original dynamics \labelcref{eq:aldiDynamics}. The modified dynamics \labelcref{eqs:scalingDiffusionDefinition} are in general no longer scale invariant, but are still invariant under translation and orthogonal transformations.

Motivated by the divergence in \Cref{ssec:transient}, we use the following regularization factor and
its derivative: for a fixed $\theta>0$,
\begin{equation}\label{eq:timeScaleDef}
    \gamma: [0, \infty)\rightarrow (0, 1],\quad \gamma(T) = \frac{1}{1 + \theta T},\quad \gamma'(T) = -\theta \gamma^2.
\end{equation}

For small $T$, the factor $\gamma(T)$ is close to one; for large $T$, it keeps $\gamma(T)C$ bounded. We first verify that the modified diffusion preserves the target distribution and then prove geometric ergodicity for the diffusion and its discretization.
\begin{lemma}\label{lem:generatorALDITimescale}
    Let $\gL$ and $\bar \gL$ be the generators of the diffusions \labelcref{eq:aldiDynamics,eqs:scalingDiffusionDefinition} respectively. Letting $\Gamma$ be the carr\'e-du-champ operator of $\gL$,
    \begin{equation*}
    \Gamma(u,v) = \sum_i (\nabla_i u)^\top C (\nabla_i v),
    \end{equation*}
    the two generators are related by
    \begin{equation*}
        \bar \gL V = \gamma \gL V + \gamma' \Gamma(\Tr C, V). 
    \end{equation*}
\end{lemma}
\begin{proof}
    Direct computation yields
    \begin{align*}
        \bar \gL V - \gamma \gL V &=  \sum_j \left(\frac{2}{N} \gamma' C(\rvx_j - \bar \rvx)\right) \cdot \nabla_j V\\
        &=  \gamma' \sum_j (\nabla_j \Tr C)^\top C \nabla_jV = \gamma' \Gamma(\Tr C,V).
    \end{align*}
\end{proof}
The correction term in $B(C)$ is derived such that the target distribution $\Pi_*$ is a stationary distribution of the scaled dynamics.

\begin{proposition}
    For the scaled diffusion \labelcref{eq:timeScaleDef}, $\Pi_*$ is a stationary distribution of the diffusion \labelcref{eqs:scalingDiffusionDefinition}.
\end{proposition}
\begin{proof}
    The original generator $\gL$ can equivalently be written in divergence form as 
    \begin{align*}
        \gL V &= \sum_i \left[-C\nabla f(\rvx_i) + \nabla_i \cdot C\right] \cdot \nabla_i V + \sum_i\Tr(C \nabla_i^2 V)\\
        &= \frac{1}{\Pi_*} \sum_i \nabla_i \cdot (\Pi_* C \nabla_i V).
    \end{align*}
 Taking adjoints in the \Cref{lem:generatorALDITimescale} yields
    \begin{align*}
        \bar \gL^\dagger \Pi_* &= \gL^\dagger (\gamma \Pi_*) - \sum_i \nabla_i \cdot (C (\nabla_i \gamma) \Pi_*)\\
        &= \Pi_* \gL \gamma - \sum_i \nabla_i \cdot (C (\nabla_i \gamma) \Pi_*) = 0.
    \end{align*}
\end{proof}

Geometric ergodicity of the continuous time modified diffusion can be proved similarly to \Cref{thm:ctsGeoErgo}. In particular, the regularized dynamics also satisfy a Foster--Lyapunov condition with the same Lyapunov function.

\begin{theorem}\label{thm:timeScaleErgo}
    Assume that $f$ satisfies \Cref{assmp:distantConvexityND}, and let $W_a$ be the Lyapunov function \labelcref{eq:LyapunovWa}. Then for $N \ge d+3$ and sufficiently small $a>0$, the generator of the regularized diffusion \labelcref{eqs:scalingDiffusionDefinition,eq:timeScaleDef}
    satisfies the Foster--Lyapunov condition
    \begin{equation}
        \bar \gL W_a \le -c W_a + b \mathbf{1}_\mathrm{K}
    \end{equation}
    for some $c>0,\, b \in \R$ and some compact $\mathrm{K} \subset M$. In particular, the diffusion is geometrically ergodic with stationary distribution $\Pi_*$.
\end{theorem}
The proof is deferred to \Cref{ssec:CtsTimeScaled}.

\subsubsection{Discrete-time geometric ergodicity}\label{ssec:discAdapGeo}
The forward Euler–Maruyama discretization of
the regularized diffusion \labelcref{eqs:scalingDiffusionDefinition} is defined as follows. For a step size $\eta>0$, let $C_k$ be the covariance at iteration $k$, and $B_k = B(C_k), \gamma_k = \gamma(\Tr C_k), \gamma'_k = \gamma'(\Tr C_k)$ as given in \labelcref{eq:BDef,eq:timeScaleDef}. The updates are
\begin{equation}\label{eq:discreteTimeScaleDef}
    \rvx_{i}^{(k+1)} = \rvx_i^{(k)} + \eta \left[-\gamma_k C_k \nabla f(\rvx_i^{(k)}) + B_k (\rvx_i^{(k)} - \bar \rvx^{(k)})\right] + \sqrt{2 \eta \gamma_k C_k} \xi_i^{(k)}
\end{equation}
for some i.i.d. $d$-dimensional standard Gaussians $\xi_j^{(k)}$. The introduction of the scaling $\gamma$ allows us to show that the drift scales at most linearly as $\|\rmX\|_F \rightarrow \infty$. Furthermore, since we may choose the same Lyapunov function for all sufficiently small step size, we can obtain convergence of the stationary distributions.

\begin{theorem}\label{thm:discreteTimeScaleErgo}
    Assume $f$ satisfies \Cref{assmp:distantConvexityND}, that $N \ge d+3$, and let $W_a$ be the Lyapunov function \labelcref{eq:LyapunovWa}. There exist constants $a,b,c>0$, compact $\mathrm{K} \subset M$, and a threshold $\eta_* >0$ such that for all $0<\eta < \eta_*$, the Markov kernel $Q_\eta$ of the discretized regularized diffusion \labelcref{eq:discreteTimeScaleDef} satisfies
    \begin{equation} \label{eq:DiscreteTimeLyapunovDecay}
        Q_\eta W_a \le (1 - c\eta) W_a + b \eta \mathbf{1}_{\mathrm{K}}.
    \end{equation}
    In particular, there exists a unique  invariant probability distribution $\Pi_{*, \eta}$ such that for any initial state $\rmX \in \mathrm{M}$, $Q_\eta^k (\rmX, \cdot)$ converges geometrically in total variation to $\Pi_{*, \eta}$ as $k \rightarrow \infty$. Furthermore, $\Pi_{*, \eta} \rightharpoonup \Pi_*$ weakly as $\eta \rightarrow 0$.
\end{theorem}
The proof is deferred to \Cref{ssec:discreteTimeErgo}. A key technical difficulty is a Taylor expansion with bounded remainder for matrix inverses, utilizing Gaussian concentration for uniform bounds away from $\partial \mathrm{M}$.

\subsection{Affine-invariant discretization through step size clamping}\label{ssec:clampingdiscupdate}
The previous approach leads to discretizations that lose affine invariance. In this section, we discretize the original dynamics \labelcref{eq:aldiDynamics} using a state-dependent, affine-invariant step size. This step size is capped at a prescribed value $\eta>0$, and decreases when the preconditioned gradient update becomes large in order to control the drift.

For ensembles $\rmX = \{\rvx_1,...,\rvx_N\}$, define the function and capped step size
\begin{gather}
    D(\rmX)
    =\frac1N\sum_{i=1}^N\nabla f(\rvx_i)^\top C(\rmX)\nabla f(\rvx_i), \quad  h_{\eta,\delta}(\rmX)
    =\min\left\{\eta,\frac{\delta}{\sqrt{D(\rmX)}}\right\},\label{eq:tamedStepSize}
\end{gather}
where we define $h_{\eta,\delta}=\eta$ if $D=0$. For a step size $\eta>0$ and regularization parameter $\delta>0$, we propose the gradient-tamed update 
\begin{equation}\label{eq:tamedALDI}
    \rvx_i^{(k+1)}
    =\rvx_i^{(k)}+h_{\eta,\delta}
    \left[-C\nabla f(\rvx_i^{(k)})+\frac{d+1}{N}(\rvx_i^{(k)}-\bar\rvx^{(k)})\right]
    +\sqrt{2h_{\eta,\delta}C_k}\xi_i^{(k)}.
\end{equation}
The square root of the covariance is allowed to be rectangular, such as $S$ given in \labelcref{eq:nonsymRoot}. Since the value of $D(\rmX)$ is unchanged under an affine transformation and the step size is shared by all particles, this retains full affine invariance.

The quantity $D(\rmX)$ is the mean squared norm of the preconditioned gradients $C\nabla f(\rvx_i)$ in the inverse-covariance metric. Thus $h_{\eta,\delta}\sqrt{D(\rmX)}\le\delta$ bounds the size of the gradient contribution to the update in this metric.

A similar computation to the previous section gives geometric ergodicity of this scheme. The proof is deferred to \Cref{ssec:proofTamedALDI} and utilizes the improved Lyapunov function in \Cref{appsec:dplus2}. 

\begin{theorem}\label{thm:tamedALDI}
    Assume $f$ satisfies \Cref{assmp:distantConvexityND}, and that $N \ge d+2$. Fix $\kappa \in (0,1)$. There exists a Lyapunov function $V_{a,\kappa}$ and constants $a,b,c>0$, compact $\mathrm{K} \subset M$, and thresholds $\delta_*, \eta_* >0$ such that: for each fixed $\delta \in (0, \delta_*]$, for all $0<\eta \le \min(\delta,\eta_*)$, the Markov kernel $Q_{\eta,\delta}$ of the clamped update \labelcref{eq:tamedALDI} satisfies
    \begin{equation} \label{eq:TamedLyapunovDecay}
        Q_{\eta,\delta} V_{a,\kappa} \le (1 - c\eta) V_{a,\kappa} + b \eta \mathbf{1}_{\mathrm{K}}.
    \end{equation}
    In particular, the chain has a unique  invariant probability distribution $\Pi_{\eta,\delta}$ and is geometrically ergodic. Furthermore, $\Pi_{\eta,\delta} \rightharpoonup \Pi_*$ weakly as $\eta \rightarrow 0$.
\end{theorem}

\section{Proofs} \label{sec:proofs}
This section proves the main results \Cref{thm:ctsGeoErgo,thm:timeScaleErgo,thm:discreteTimeScaleErgo,thm:tamedALDI}. Proofs of auxiliary results will be deferred to the appendix.

For an ensemble $\rmX$, we additionally define the following functions:
\begin{subequations}
\begin{gather}
    E \coloneqq \frac{1}{N} \sum_{i=1}^N f(\rvx_i),\quad T \coloneqq \Tr(C),\quad J \coloneqq \Tr(C^{-1}),\\
    \rvg_i \coloneqq \nabla f(\rvx_i),\quad \bar \rvg \coloneqq \frac{1}{N} \sum_{i=1}^N \rvg_i,\quad D \coloneqq \frac{1}{N}\sum_{i=1}^N \rvg_i^\top C \rvg_i \in \R,\\
    K \coloneqq \frac{1}{N} \sum_{i=1}^N (\rvx_i - \bar \rvx) \rvg_i^\top = \frac{1}{2N^2} \sum_{i,j=1}^N (\rvx_i - \rvx_j) (\rvg_i - \rvg_j)^\top \in \R^{d \times d}.
\end{gather}
\end{subequations}

Without loss of generality, we assume that $f$ is non-negative and that $\nabla f(0) = 0$. We first list some useful trace inequalities and standard results for distantly convex functions.
\begin{lemma}\label{lem:traceIneqs}
    The following trace inequalities hold.
    \begin{align*}
        |\Tr K| &\le LT,  &|\Tr (KC)| &\le LT^2, &|\Tr(KC^2)| &\le LT^3,\\
        |\Tr (KC^{-1})| &\le L\sqrt{TJ},&TJ&\ge d^2 , &J^{-1} |\Tr(KC^{-1})|&\le \frac{LT}{d}. 
    \end{align*}
\end{lemma}

\begin{proposition}[Consequences of distant convexity]\label{prop:DistantConvexity}
    Suppose that $f$ satisfies \Cref{assmp:distantConvexityND}. There exist constants $c_1, b_1, c_2, b_2>0$ depending only on $f$ and $d$ such that for all $x \in \R^d$,
    \begin{align*}
        f(x) &\ge c_1 \|x\|^2 - b_1,\\
        \|\nabla f(x) \|^2 &\ge c_2 f(x) - b_2.
    \end{align*}
    In particular, there exist constants $C_T, B_T>0$ such that 
    \begin{align}
        D \ge C_T T^2 - B_T + \bar \rvg^\top C \bar \rvg. \label{eq:DLowerbound}
    \end{align}
    Moreover, for any $J_0>0$, the following holds on the set $\{J \le J_0\}$:
    \begin{align}
        D \ge \frac{c_2}{J_0} E - \frac{b_2}{J_0}.\label{eq:DLowerbound2}
    \end{align}
\end{proposition}

The proofs are deferred to \Cref{appsec:traceIneqs,appsec:propDistConvex}. This will be useful since the compact sets in \Cref{thm:ctsGeoErgo,thm:timeScaleErgo,thm:discreteTimeScaleErgo,thm:tamedALDI} will all take the form $\mathrm{K} = \{T \le T_0, J \le J_0, E \le E_0\}$ for some $T_0, J_0, E_0>0$.

\begin{proposition}\label{prop:KCompact}
    Assuming $f$ satisfies \Cref{assmp:distantConvexityND}, for any $T_0, J_0, E_0 >0$, the set $\mathrm{K} = \{T \le T_0, J \le J_0, E \le E_0\}\subset \mathrm{M}$ is compact.
\end{proposition}
\begin{proof}
     The bound $E\le E_0$ and coercivity of $f$ make $\mathrm{K}$ bounded in $\R^{N\times d}$. Since $T$, $J$, and $E$ are continuous away from the singular boundary, $\mathrm{K}$ is also closed in $\R^{N\times d}$ and thus compact in $\R^{N\times d}$. As $\mathrm{K}$ is also disjoint from $\partial \mathrm{M}$, normality gives compactness in $\mathrm{M}$.
\end{proof}
\subsection{Proof of Theorem \ref{thm:ctsGeoErgo}} \label{proof:ctsGeoErgo}
We need to compute $\gL \gV$, and begin with $\gL E$.
\begin{lemma}\label{lem:generatorE}
    The generator applied to $E$ is 
    \begin{align}
        \gL E &= -D + P, \label{eq:LE}\\ 
        P &\coloneqq  \frac{d+1}{N} \Tr K + \frac{1}{N}\sum_i C\! :\! \nabla^2 f(\rvx_i). \label{eq:PDef}
    \end{align}
    Furthermore, if $\|\nabla^2 f\| \le L$, then $P$ can be bounded as 
    \begin{align}
        P \le p_0 T,\quad p_0 = L(1+\frac{d+1}{N}).\label{eq:PUpperBound}
    \end{align}
\end{lemma}
\begin{proof}
    Using the generator \labelcref{eq:ALDIGeneratorDef}, we can compute
    \begin{align*}
        \gL E &= \sum_i (-C \nabla f(\rvx_i) + \frac{d+1}{N} (\rvx_i - \bar \rvx)) \cdot \left(\frac{1}{N}\nabla f(\rvx_i)\right) + \frac{1}{N} \sum_i C \!:\! \nabla^2 f(\rvx_i)\\
        &= -\frac{1}{N} \sum_i \rvg_i^\top C \rvg_i + \frac{d+1}{N^2} \sum_i (\rvx_i - \bar \rvx) \cdot \rvg_i + \frac{1}{N} \sum_i C \!:\! \nabla^2 f(\rvx_i)\\
        &= -D + \frac{d+1}{N} \Tr(K) + \frac{1}{N} \sum_i C \!:\! \nabla^2 f(\rvx_i).
    \end{align*}
    The bound on $P$ follows from \Cref{lem:traceIneqs} and $\sum_i C \!:\! \nabla^2 f(\rvx_i) \le NLT$. 
    
\end{proof}

We next compute $\gL J$.

\begin{lemma}\label{lem:GeneratorJ}
    The generator applied to $J = \Tr(C^{-1})$ is 
    \begin{equation}\label{eq:LJ}
        \gL J = 2 \Tr(KC^{-1}) + \left[-2 + \frac{2d+4}{N}\right] \Tr(C^{-1}).
    \end{equation}
\end{lemma}
The calculation is more involved and deferred to \Cref{appsec:GeneratorJ}. We can now compute the generator applied to the Lyapunov function, $\gL W_a$. By chain rule, the generator applied to the exponential of a test function $V$ is 
    \begin{align*}
        \gL e^V = (\gL V)e^V + \Gamma(V,V) e^V.
    \end{align*}
    Let $\boldsymbol{\delta}_i = \rvx_i - \bar \rvx$. The generator applied to the Lyapunov function $W_a = e^{aE} J$ is
    \begin{align*}
        \gL W_a &= J \gL e^{aE} + e^{aE} \gL J + 2\Gamma(e^{aE}, J)\\
        &= J \left[a \gL E + a^2\frac{1}{N^2} \sum_i  \rvg_i^\top C \rvg_i\right]e^{aE} + e^{aE} \gL J + 2a \sum_i (\nabla_i J)^\top C \frac{1}{N} \rvg_i e^{aE}\\
        &= a W_a\gL E  + \frac{a^2}{N} D W_a + e^{aE} \gL J - \frac{4a}{N^2} \sum_i \boldsymbol{\delta}_i^\top C^{-1} \rvg_i e^{aE}\\
        &= a W_a \gL E  + \frac{a^2}{N} D W_a + e^{aE} \gL J - \frac{4a}{N} \Tr(KC^{-1}) e^{aE}. 
    \end{align*}
    Dividing by $W_a$ and substituting \labelcref{eq:LE,eq:LJ} yields
    \begin{align}
        \frac{\gL W_a}{W_a} = \left(-a + \frac{a^2}{N}\right)D+aP + \left(-2 + \frac{2d+4}{N}\right) + \left(2 - \frac{4a}{N}\right) \frac{\Tr(KC^{-1})}{J}. \label{eq:LWaoverWa}
    \end{align}

It now remains to bound these terms using the distant convexity assumption. For $N \ge d+3, a \in (0,N/2)$, define the following constants
    \begin{align}
        \lambda \coloneqq 2 - \frac{2d+4}{N} > 0,\quad c_a \coloneqq a\left(1 - \frac{a}{N}\right)>0. \label{eq:ConstantDefsLambdaca}
    \end{align}
    Equation \labelcref{eq:LWaoverWa} becomes
    \begin{align}
        \frac{\gL W_a}{W_a} &\le -\lambda - c_a D + aP + \left(2 - \frac{4a}{N}\right) \frac{\Tr(KC^{-1})}{J} .\label{eq:WaUnsimplifiedBound}
    \end{align}
    Using \labelcref{eq:PUpperBound}, \labelcref{eq:DLowerbound} and \Cref{lem:traceIneqs}, we obtain
    \begin{equation}
        \frac{\gL W_a}{W_a}\le -\lambda - c_a C_T T^2 + c_a B_T + a p_0 T + \left(2 - \frac{4a}{N}\right) L \sqrt{\frac{T}{J}}
    \end{equation}
    The middle three terms constitute a quadratic in $T$, with coefficients all of order $a$. Therefore, we can choose $a>0$ sufficiently small such that 
    \begin{align*}
        \sup_{T \in \R} \left\{-\frac{1}{2}c_a C_T T^2 + ap_0 T + c_a B_T\right\} = c_a B_T + \frac{a^2 p_0^2}{2c_a C_T} \le \frac{\lambda}{4}.
    \end{align*}
    Then for this choice of $a$, we have the bound
    \begin{equation*}
        \frac{\gL W_a}{W_a} \le -\frac{3\lambda}{4} - \frac{c_a C_T}{2} T^2+ (2 - \frac{4a}{N})L\sqrt{\frac{T}{J}}.
    \end{equation*}
    \begin{enumerate}
        \item Since $TJ \ge d^2$, we have $\sqrt{T/J} \le T/d$. Therefore, there exists a $T_0>0$ such that if $T \ge T_0$, then 
        \begin{equation*}
            \frac{\gL W_a}{W_a} \le -\frac{3\lambda}{4} - \frac{c_a C_T}{2}T^2 + (2 - \frac{4a}{N}) L \frac{T}{d} \le -\frac{\lambda}{2}.
        \end{equation*}
        \item For $T \le T_0$, choose a sufficiently large constant $J_0>0$ such that for all $J \ge J_0$,
        \begin{align*}
            (2 - \frac{4a}{N}) L \sqrt{\frac{T}{J}} \le \frac{\lambda}{4}
        \end{align*}
        which again implies that 
        \begin{equation*}
            \frac{\gL W_a}{W_a} \le -\frac{\lambda}{2}.
        \end{equation*}
        \item When $T \le T_0$ and $J \le J_0$, using \labelcref{eq:DLowerbound2}, \labelcref{eq:PUpperBound} and \Cref{lem:traceIneqs}, we obtain
        \begin{equation}
            \frac{\gL W_a}{W_a} \le -\lambda - c_a \left[\frac{c_2}{J_0} E - \frac{b_2}{J_0}\right] + ap_0 T + (2 - \frac{4a}{N})\frac{L}{d}T.
        \end{equation}
        Since $c_a, c_2, J_0>0$, we can choose $E_0>0$ large enough such that for $E \ge E_0$, that $\gL W_a/W_a \le -\lambda/2$. 
        
        \item The remaining set $\mathrm{K} \coloneqq \{T \le T_0, J \le J_0, E \le E_0\}$ is compact by \Cref{prop:KCompact}. 
    \end{enumerate}

    Since $W_a$ is continuous on $\mathrm{M}$, we have shown that 
    \begin{equation*}
        \gL W_a \le -\frac{\lambda}{2} W_a + b \mathbf{1}_\mathrm{K}
    \end{equation*}
    for some compact $\mathrm{K} \subset \mathrm{M}$.

On $M$, the diffusion is locally uniformly positive, and furthermore $M$ is path-connected for $N \ge d+2$ \cite[Lem. A.2]{garbuno2020affine}. Hence the (strong-Feller) diffusion is positive recurrent and Lebesgue-irreducible \cite[Thm. 3.1]{kliemann1987recurrence}. Since the Lyapunov function $W_a$ is norm-like and satisfies the drift condition $\gL\gV \le -c \gV + b$ for some $c>0, b\in \R$, geometric ergodicity follows from \cite[Thm. 6.1]{meyn1993stability}.

\subsection{Proof of Theorem \ref{thm:timeScaleErgo}} \label{ssec:CtsTimeScaled}

With the proposed regularization, we have the simple bounds
\begin{equation}\label{eq:gammaBounds}
    \gamma \le 1,\quad \gamma T \le \theta^{-1}, \quad |\gamma'|T \le \gamma, \quad|\gamma'| T^2 \le \theta^{-1}.
\end{equation}
We first compute the carr\'e-du-champ between $T$ and $W_a$.
\begin{align*}
    \Gamma(T,W_a) &= \frac{2}{N}\sum_i (\rvx_i - \bar \rvx)^\top C \nabla_i (W_a)\\
    &= \frac{2}{N} \sum_i (\rvx_i - \bar \rvx)^\top C \left[\frac{a}{N}\rvg_i W_a + e^{aE} \nabla_i J\right]\\
    &= \frac{2a}{N^2} W_a\sum_i (\rvx_i - \bar \rvx)^\top C \rvg_i  + \frac{2}{N} \sum_i (\rvx_i - \bar \rvx)^\top C \left(-\frac{2}{N} C^{-2} (\rvx_i - \bar \rvx)\right) e^{aE}\\
    &= \frac{2a}{N} W_a \Tr(CK) - \frac{4d}{N} e^{aE}.
\end{align*}
Using \Cref{lem:generatorALDITimescale} and \labelcref{eq:LWaoverWa} gives 
\begin{align}
    \frac{\bar \gL W_a}{W_a} &= \gamma \left[\left(-a + \frac{a^2}{N}\right)D+aP + \left(-2 + \frac{2d+4}{N}\right) + \left(2 - \frac{4a}{N}\right) \frac{\Tr(KC^{-1})}{J}\right] \notag \\&\qquad + \gamma' \left[\frac{2a}{N} \Tr(CK) - \frac{4d}{NJ}\right] \label{eq:LBarWa}\\
    &\le \gamma \left \{-\lambda - c_a C_T T^2 + c_a B_T + a p_0 T + \left(2 - \frac{4a}{N}\right) L \sqrt{\frac{T}{J}} \right\}\notag\\
    &\qquad + |\gamma'| \left[\frac{2aL}{N} T^2 + \frac{4d}{NJ}\right]. \label{eq:LBarWaBound}
\end{align}
The additional terms can be absorbed into the quadratic in $T$ using $|\gamma'| T \le \gamma$ and $J^{-1} \le Td^{-2}$:
\begin{align*}
    \frac{\bar \gL W_a}{W_a} &\le -\gamma \lambda - c_a C_T \gamma T^2 + \left(a p_0 + \frac{2aL}{N}\right) \gamma T + c_a \gamma B_T + \left(2 - \frac{4a}{N}\right) L \gamma \sqrt{\frac{T}{J}} + \frac{4d}{NJ}|\gamma'|.
\end{align*}
Choose $a>0$ sufficiently small (independent of $\eta$) such that
\begin{align*}
    \sup_{T \ge 0}\left\{ -\frac{1}{2}c_a C_T T^2 + a\left(p_0 + \frac{2L}{N}\right) T + c_a B_T \right\}\le \lambda/4.
\end{align*}
Then we have the estimate
\begin{equation*}
    \frac{\bar \gL W_a}{W_a} \le -\frac{3 \lambda}{4}\gamma - \frac{1}{2} c_a C_T \gamma T^2 + \left(2 - \frac{4a}{N}\right) L\gamma \sqrt{\frac{T}{J}} + \frac{4d}{NJ}|\gamma'|.
\end{equation*}
\begin{enumerate}
    \item Since $\sqrt{T/J} \le T/d$, the second-to-last term is bounded as $T \rightarrow \infty$. Since $J^{-1} \le Td^{-2}$ and $|\gamma'|T \le \gamma \le 1$, the last term is also bounded. Choose $T_0$ sufficiently large such that for all $T \ge T_0$, $(\bar \gL W_a) / W_a \le -1$. Such $T_0$ exists since $\gamma T^2 \sim T$ and $c_a C_T >0$. 
    
    \item On $T \le T_0$, $\gamma$ is lower bounded by some $\gamma_0 = \gamma(T_0)$, and $|\gamma'|$ is upper bounded by $\theta$. Choose $J_0$ sufficiently large such that 
\begin{equation*}
    \left(2 - \frac{4a}{N}\right)L \sqrt{\frac{T_0}{J_0}} + \frac{4d}{NJ_0} \theta \le \frac{\lambda}{4} \gamma_0.
\end{equation*}
Then on $\{T \le T_0, J \ge J_0\}$, using $\gamma \le 1$,  we have the inequalities
\begin{align*}
    \frac{\bar \gL W_a}{W_a} \le -\frac{3\lambda }{4} \gamma_0 + \left(2 - \frac{4a}{N}\right) L\sqrt{\frac{T_0}{J_0}} +\frac{4d}{NJ_0}\theta\le -\frac{\lambda}{2} \gamma_0.
\end{align*}

\item The remaining region is $\{T \le T_0, J \le J_0\}$. We still have that $\gamma \ge \gamma_0$ on this region. Recalling the bound from \labelcref{eq:DLowerbound2}
\begin{equation*}
    D \ge \frac{c_2}{J_0} E - \frac{b_2}{J_0}\qquad \text{in }  \{J \le J_0\},
\end{equation*}
combining into \labelcref{eq:LBarWa} yields the upper bound
\begin{equation*}
    \frac{\bar \gL W_a}{W_a} \le -\gamma_0 c_a\frac{c_2}{J_0}E + c
\end{equation*}
for some constant $c$, using $J^{-1} \le Td^{-2}$ to bound negative powers of $J$ with the upper bound on $T$. Finally choose $E_0$ sufficiently large such that whenever $E \ge E_0,T \le T_0, J \le J_0$, we have $\bar \gL W_a / W_a \le -1$.

\item The remaining region $\mathrm{K} = \{T \le T_0, J \le J_0, E \le E_0\}$ is compact as before. 
\end{enumerate}
Since $W_a$ is continuous, both $\bar \gL W_a/W_a$ and $W_a$ are upper bounded on $\mathrm{K}$, therefore the Foster--Lyapunov condition holds:
\begin{equation}
    \bar \gL W_a \le -\min(1, \lambda \gamma_0/2) W_a + b \mathbf{1}_\mathrm{K}.
\end{equation}
The concluding argument from the previous theorem yields geometric ergodicity.

\subsection{Proof of Theorem \ref{thm:discreteTimeScaleErgo}}\label{ssec:discreteTimeErgo}

Suppose $\rmX \in \mathrm{M}$. The discrete time update \labelcref{eq:discreteTimeScaleDef} can be written in matrix form as
\begin{equation*}
    \rmX_+ = \rmX + \eta \left[- \gamma \begin{pmatrix}
        \rvg_1^\top \\ ... \\ \rvg_N^\top
    \end{pmatrix}C + P_c \rmX B\right] + \sqrt{2 \eta \gamma} \Xi C^{1/2}
\end{equation*}
where $C^{1/2}$ is the symmetric square root, and $P_c = I_N - N^{-1} \mathbf{1}_N\mathbf{1}_N^\top $ is the projection onto $\langle \mathbf{1}_N\rangle^\perp$, and $\Xi$ is an $N\times d$ matrix with i.i.d. standard Gaussian entries. Now consider the centered variables
\begin{equation*}
    \rmY = P_c \rmX \in \R^{N \times d},\quad \rmZ = \rmY C^{-1/2}.
\end{equation*}
This normalization removes the scale and conditioning from the current covariance. Since $\rmY^\top \rmY = NC$, we have $\rmZ^\top \rmZ = NI_d$. In particular, $\rmZ$ will lie in a compact set, which will allow for a uniform control later. The centered update scaled by the old covariance satisfies
\begin{align}\label{eq:ZPlus}
     \rmY_+ C^{-1/2} &= \rmZ + \eta \Delta_\rmX + \sqrt{2 \eta \gamma} P_c \Xi,\\
     \Delta_\rmX &\coloneqq -\gamma \tG C^{1/2} + P_c \rmX B C^{-1/2}
\end{align}
where $\Delta_\rmX$ is the deterministic component of the update, and $\tG \coloneqq P_c (\rvg_1\ ...\ \rvg_N)^\top \in \R^{N \times d}$ is the centered gradient matrix. Since $B$ is a linear combination of powers of $C$, it commutes with $C^{-1/2}$, and we have the equivalent expression
\begin{equation}\label{eq:DeltaX}
    \Delta_\rmX = -\gamma \tG C^{1/2} + \rmZ B.
\end{equation}
We first show that the deterministic part of the scaled update in $Y$ is uniformly bounded.
\begin{lemma}
    The deterministic update $\Delta_\rmX$ is uniformly bounded over all $\rmX \in \mathrm{M}$, by some constant depending only on $N,d, \theta, L$. Furthermore,
    \begin{equation}\label{eqs:timescaleIneqs}
        \|\gamma C\| \le \theta^{-1},\quad \sum_j \|-\gamma C \rvg_j + B \rvy_j\|^2 \lesssim_{\theta, L,N,d} 1 + N \gamma D.
    \end{equation}
\end{lemma}
\begin{proof}
    First note that
    \begin{align*}
        \Tr(\tG^\top \tG) = \sum_i \|\rvg_i - \bar \rvg\|^2 \le \sum_i \|\rvg_i - \nabla f (\bar \rvx)\|^2 \le NL^2 \Tr(C).
    \end{align*}
    Therefore, the first component can be bounded as
    \begin{align}
        \|\tG C^{1/2}\|_F^2 = \Tr(C \tG^\top \tG) \le \Tr(C) \Tr(\tG^\top \tG) \le NL^2T^2. \label{eq:GCnHalfBound}
    \end{align}
    The second component $\rmZ B$ can be bounded as follows:
    \begin{align*}
        \|\rmZ B\| &\le \|\rmZ\| \|B\|\\
        &\le \sqrt{N} \left[\frac{d+1}{N}\gamma + \frac{2}{N} |\gamma'| \Tr(C)\right]\\
        &\le \frac{d+3}{\sqrt{N}}\gamma
    \end{align*}
    where we use $|\gamma'|T \le \gamma \le 1$. Combining yields
    \begin{equation*}
        \|\Delta_\rmX\| \le \gamma \sqrt{N} LT + \frac{d+3}{\sqrt{N}}\gamma \le \sqrt{N}L\theta^{-1} + \frac{d+3}{\sqrt{N}}.
    \end{equation*}

    The first inequality in \labelcref{eqs:timescaleIneqs} follows directly from $\|C\| \le T$. For the second, we have that
    \begin{align*}
        \sum_j \|\gamma C \rvg_j\|^2 &\le \gamma^2 \sum_j \rvg_j^\top C^\top C \rvg_j \le \gamma^2 \|C\| \sum_j \rvg_j^\top C \rvg_j\\
        &\le N\gamma^2 TD \le N \theta^{-1} \gamma D,
    \end{align*}
    and also
    \begin{align*}
        \sum_j \|B \rvy_j\|^2 = N \Tr(B C B^\top)\le N \|B\|^2 T \le N \left(\frac{d+3}{N}\gamma\right)^2 T \le \frac{(d+3)^2}{N} \theta^{-1} \gamma.
    \end{align*}
    The desired inequality follows from the inequality $\|u+v\|^2\le2(\|u\|^2+\|v\|^2)$.
\end{proof}

The previous lemma shows that the squared drift of the process grows at most as $1+\gamma D$, i.e. with norm growing at most linearly in $\|\rmX\|_F$. This yields a desired growth condition.

The following lemma, proved in \Cref{app:uniformBoundUpdates}, justifies a uniform Taylor expansion of the singular map $H\mapsto H^{-1}$ after taking expectations.

\begin{lemma}\label{lem:uniformBoundUpdates}
    Let $n = N-1 \ge d+2$, and consider the compact subset
    \begin{equation*}
        \gZ = \left\{Z \in \R^{N \times d} \mid \mathbf{1}_N^\top Z = 0,\, Z^\top Z = N I_d\right\}.
    \end{equation*}
    Let $M < \infty$, and define $\gU \coloneqq \{U \in \R^{N \times d} \mid P_c U = U ,\, \|U\|_F \le M\} $ as the set of bounded centered matrices. Suppose $U \in \gU$ and $\gamma \in [0,1]$. Further suppose that $R$ is a symmetric matrix satisfying $\gamma I_d \preceq R \preceq 2\gamma I_d$, and let $\Xi \in \R^{N \times d}$ have i.i.d.\ standard Gaussian entries. Define the random matrices
    \begin{equation}
        \Lambda_\eta = Z + \eta U + \sqrt{2\eta} P_c \Xi R^{1/2},\quad H_\eta \coloneqq \frac{1}{N} \Lambda_\eta^\top \Lambda_\eta .
    \end{equation}
    Then the following first order expansion holds, where the remainder is uniform over $Z \in \gZ$, $U \in \gU$, $\gamma \in [0,1]$, and $R$ satisfying $\gamma I_d\preceq R \preceq 2\gamma I_d$ as $\eta \rightarrow 0$,
    \begin{gather}
        \mathbb{E} H_\eta^{-1} = I_d + \eta B_{-1} + E_\eta,\quad \|E_\eta\| = \mathcal{O}(\eta^{3/2}),\\
        B_{-1} \coloneqq -\frac{1}{N}(Z^\top U + U^\top Z) + \frac{2}{N}\left[(d+3-N) R + (\Tr R) I_d\right].
    \end{gather}
\end{lemma}

We can now begin to prove the Foster--Lyapunov condition. Let $Q_\eta$ denote the Markov kernel corresponding to the discrete time update. We first show that it satisfies an upper bound of the form
\begin{equation*}
    \frac{Q_\eta W_a}{W_a} \le \exp(\mathcal{O}(\eta)) (1 + \mathcal{O}(\eta) + \mathcal{O}(\eta^{3/2}))
\end{equation*}
where the combined first order terms are negative outside a compact set for sufficiently small $\eta$. We will find that the remainders will be bounded by a $\mathcal{O}(\eta^{3/2})$ and a $\mathcal{O}(\eta^2 (1+\gamma D))$ term, which can be absorbed by the negative first order drift.

\subsubsection{Notation}
Let $F(\rmX) = \sum_j f(\rvx_j) = NE$, so that the update can be written
\begin{equation}\label{eq:XPlus}
    \rmX_+ = \rmX + \eta \left[- \gamma  \nabla F C + P_c \rmX B\right] +  \Xi\sqrt{2 \gamma \eta C}.
\end{equation}
Decompose the increment $\delta \rmX \coloneqq \rmX_+ - \rmX$ as the sum of a deterministic and a random component
\begin{align*}
    \delta \rmX &= \eta \rmH + \Xi\sqrt{2 \gamma \eta C} ,\\
    \rmH &\coloneqq -\gamma \nabla F C + P_c \rmX B.
\end{align*}
Recall from \labelcref{eqs:timescaleIneqs} that $\|\rmH\|_F^2 \lesssim_{\theta, L,N,d} 1+\gamma D$. Moreover, the relationship with the deterministic component $\Delta_\rmX$ is 
\begin{equation}
    \Delta_\rmX = P_c \rmH C^{-1/2}.
\end{equation}

\subsubsection{Constructing the inequality}\label{sssec:ineqConstruction}
By definition, the Markov kernel applied to $W_a$ is 
\begin{align*}
    \frac{Q_\eta W_a}{W_a}(\rmX) = \mathbb{E} \left[e^{a(E_+ - E)} \frac{J_+}{J} \right].
\end{align*}
Since $\nabla^2 f \preceq L I_d$, we have the following quadratic upper bound on $F$:
\begin{equation*}
    F(\rmX + \delta \rmX) \le F(\rmX) + \nabla F(\rmX) \!:\!\delta \rmX + \frac{L}{2} \|\delta \rmX\|_F^2.
\end{equation*}
We will bound the expectation by constructing another probability distribution where $\Xi$ is no longer a standard Gaussian matrix. Define a tilting term given by the upper bound, and its normalizing constant
\begin{align}
    \Psi \coloneqq \exp( \frac{a}{N} \left[\nabla F(\rmX) \!:\! \delta \rmX + \frac{L}{2} \|\delta \rmX\|_F^2\right]),\quad 
    \hat Z \coloneqq \mathbb{E} \Psi.
\end{align}
We verify below that $\hat Z$ is finite. Let $\mathbb{P}$ denote the distribution of the update $\rmX_+$ or equivalently of $\Xi$. We define a new probability density for $\rmX_+$ (also on $\Xi$) by
\begin{align}
    \frac{\dd{\hat{\mathbb{P}}}}{\dd{\mathbb{P}}} = \frac{\Psi}{ \hat{Z}}.
\end{align}
The change of measure gives
\begin{align}
    \frac{Q_\eta W_a}{W_a} \le \hat Z \hat{\mathbb{E}} \left[\frac{J_+}{J} \right].
\end{align}
\subsubsection{Tilting: normalizing constant}\label{sssec:normalizing}
To compute $\hat Z$, we use the following identity for the expectation of quadratics of Gaussians.
\begin{lemma}\label{lem:gaussianClosedForms}
    For a standard Gaussian vector $\xi$, vector $\ell$ and symmetric matrix $S$ satisfying $I-2S \succ 0$, 
    \begin{equation*}
        \mathbb{E} \left[e^{\ell \cdot \xi + \xi^\top S \xi}\right] = \det(I - 2S)^{-1/2} \exp(\frac{1}{2} \ell^\top (I - 2S)^{-1} \ell).
    \end{equation*}
\end{lemma}
Directly applying this with the definition of $\hat Z = \mathbb{E}\Psi$, where $\Xi$ is a matrix with standard Gaussian entries,
\begin{align*}
    \hat Z &= \mathbb{E}\exp\left[\frac{a}{N} \left(\nabla F \!:\! (\eta \rmH + \Xi\sqrt{2 \eta \gamma C} ) + \frac{L}{2} (\eta \rmH +\Xi \sqrt{2 \eta \gamma C} ) \!:\!(\eta \rmH + \Xi\sqrt{2 \eta \gamma C} )\right)\right]\\
    &= \mathbb{E} \exp\left[\frac{a}{N} \left(\eta\nabla F \!:\! \rmH+\frac{L}{2} \eta^2 \rmH\!:\!\rmH + (\nabla F + L \eta \rmH) \!:\! (\Xi\sqrt{2 \eta \gamma C} ) + \frac{L}{2} \Tr( 2\eta \gamma \Xi C \Xi^\top)\right)\right].
\end{align*}
Applying \Cref{lem:gaussianClosedForms} with $\ell = \frac{a}{N} (\nabla F + L \eta \rmH) \sqrt{2 \eta \gamma C}$ and $S = \frac{a}{N} L \eta (I_N \otimes \gamma C)$, we obtain
\begin{align}
    \log \hat{Z} &= -\frac{1}{2} \log\det\left(I - 2S\right) + \frac{1}{2} \ell^\top (I-2S)^{-1} \ell + \frac{a}{N} \eta \nabla F \!:\! \rmH + \frac{aL}{2N} \eta^2 \|\rmH\|_F^2. \label{eq:logZHat}
\end{align}

Since $\|\gamma C\| \le \theta^{-1}$, we immediately obtain that $I_{Nd} - 2S \succ 0$ if $\frac{a}{N} L\eta \theta^{-1} < \frac{1}{2}$. The normalizing constant $\hat Z$ is therefore finite for all sufficiently small $\eta$, depending on $a>0$. 
\begin{condition}\label{cond:boundedS}
    We assume that $\eta$ is sufficiently small so that $\frac{a}{N} L\eta \theta^{-1} < \frac{1}{4}$, thus $\frac{1}{2} I\preceq(I-2S)\preceq I$ and all Taylor expansions have uniformly bounded remainder.
\end{condition}
The following arguments will assume that \Cref{cond:boundedS} holds. It remains to expand $\log \hat Z$ in terms of powers of $\eta$. We identify the first order term, and show that the second order term is uniformly bounded by $\mathcal{O}(\eta^2 (1+\gamma D))$.

\textbf{Log-determinant term.} The Taylor expansion of $\log \det$ about $I$ is 
\begin{align*}
\log\det(I-A) = - \Tr A + \mathcal{O}(\|A\|^2).
\end{align*}
Since $S = \mathcal{O}(\eta)$ uniformly, applying the expression with $A=2S$ gives
\begin{align}
    -\frac{1}{2} \log\det (I-2S) &= \Tr S + \mathcal{O}(\eta^2) = aL \eta\gamma T + \mathcal{O}(\eta^2). \label{eq:logDet1}
\end{align}
\textbf{Quadratic term.} Since $S = \mathcal{O}(\eta)$, we obtain $(I-2S)^{-1} = I+\mathcal{O}(\eta)$. Therefore
\begin{align*}
    \frac{1}{2} \ell^\top (I-2S)^{-1} \ell &= \frac{a^2}{2N^2} (\nabla F + L\eta \rmH) (2 \eta \gamma C)(I - 2\frac{aL}{N} \eta \gamma C)^{-1} \!:\!(\nabla F + L \eta \rmH)\\
    &= \frac{a^2}{N^2} \eta (\nabla F + L \eta \rmH) (\gamma C + \mathcal{O}(\eta) \gamma C)\!:\! (\nabla F + L \eta \rmH).
\end{align*}
By definition of $D$, we have that $\nabla F C\!:\!\nabla F  = \sum_j \rvg_j^\top C \rvg_j = ND$. The higher order cross terms are all bounded similarly using \labelcref{eqs:timescaleIneqs}:
\begin{gather*}
    |\rmH(\gamma C)\!:\! \nabla F|\le \|\rmH\|_F \|\gamma  \nabla F C\|_F \lesssim \sqrt{1+\gamma D} \sqrt{\gamma D},\\
    |\rmH (\gamma C)\!:\! \rmH| \le \|\rmH\|_F^2 \|\gamma C\|_F \lesssim 1+\gamma D.
\end{gather*}
Combining yields
\begin{equation*}
    \frac{1}{2}\ell^\top (I-2S)^{-1} \ell = \frac{a^2}{N}\eta \gamma D + \mathcal{O}(\eta^2(1+\gamma D)).
\end{equation*}

\textbf{Drift terms.} A direct computation gives
\begin{align*}
    \frac{a}{N} \nabla F \!:\!\rmH &= -a \gamma D + \frac{a}{N} \sum_j \rvg_j^\top \left(\frac{d+1}{N} \gamma I_d + \frac{2 \gamma' C}{N}\right)(\rvx_j - \bar \rvx)\\
    &=-a\gamma D + \frac{a}{N}(d+1) \gamma \Tr( K) + \frac{2a \gamma'}{N} \Tr(KC).
\end{align*}
The last term satisfies $\frac{aL}{2N}\eta^2 \|\rmH\|_F^2 \lesssim \eta^2(1+\gamma D)$ by \labelcref{eqs:timescaleIneqs}.

Substituting these estimates into \labelcref{eq:logZHat} yields the first order expansion with uniform remainder (for sufficiently small $\eta$ depending on $a>0$)
\begin{align}
    \log \hat Z &= \left(-a + \frac{a^2}{N}\right)\eta \gamma D +\frac{a(d+1)}{N} \eta\gamma \Tr(K) + \frac{2a \gamma'}{N} \eta\Tr(KC) + aL\eta \gamma T + \mathcal{O}(\eta^2(1+\gamma D)) \notag \\
    &\le-c_a \eta \gamma D + a \left(1+\frac{d+1}{N}\right) \eta \gamma LT + \frac{2a|\gamma'|}{N}\eta LT^2 + \mathcal{O}(\eta^2(1+\gamma D)). \label{eq:logZBound}
\end{align}
Here, as before, $c_a = a - \frac{a^2}{N}$ is positive for sufficiently small $a>0$, and we use the inequalities $|\Tr(K)| \le LT,\, |\Tr(KC)| \le LT^2$.

\subsubsection{Tilting: expectation}\label{sssec:expectation}
We now expand $J_+ = \Tr(C_+^{-1})$ using \Cref{lem:uniformBoundUpdates}. From \labelcref{eq:ZPlus}, we have that $J_+ = \Tr(C_+^{-1}) = \Tr(C^{-1/2} H_\eta^{-1} C^{-1/2})$. It remains to find the expectation of this under $\hat{\mathbb{P}}$, by identifying the law of $\Xi$ under $\hat{\mathbb{P}}$.

We first identify the effective change of variables for $\Xi$: instead of having standard i.i.d.\ Gaussian entries, the covariance of $\Xi$ under $\hat{\mathbb{P}}$ is derived as follows. We apply the change of variables for the random matrix $\delta \rmX$: 
\begin{gather*}
    \frac{\dd{\mathbb{P}}}{\dd{\mathrm{Leb}}}(\delta \rmX) \sim \gN(\eta \rmH, I_N \otimes 2 \eta \gamma C),\\
    \frac{\dd{\hat{\mathbb{P}}}}{\dd{\mathbb{P}}} (\delta \rmX) \propto \exp(\frac{1}{2} \left(\delta \rmX + \frac{1}{L}\nabla F\right)^\top \left(\frac{a}{N} L I_{Nd}\right) \left(\delta \rmX + \frac{1}{L}\nabla F\right)).
\end{gather*}
Applying the rule for product of Gaussian densities \cite{petersen2008matrix}, we obtain that 
\begin{align*}
    \frac{\dd{\hat{\mathbb{P}}}}{\dd{\mathrm{Leb}}}(\delta \rmX) &\sim \gN(\rmM_c, \Sigma_c),\\
    \Sigma_c &= \left((I_N \otimes 2 \eta \gamma C)^{-1} - \frac{a}{N} L I_{Nd}\right)^{-1}\\
    &= I_N \otimes \left(I_d - 2\frac{a}{N}\eta \gamma LC\right)^{-1} (2 \eta \gamma C),\\
    \rmM_c &= \Sigma_c \left(\frac{aL}{N} \frac{1}{L}\nabla F + (I_N \otimes 2\eta \gamma C)^{-1} \eta \rmH\right)\\
    &= \left(2\frac{a}{N} \eta \gamma \nabla F C + \eta \rmH\right) \left(I - 2\frac{a}{N}\eta \gamma LC\right)^{-1}.
\end{align*}
Applying this to the update \labelcref{eq:ZPlus}, we compute
\begin{align*}
    P_c \rmM_c C^{-1/2} &= \left(2\frac{a}{N} \eta \gamma \tG C^{1/2} + \eta P_c \rmH C^{-1/2}\right) \left(I_d - 2\frac{a}{N}\eta \gamma LC\right)^{-1}\\
    &= \eta \Delta_\rmX + \left(2\frac{a}{N}\eta \gamma \tG C^{1/2} + 2\frac{aL}{N} \eta^2 \gamma \Delta_\rmX C\right)\left(I_d - 2\frac{a}{N}\eta \gamma LC\right)^{-1}\\
    &= \eta \Delta_\rmX + 2\frac{a \gamma}{N} \eta (\tG + L \eta P_c \rmH)\left(I_d - 2\frac{a}{N} \eta \gamma LC\right)^{-1}C^{1/2}.
\end{align*}
Therefore, the centered update is distributed as
\begin{align*}
    \rmY_+C^{-1/2} &= \rmZ + P_c \delta \rmX C^{-1/2}\\
    &= \rmZ + \eta U_\eta + \sqrt{2\eta} P_c \Xi' R_\eta^{1/2},
\end{align*}
where $\Xi'$ is an $N \times d$ matrix with i.i.d.\ standard Gaussian entries, and $U_\eta, R_\eta$ are defined as 
\begin{align}
    U_\eta &\coloneqq \Delta_\rmX + 2\frac{a\gamma}{N} (L \eta P_c \rmH + \tG) \left(I_d - 2\frac{a}{N} \eta \gamma LC\right)^{-1} C^{1/2}, \label{eq:UEta}\\
    R_\eta &\coloneqq \gamma \left(I_d - 2\frac{a}{N} \eta \gamma LC\right)^{-1}.
\end{align}

To apply \Cref{lem:uniformBoundUpdates}, we show that $\gamma I_d\preceq R_\eta \preceq 2 \gamma I_d$ and $U_\eta$ is uniformly bounded over all $\rmX$, for sufficiently small $\eta$. The inequality conditions for $R_\eta$ are directly satisfied under \Cref{cond:boundedS}.

To show that $U_\eta$ is uniformly bounded, we consider the zeroth and first order components in \labelcref{eq:UEta}. The zeroth order component is
\begin{equation*}
    \Delta_\rmX + \frac{2a\gamma}{N} \tG C^{1/2}.
\end{equation*}
The first order terms are
\begin{align*}
    &\quad 2\frac{aL}{N} \eta\gamma P_c \rmH C^{1/2} + (\gamma \eta P_c \rmH + \gamma \tG) \mathcal{O}(\eta\gamma C)C^{1/2} \\
    &= \mathcal{O}(\eta \Delta_\rmX \gamma C) + \mathcal{O}(\eta \Delta_\rmX \gamma C + \gamma \tG C^{1/2})\mathcal{O}(\eta \gamma C).
\end{align*}
From \labelcref{eq:DeltaX,eq:GCnHalfBound,eq:gammaBounds}, as well as $\|\rmZ\|_F^2 = Nd$, we have the uniform bound
\begin{align*}
    \Delta_\rmX &= -\gamma \tG C^{1/2} + \rmZ B\\
    &= \mathcal{O}(\gamma T) + \mathcal{O}(\gamma + \gamma'C) = \mathcal{O}(1).
\end{align*}
A similar argument shows that the zeroth order component is $\mathcal{O}(1)$ and the first order component is $\mathcal{O}(\eta)$ uniformly over all $\rmX$. Therefore, $U_\eta$ is uniformly bounded for sufficiently small $\eta \le 1$ satisfying \Cref{cond:boundedS}. 

The assumptions of \Cref{lem:uniformBoundUpdates} are now satisfied. This yields
\begin{equation}
    \hat{\mathbb{E}} H_\eta^{-1} = I + \eta B_{-1} + \mathcal{O}(\eta^{3/2}),
\end{equation}
where $B_{-1}$ is defined as
\begin{align*}
    B_{-1} = -\frac{1}{N} \left[\rmZ^\top U_\eta + U^\top_\eta \rmZ\right] + \frac{2}{N} \left[(d+3-N)R_\eta + (\Tr R_\eta)I_d\right].
\end{align*}
Using the first order expansion
\begin{align*}
    U_\eta &= -(1-\frac{2a}{N}) \gamma \tG C^{1/2}+ \rmZ B + \mathcal{O}(\eta),
\end{align*}
we obtain
\begin{align*}
    B_{-1} &= -\frac{1}{N} \left[\rmZ^\top \left(-\left(1 - \frac{2a}{N}\right) \gamma \tG C^{1/2} + \rmZ B\right) +\left(-\left(1 - \frac{2a}{N}\right) \gamma \tG C^{1/2} + \rmZ B\right)^\top \rmZ \right]\\
    &\qquad + \frac{2}{N} \left[(d+3-N) \gamma I_d + d \gamma I_d\right] + \mathcal{O}(\eta) \\
    &= \frac{1}{N} \left(1-\frac{2a}{N}\right) \gamma (\rmZ^\top \tG C^{1/2} + C^{1/2} \tG^\top \rmZ) - 2B + \frac{2}{N} \gamma(2d+3-N) I_d + \mathcal{O}(\eta).
\end{align*}
Substituting into the trace expression for $J_+$ yields
\begin{align*}
    \hat{\mathbb{E}}J_+ &= \Tr(\hat{\mathbb{E}}[H_\eta^{-1}] C^{-1})\\
    &= J + \eta \Tr(B_{-1}C^{-1}) + \mathcal{O}(\eta^{3/2}J)\\
    &= J + \eta \Tr(\frac{2}{N}\left(1-\frac{2a}{N}\right)\gamma \rmZ^\top \tG C^{-1/2} - 2BC^{-1})+ \frac{2}{N}\eta\gamma(2d+3-N)J + \mathcal{O}(\eta^{3/2}J)\\
    &= J + \eta \gamma J\left(-2 + \frac{4d+6}{N}\right)  \\
    &\qquad + \eta \Tr(\frac{2}{N} \left(1 - \frac{2a}{N}\right) \gamma C^{-1} \rmY^\top \tG - 2\frac{d+1}{N} \gamma C^{-1} - \frac{4\gamma'}{N}I_d) + \mathcal{O}(\eta^{3/2}J)\\
    &= J +  \eta \gamma J \left(-2 + \frac{2d+4}{N}\right) \\
    &\qquad + \eta \gamma \left(2 - \frac{4a}{N}\right) \Tr(KC^{-1}) - \frac{4d\gamma'}{N} \eta + \mathcal{O}(\eta^{3/2} J),
\end{align*}
where the last equality uses $\rmY^\top \tG = NK$ and the definitions of $\rmZ$ and $B$. Using the upper bound $\Tr(KC^{-1}) \le L \sqrt{TJ}$, we conclude
\begin{equation}
    \hat{\mathbb{E}}\left[\frac{J_+}{J}\right] \le 1 + \left(-2 + \frac{2d+4}{N}\right) \gamma \eta + \left(2 - \frac{4a}{N}\right) L \sqrt{\frac{T}{J}} \gamma \eta + \frac{4d |\gamma'|}{NJ}\eta + \mathcal{O}(\eta^{3/2}). \label{eq:JplusJBound}
\end{equation}

Combining \labelcref{eq:logZBound,eq:JplusJBound} yields, for constants $\kappa_1, \kappa_2>0$ independent of $a$ and depending only on $N,d,L,\theta$,
\begin{align}
    \frac{Q_\eta W_a}{W_a} &\le \exp \left(\eta \gE_{a} + \kappa_1 \eta^2 (1+\gamma D)\right) (1 + \eta\gJ_{a} + \kappa_2 \eta^{3/2}),\\
    \gE_a &\coloneqq -c_a \gamma D + \frac{a}{N}(d+1+N) \gamma LT + \frac{2aL}{N} |\gamma'|T^2,\\
    \gJ_a &\coloneqq \left(-2 + \frac{2d+4}{N}\right) \gamma + \left(2 - \frac{4a}{N}\right) L \sqrt{\frac{T}{J}} \gamma + \frac{4 d|\gamma'|}{NJ}, \label{eq:Ja}
\end{align}
for all sufficiently small $\eta$ such that \Cref{cond:boundedS} holds.

\subsubsection{Foster--Lyapunov drift condition}\label{sssec:discreteFosterLyapunov}
We now choose an $a>0$ such that for all sufficiently small $\eta$, \Cref{cond:boundedS} holds and the Markov kernel $Q_\eta$ satisfies a Foster--Lyapunov decay on $W_a$.

First observe that $\gJ_a \ge -2$. Moreover, using $J^{-1} \le Td^{-2}$, we have that the second and third terms in \labelcref{eq:Ja} are $\mathcal{O}(\gamma T) = \mathcal{O}(1)$ and $\mathcal{O}(|\gamma'| T) = \mathcal{O}(1)$ respectively. Therefore, $\gJ_a$ is uniformly bounded over all $\rmX$ for $a\le 1$. We thus assume the following condition also holds:
\begin{condition}\label{cond:logWellDefined}
    The step size $\eta$ is sufficiently small such that for all $a \in [0,1]$, we have $|\eta \gJ_a| \le 1/2$.
\end{condition}
Using the inequality $\log(1+u) \le u$, it is sufficient to show that 
\begin{equation}
    \eta \gE_a  + \kappa_1 \eta^2 (1+\gamma D) + \eta \gJ_a + \kappa_2 \eta^{3/2}
\end{equation}
is uniformly negative away from a compact set. Since $f$ is distantly convex, there exist constants $C_T>0, B_T \in \R$ such that $D \ge C_T T^2 - B_T$. 

Choose $a>0$ sufficiently small such that 
\begin{equation}
    \sup_{T \ge 0}\left\{ -\frac{c_a}{2} (C_T T^2 - B_T) + \frac{a}{N}(d+1+N) LT + \frac{2aL}{N} T \right\}\le 1 - \frac{d+2}{N}.
\end{equation}
This is possible since all coefficients are $\mathcal{O}(a)$. For this choice of $a$, \labelcref{eq:gammaBounds} gives that 
\begin{equation*}
    \gE_a \le \left(1 - \frac{d+2}{N}\right) \gamma - \frac{c_a}{2} \gamma D.
\end{equation*}

For this choice of $a>0$, it is now sufficient to show that
\begin{equation}
    -\frac{c_a}{2} \eta \gamma D + \kappa_1 \eta^2 (1+\gamma D) + \left(-1 + \frac{d+2}{N}\right)\eta\gamma + \left(2 - \frac{4a}{N}\right) L \sqrt{\frac{T}{J}} \eta \gamma + \frac{4 d|\gamma'|}{NJ}\eta + \kappa_2 \eta^{3/2}
\end{equation}
 is uniformly negative outside a compact set. We now let $\eta$ be sufficiently small such that $-c_a \eta \gamma D/4 + \kappa_1 \eta^2 \gamma D \le 0$ (noting $D \ge 0$), and absorb $\kappa_1 \eta^2$ into $\kappa_2 \eta^{3/2}$.

It now suffices to show that 
 \begin{equation} \label{eq:FinalDiscreteBound}
     h \coloneqq -\frac{c_a}{4} \eta \gamma D + \left(-1 + \frac{d+2}{N}\right)\eta \gamma + \left(2 - \frac{4a}{N}\right) L \sqrt{\frac{T}{J}} \eta \gamma + \frac{4 d|\gamma'|}{NJ}\eta + \kappa_3 \eta^{3/2}
 \end{equation}
 is uniformly negative outside a compact set, where $\kappa_3$ is another constant independent of $a, \eta$. The argument proceeds with a similar partition as in the continuous case.
     
 \begin{enumerate}
     \item Since the first term of \labelcref{eq:FinalDiscreteBound} is negative and of order $\gamma T^2 \sim T$, and the other terms are all uniformly bounded by a constant, we can choose $T_0$ sufficiently large such that for all $T \ge T_0$, 
     \begin{equation*}
         -\frac{c_a}{4} \eta \gamma D + \left(-1 + \frac{d+2}{N}\right)\eta \gamma + \left(2 - \frac{4a}{N}\right) L \sqrt{\frac{T}{J}} \eta \gamma + \frac{4 d|\gamma'|}{NJ}\eta \le -2\eta,
     \end{equation*}
     for all $\eta$.
     
     \item On $\{T \le T_0\}$, $\gamma$ is lower bounded by $\gamma \ge \gamma_0 > 0$. Now choose $\eta$ sufficiently small such that 
     \begin{equation*}
         \kappa_3 \eta^{3/2} \le \frac{1}{2}\left(1 - \frac{d+2}{N}\right) \gamma_0 \eta\quad \text{and} \quad \kappa_3 \eta^{3/2} \le \eta.
     \end{equation*}
     This is the final step size restriction. We thus have that on $\{T \ge T_0\}$, $h \le -\eta$; furthermore on $\{T \le T_0\}$,
     \begin{equation*}
         h \le h' \coloneqq -\frac{c_a}{4} \eta \gamma D + \frac{1}{2}\left(-1 + \frac{d+2}{N}\right) \eta \gamma + \left(2 - \frac{4a}{N}\right) L \sqrt{\frac{T}{J}} \eta \gamma + \frac{4 d|\gamma'|}{NJ}\eta.
     \end{equation*}
    The first term is non-positive. Choosing sufficiently large $J_0$, we have for $J \ge J_0$,  
    \begin{equation*}
        h' \le \frac{1}{4} \left(-1 + \frac{d+2}{N}\right) \eta \gamma \le \frac{1}{4}\left(-1 + \frac{d+2}{N}\right) \eta \gamma_0 < 0.
    \end{equation*}

     \item On $\{T \le T_0, J \le J_0\}$, we can now use the lower bound $D \ge c_2 E/J_0 - b_2 / J_0$ from \Cref{prop:DistantConvexity}. Finally choose sufficiently large $E_0$ such that if $E \ge E_0$, then
     \begin{equation*}
         h' \le \frac{1}{4}\left(-1 + \frac{d+2}{N}\right) \eta \gamma_0,
     \end{equation*}
     since all other terms in $h'$ are bounded by a constant (depending on $T_0, J_0$).

     \item The remaining region $\mathrm{K} = \{T \le T_0, J \le J_0, E \le E_0\} \subset 
     \mathrm{M}$ is compact from \Cref{prop:KCompact}. Since all terms inside the exponential are $\mathcal{O}(\eta)$, $h$ is upper bounded on this region by $b\eta >0$ for all $\eta \le 1$, where $b$ depends on $a$ but not $\eta$.
 \end{enumerate}

We have shown that there exists $a>0$ such that for all sufficiently small $\eta$, the following bound holds:
\begin{equation}
    \frac{Q_\eta W_a}{ W_a} \le \exp(\frac{1}{4}\left(-1 + \frac{d+2}{N}\right)\eta \gamma_0 + b \eta \mathbf{1}_\mathrm{K})
\end{equation}
for some constant $b$ and compact set $\mathrm{K}$ depending on $a$ but not on $\eta$. Since $N \ge d+3$, this shows the desired Foster--Lyapunov condition. Geometric ergodicity follows from \cite[Thm. 6.3]{meyn1992stability1}.

\subsubsection{Weak convergence of stationary distributions} \label{sssec:weakconv1}
To show that the stationary distributions converge, we use that the same Lyapunov function is used in both continuous and discrete-time Foster--Lyapunov conditions. Under the same assumptions, \Cref{thm:timeScaleErgo} gives that $\Pi_*$ is the unique invariant distribution of $\bar \gL$. 

The collection of discrete-time stationary distributions $\Pi_{*, \eta}$ for $\eta < \eta_*$ are tight: since $a>0$ is fixed, from \labelcref{eq:DiscreteTimeLyapunovDecay}, we have that 
\begin{equation}
    Q_{\eta} \Pi_{*,\eta} W_a \le (1 - c \eta) \Pi_{*, \eta}W_a  + b \eta \quad \Rightarrow \quad \Pi_{*, \eta}W_a  \le \frac{b}{c}
\end{equation}
for some $c,b>0$ independent of $\eta$. By Markov's inequality we obtain that
\begin{equation*}
    \Pi_{*, \eta} \mathbf{1}(W_a \ge R) \le \frac{b}{cR}.
\end{equation*}
Since $W_a$ is coercive, we have that the family of measures $\{\Pi_{*, \eta} \mid \eta \in(0,\eta_*)\}$ is tight. 

By Prokhorov's theorem, there exists a limiting probability measure $\nu$ such that up to a subsequence, $\Pi_{*, \eta} \rightharpoonup \nu$. We now wish to show that $\nu = \Pi_*$, the stationary distribution of the continuous flow. By \cite[Thm. 4.9.17]{echeverria1982criterion,ethier2009markov}, it suffices to show that $\nu(\bar \gL \phi) = 0$ for all test functions $\phi \in \gC_c^\infty(\mathrm{M})$, where we recall $\bar \gL$ is the generator of the continuous time-scaled process. Since $\Pi_{*, \eta}$ is stationary,
\begin{equation}
    0 = \int \frac{Q_\eta \phi - \phi}{\eta}\dd{\Pi_{*, \eta}} = \int \left(\frac{Q_\eta \phi - \phi}{\eta} - \bar \gL \phi\right)\dd{\Pi_{*,\eta}} + \int \bar \gL \phi \dd{\Pi_{*,\eta}}.
\end{equation}
It remains to prove that for any test function $\phi$, 
\begin{equation}\label{eq:QDiscreteSupPhi}
    \left\|\frac{Q_\eta \phi - \phi}{\eta} - \bar \gL \phi\right\|_{\infty} \rightarrow 0\quad \text{as } \eta \rightarrow 0,
\end{equation}
which would show that $\Pi_{*, \eta} (\bar \gL \phi) \rightarrow 0$ as $\eta \rightarrow 0$ and therefore $\nu(\bar \gL \phi) =  0$.

Recall from \labelcref{eq:XPlus} that the discrete-time update can be written in terms of a deterministic and random component $\rmX_+ = \rmX + \eta \rmH + \Xi \sqrt{2 \gamma \eta C}$. Moreover, the continuous generator $\bar \gL$ of the time scaled process \labelcref{eqs:scalingDiffusionDefinition} can be written as follows, where $\rvh_j$ are the rows of $\rmH$,
\begin{equation*}
     \bar \gL V = \sum_j \rvh_j \cdot \nabla_j V + \gamma \sum_j C\!:\! \nabla^2_j V.
\end{equation*}
Let $S' = \supp \phi + \bar B_r(0)$ for sufficiently small $r>0$ such that $S' \subset \mathrm{M}$. Within $S'$, the uniform limit \labelcref{eq:QDiscreteSupPhi} holds directly by Taylor's theorem, as $\mathbb{E}\|\rmX_+ - \rmX\|^3_F = \mathcal{O}(\eta^{3/2})$ uniformly. 

Outside $S'$, we have that $\bar \gL \phi = 0$. Showing \labelcref{eq:QDiscreteSupPhi} follows from proving
\begin{equation*}
    \sup_{S'^c} \left|\eta^{-1} {Q_\eta \phi}\right| \rightarrow 0.
\end{equation*}

This follows from Gaussian concentration: recall from \labelcref{eqs:timescaleIneqs} that $\|\rmH\|_F \lesssim 1 + \|\rmX\|_F$ (since $\gamma D \lesssim \|\rmX\|_F^2$). Therefore, for sufficiently small $\eta$ independent of $\rmX$, we have that $\eta \|\rmH\|_F \le \frac{1}{2}\dist(\rmX, \supp \phi)$ for all $\rmX \in S'^{c}$. Therefore,
\begin{align*}
    \eta^{-1} |Q_\eta \phi(\rmX)| &\le \eta^{-1} \mathbb{E} |\phi(\rmX_+)|\\
    &\le \eta^{-1} \|\phi\|_\infty \mathbb{P}(\rmX_+ \in \supp \phi)\\
    &\le \eta^{-1} \|\phi\|_\infty \mathbb{P}(\|\Xi \sqrt{2 \eta \gamma C}\|_F \ge \dist(\rmX, \supp \phi) - \eta \|\rmH\|_F)\\
    &\le \eta^{-1} \|\phi\|_\infty \mathbb{P}(\|\Xi \sqrt{2 \eta \gamma C}\|_F \ge \frac{1}{2}\dist(\rmX, \supp \phi))\\
    &\lesssim \eta^{-1} \|\phi\|_\infty \exp(-c \dist(\rmX, \supp \phi)^2 \eta^{-1})\\
    &\lesssim \eta^{-1} \|\phi\|_\infty \exp(-cr^2 \eta^{-1}) \rightarrow 0
\end{align*}
as $\eta \rightarrow 0$, since $\|\gamma C\| \le \theta^{-1}$. Here $c$ is a constant depending on $\supp \phi$ and $\theta$, and is independent of $\rmX$. Taking supremum yields the desired limit \labelcref{eq:QDiscreteSupPhi}, and therefore $\nu = \Pi_*$. Taking subsequences concludes that $\Pi_{*, \eta} \rightharpoonup \Pi_*$.

\subsection{Proof of Theorem \ref{thm:tamedALDI}}\label{ssec:proofTamedALDI}
The proof follows the preceding argument for \Cref{thm:discreteTimeScaleErgo}, but the deterministic update is not uniformly bounded as $T$ grows. We therefore use two estimates of $\hat{\mathbb{E}}[(J_+/J)^\kappa]$: a first-order expansion for bounded $T$, which yields negative drift, and a global bound whose growth is absorbed by the normalizing constant for large $T$.

We prove the result now for $N \ge d+2$ using the Lyapunov function defined in Appendix \ref{appsec:dplus2}: for a fixed $\kappa \in (0,1)$ and appropriate $a\in (0, 1)$,
\begin{equation*}
    V_{a,\kappa} \coloneqq e^{\frac{a}{N} \sum_i f(\rvx_i)} \Tr(C^{-1})^\kappa.
\end{equation*}
Following the notation of the previous section, the update can be written $\rmX_+ = \rmX + \delta \rmX$, where
\begin{align*}
    \delta \rmX &= h_{\eta, \delta} \rmH + \Xi \sqrt{2 h_{\eta, \delta} C},\\
    \rmH &\coloneqq -\nabla F C + \frac{d+1}{N}P_c \rmX,
\end{align*}
and $\Xi$ is an $N\times d$ matrix with i.i.d. standard Gaussian entries. The Markov kernel $Q_{\eta,\delta}$ applied to the Lyapunov function $V_{a,\kappa}$ is
\begin{equation*}
    \frac{Q_{\eta,\delta}V_{a,\kappa}}{V_{a,\kappa}} = \mathbb{E} \left[e^{a(E_+-E)} \left(\frac{J_+}{J}\right)^\kappa\right].
\end{equation*}
We now consider the same tilting term
\begin{align*}
    \Psi \coloneqq \exp( \frac{a}{N} \left[\nabla F(\rmX) \!:\! \delta \rmX + \frac{L}{2} \|\delta \rmX\|_F^2\right]),\quad 
    \hat Z \coloneqq \mathbb{E} \Psi.
\end{align*}
Define a new probability density $\mathrm{d}\hat{\mathbb{P}}/\mathrm{d}\mathbb{P} = \Psi/\hat Z$, so the change of measure gives
\begin{align*}
    \frac{Q_{\eta,\delta} V_{a,\kappa}}{V_{a,\kappa}} \le \hat Z \hat{\mathbb{E}} \left[\left(\frac{J_+}{J}\right)^\kappa \right].
\end{align*}

\subsubsection{Tilting: normalizing constant}\label{sssec:tamedNormalizing}
By definition of $\hat Z$, applying \Cref{lem:gaussianClosedForms} with $\ell = \frac{a}{N} (\nabla F + L h_{\eta,\delta} \rmH) \sqrt{2 h_{\eta,\delta} C}$ and $S = \frac{a}{N} L h_{\eta,\delta} (I_N \otimes  C)$, we obtain
\begin{equation}
    \log \hat Z = -\frac{1}{2} \log\det (I - 2S) + \frac{1}{2} \ell^\top (I-2S)^{-1}\ell + \frac{a}{N} h_{\eta, \delta}\nabla F \!:\!\rmH + \frac{aL}{2N} h_{\eta, \delta}^2 \|\rmH\|_F^2. \label{eq:loghatZ2}
\end{equation}

To make sure that $I-2S$ is uniformly positive for sufficiently small $a, \eta$, we need the following bounds.
\begin{lemma}
    There are constants $s_0, s_1>0$ independent of $\eta,\delta$ such that
    \begin{equation}
        h_{\eta,\delta} \le \eta,\quad h_{\eta,\delta} \sqrt{D} \le \delta,\quad h_{\eta,\delta} T \le s_0 \eta + s_1 \delta. \label{eq:hBounds}
    \end{equation}
    Furthermore, if $0<\eta \le \delta$, then
    \begin{equation}
        h(1+D) \ge \eta. \label{eq:hDEtaBound}
    \end{equation}
\end{lemma}
\begin{proof}
    The first two inequalities follow from the definition of $h_{\eta,\delta}$. For the last inequality, from \labelcref{eq:DLowerbound}, there exist constants $C_T, B_T>0$ such that $D \ge C_T T^2 - B_T$. Therefore, $C_T T^2 \le D + B_T$, and 
    \begin{equation*}
        T \le C_T^{-1/2} (\sqrt{D} + \sqrt{B_T}) = s_0 + s_1 \sqrt{D},\quad s_0 \coloneqq \sqrt{B_T/C_T},\, s_1 \coloneqq C_T^{-1/2}.
    \end{equation*}
    Multiplying by $h_{\eta,\delta}$ and using the first two inequalities of \labelcref{eq:hBounds} gives the last inequality. 

    For \labelcref{eq:hDEtaBound}, if $h_{\eta,\delta} = \eta$, then the bound is trivial. If $h_{\eta,\delta} < \eta$, i.e. $\delta /\sqrt{D} < \eta$. Therefore $\sqrt{D} > \delta/\eta \ge 1$, and 
    \begin{equation*}
        h_{\eta,\delta}D = \delta \sqrt{D}  \ge \delta \ge \eta.
    \end{equation*}
\end{proof}
Now use \labelcref{eq:hBounds} and $0 \preceq S \preceq aLh_{\eta,\delta} T N^{-1} I_{Nd}$. Fix $\delta_* = L^{-1}(s_0+s_1)^{-1}/4$. Then for any $\delta \in (0,\delta_*]$ and $\eta \le \delta$, it holds that $I_{Nd} \preceq (I_{Nd}-2S)^{-1} \preceq 2I_{Nd}$. Here we used the restriction $a < 1 \le N/2$. To summarize, we have the following restriction on the parameters.
\begin{condition}\label{cond:tamed}
    Replacing $L$ with $\max(1,L)$ if necessary, we assume that $L(s_0+s_1)\delta \le 1/4$, and that $\eta \le \min(1,\delta)$. Then $h_{\eta,\delta} \le 1$, $h_{\eta,\delta}T \le 1$, $I_{Nd} \preceq (I_{Nd}-2S)^{-1} \preceq 2I_{Nd}$ and all Taylor expansions have uniformly bounded remainder.
\end{condition}

We now upper bound $\log \hat Z$ in terms of powers of $h_{\eta,\delta}$. 

\textbf{Log-determinant term.} For $0 \le u \le 1/2$, $-\log (1-u) \le 2u$. Applying this to the eigenvalues of $S$, we obtain
\begin{equation}
    -\frac{1}{2} \log\det (I-2S) \le  2aL h_{\eta,\delta} T.
\end{equation}

\textbf{Quadratic term.} From the condition, we get $(I-2S)^{-1} \preceq 2I$. Therefore the quadratic term can be bounded from above as
\begin{align*}
    \frac{1}{2} \ell^\top (I-2S)^{-1} \ell &= \frac{a^2}{2N^2} \left(\nabla F + Lh_{\eta,\delta}\rmH\right) (2h_{\eta,\delta}C) (I - \frac{2a}{N}Lh_{\eta,\delta}C)^{-1} \!:\! (\nabla F + Lh_{\eta,\delta}\rmH)\\
    & \le \frac{2a^2}{N^2} h_{\eta,\delta} (\nabla F + Lh_{\eta,\delta}\rmH)C\!:\!(\nabla F + Lh_{\eta,\delta}\rmH).
\end{align*}
Using the bounds $\Tr(KC) \le LT^2$ and $T^2 \le c(D+1)$,
\begin{align*}
    \nabla F C \!:\! \nabla F &= ND,\\
    |\nabla F C \!:\! \rmH| &\le NTD + (d+1)|\Tr(KC)| \le cNT(D+1),\\
    \rmH C\!:\! \rmH &\le cNT^2(D+1)
\end{align*}
where $c$ are some positive constants depending on $C_T, B_T,N,d,L$, and are independent of $a,\eta,\delta$. Therefore
\begin{equation*}
    \frac{1}{2} \ell^\top (I-2S)^{-1}\ell \le \frac{a^2}{N^2} h_{\eta,\delta} \left[2ND + c h_{\eta,\delta}T(D+1) + c h_{\eta,\delta}^2 T^2(D+1)\right]. 
\end{equation*}

\textbf{Drift terms.} Direct computation yields
\begin{align*}
    \frac{a}{N}h_{\eta,\delta} \nabla F \!:\! \rmH &= -ah_{\eta,\delta}D + a\frac{d+1}{N} h_{\eta,\delta} \Tr K \\
    &\le -ah_{\eta,\delta}D + a\frac{d+1}{N} h_{\eta,\delta} LT.
\end{align*}
The final term of \labelcref{eq:loghatZ2} can be bounded using
\begin{equation*}
    \|\rmH\|_F^2 \le 2 \|\nabla FC\|_F^2 + 2\|\frac{d+1}{N} P_c \rmX\|_F^2 \le 2NTD + 2 \left(\frac{d+1}{N}\right)^2NT.
\end{equation*}
Combining everything yields the upper bound for $\log \hat Z$,
\begin{align*}
    \log \hat Z &\le -ah_{\eta,\delta} D + a\frac{d+1}{N} h_{\eta,\delta} LT + aLh_{\eta,\delta}^2\left(TD + \left(\frac{d+1}{N}\right)^2 T\right) + 2aLh_{\eta,\delta}T \\
    &\qquad + \frac{a^2}{N^2} h_{\eta,\delta} \left[2ND + c h_{\eta,\delta}T(D+1) + c h_{\eta,\delta}^2 T^2(D+1)\right].
\end{align*}
Collecting terms and using $h_{\eta,\delta} , h_{\eta,\delta}T\le 1$ gives
\begin{equation*}
    \log \hat Z \le -ah_{\eta,\delta}\left[1 - \frac{2a}{N}- c(h_{\eta,\delta}T)- c(h_{\eta,\delta}T)^2\right]D + c ah_{\eta,\delta} (1+T).
\end{equation*}
Now decrease $\delta_*$ (and therefore $\eta_*$, so $\sup h_{\eta,\delta}T \rightarrow 0$) such that for all sufficiently small $a$, the bracketed coefficient of $D$ is lower bounded by $1/2$. Since $D \ge C_T T^2 - B_T$, taking the supremum over $T \ge 0$ allows us to absorb the second term. We therefore get, for $\delta \le \delta_*$ and $\eta \le \delta$ satisfying \Cref{cond:tamed},
\begin{equation}
    \log \hat Z \le -\frac{a}{4} h_{\eta,\delta}D + C_1ah_{\eta,\delta},\label{eq:loghatZ}
\end{equation}
where $C_1$ is some constant depending on the function $f$ and $N,d$, and is independent of $\eta,\delta$.

\subsubsection{Tilting: expectation for small $T$}\label{sssec:tiltsmallT} An identical calculation to \Cref{sssec:expectation} holds, with $\eta \gamma$ replaced with $h_{\eta,\delta}$ and $\rmH$ replaced with the appropriate definition. In particular, letting $\rmY = P_c \rmX$ and $\rmZ = \rmY C^{-1/2}$, we get that the centered update satisfies
\begin{equation}
    P_c \rmX_+ C^{-1/2} = \rmZ + h_{\eta,\delta} U_h + \sqrt{2 h_{\eta,\delta}} P_c \Xi' R_h^{1/2} \label{eq:PcXC}
\end{equation}
where $\Xi'$ is an $N \times d$ matrix with i.i.d.\ standard Gaussian entries, and $U_h, R_h$ are defined as
\begin{align}
    U_h &\coloneqq \left[-\left(1 - \frac{2a}{N}\right) \tG C^{1/2} + \frac{d+1}{N} \rmZ\right] R_h,\\
    R_h &\coloneqq \left(I_d - \frac{2aLh_{\eta,\delta}}{N}C\right)^{-1},
\end{align}
and $\tG = P_c \nabla F$ is the centered gradient matrix. Defining $A = C^{-1}/J$ and
\begin{equation*}
    H_h \coloneqq \frac{1}{N} (P_c \rmX_+ C^{-1/2})^\top (P_c \rmX_+ C^{-1/2}),
\end{equation*}
the updated covariance satisfies $H_h^{-1} = C^{1/2} C_+^{-1} C^{1/2}$, and therefore $J_+/J = \Tr(AH_h^{-1})$. On a region $\{T \le T_0\}$, $U_h$ is bounded and satisfies $P_c U_h = U_h$. By \Cref{cond:tamed}, $I_d \preceq R_h \preceq 2I_d$. Since $A \succeq 0$ and $\Tr A = 1$, we can now use the improved \Cref{applem:inverseMatBound} to obtain the first order estimate 
\begin{align*}
    & \hat{\mathbb{E}} \left[\left(\frac{J_+}{J}\right)^\kappa\right] = \hat{\mathbb{E}} \left[\Tr(AH_h^{-1})^\kappa\right] \\
    &= 1 + h_{\eta,\delta} \left[-2\kappa \left(1 - \frac{d+2}{N}\right) + 2 \kappa\left(1 - \frac{2a}{N}\right) \frac{\Tr(KC^{-1})}{J} - \frac{4\kappa(1-\kappa)}{N} \Tr(R_h A^2)\right] \\ &\qquad + \mathcal{O}(h_{\eta,\delta}^{3/2})\\
    &\le 1 + h_{\eta,\delta}\left[-\lambda_{\kappa} + \beta_{a,\kappa} \frac{\Tr(KC^{-1})}{J}\right] + C_{T_0} h_{\eta,\delta}^{3/2}.
\end{align*}
We used $\Tr(A^2) \ge d^{-1}$ in the last inequality, which holds by Cauchy--Schwarz. Here the constant $C_{T_0}$ depends on $T_0$, the estimate holds over $\{T \le T_0\}$, and the constants are
\begin{align*}
    \beta_{a,\kappa} &\coloneqq 2\kappa\left(1-\frac{2a}{N}\right)>0,\\
    \lambda_{\kappa} &\coloneqq 2\kappa\left(1-\frac{d+2}{N}\right)+\frac{4\kappa(1-\kappa)}{Nd}>0.
\end{align*}
This gives a desired first order bound over bounded $T$ when $h_{\eta,\delta}$ is sufficiently small. 

\subsubsection{Tilting: expectation for large $T$}\label{sssec:expectationLargeT}
We now apply \labelcref{eq:PcXC} to \Cref{applem:inverseMatBound} using modified parameters to obtain a global upper bound. Since $h_{\eta,\delta}T \le s_0 \eta + s_1 \delta$, we consider the parameterization
\begin{equation*}
    \tau \coloneqq h_{\eta,\delta}(1+T),\quad \tilde U \coloneqq \frac{U_h}{1+T},\quad \tilde R \coloneqq \frac{R_h}{1+T},\quad \tilde \gamma \coloneqq \frac{1}{1+T}.
\end{equation*}
Then $\tilde \gamma \in [0,1]$, $\tilde \gamma I_d \preceq \tilde R \preceq 2 \tilde \gamma I_d$, and $\tilde U$ is also uniformly bounded since $\|\tG\|^2_F \le NL^2T$. Using \Cref{applem:inverseMatBound}, the first order term is uniformly bounded as $\tau \rightarrow 0$. Therefore, we have
\begin{equation}
    \hat{\mathbb{E}}\left[\left(\frac{J_+}{J}\right)^\kappa\right] \le 1 + C_2 \tau \le \exp( C_2 h_{\eta,\delta}(1+T)), \label{eq:globalJJkappaBound}
\end{equation}
where $C_2$ is some constant independent of $\eta, \delta$. Since $\tau \le (1+s_0)\eta +s_1 \delta$, we may decrease $\delta_*$ and $\eta_*$ as necessary such that the bound holds uniformly.

\subsubsection{Foster--Lyapunov drift condition}
Fix any $\kappa \in (0,1)$, and let $a$ satisfy
\begin{equation*}
    0 < a \le \min \left(1, N/8, \lambda_{\kappa}/(4C_1)\right),
\end{equation*}
where $C_1$ is the constant coefficient appearing in \labelcref{eq:loghatZ}, and $\delta_*$ be chosen sufficiently small to satisfy $\delta_* \le \max(1,L)^{-1}(s_0+s_1)^{-1}/4$ and the conditions at the ends of \Cref{sssec:tamedNormalizing,sssec:expectationLargeT}. Then for $\delta \le \delta_*$ and any $\eta \le \min(1,\delta)$,
\begin{equation}
    \log \hat Z \le -\frac{a}{4} h_{\eta,\delta} D + \frac{\lambda_\kappa}{4} h_{\eta,\delta}. \label{eq:logHatZLambda}
\end{equation}

For large $T$, we can use the bound $D \ge C_T T^2 - B_T$ to obtain the desired control. Combining  \labelcref{eq:globalJJkappaBound,eq:logHatZLambda}, we get the uniform bound
\begin{equation*}
    \log \frac{Q_{\eta,\delta}V_{a,\kappa}}{V_{a,\kappa}} \le -\frac{a}{4} h_{\eta,\delta} D + \frac{\lambda_\kappa}{4}h_{\eta,\delta}+ C_2 h_{\eta,\delta}(1+T).
\end{equation*}
Using the quadratic growth of $D$, there exists a $T_0$ independent of $\eta$ such that, 
\begin{equation}
 \log \frac{Q_{\eta,\delta}V_{a,\kappa}}{V_{a,\kappa}} \le -\frac{a}{8} h_{\eta,\delta} (1+D),\quad T \ge T_0. \label{eq:tamedTgeT0}
\end{equation}

Using this $T_0$ in \Cref{sssec:tiltsmallT} and using $\log(1+u) \le u$, we obtain a uniform bound on $\{T \le T_0\}$: for a now fixed $C_{T_0}>0$,
\begin{align*}
     \log \frac{Q_{\eta,\delta}V_{a,\kappa}}{V_{a,\kappa}} \le -\frac{a}{4} h_{\eta,\delta} D + \frac{\lambda_\kappa}{4} h_{\eta,\delta}+  h_{\eta,\delta}\left[-\lambda_{\kappa} + \beta_{a,\kappa} \frac{\Tr(KC^{-1})}{J}\right] + C_{T_0} h_{\eta,\delta}^{3/2}.
\end{align*}
Finally shrink $\eta_*$ such that $C_{T_0} \sqrt{\eta_*} \le \lambda_\kappa/4$. Since $h_{\eta,\delta} \le \eta$, this yields the bound on $\{T \le T_0\}$:
\begin{equation*}
    \log \frac{Q_{\eta,\delta}V_{a,\kappa}}{V_{a,\kappa}} \le h_{\eta,\delta} \left[-\frac{a}{4} D - \frac{\lambda_\kappa}{2} + \beta_{a,\kappa} \frac{\Tr(KC^{-1})}J{}\right],\quad T \le T_0.
\end{equation*}
The conclusion comes from manipulating this inequality similarly to before, but now we need to bound using $h_{\eta,\delta}(1+D) \ge \eta$ in order to obtain the desired result in terms of $\eta$.
\begin{enumerate}
    \item Since $\Tr(KC^{-1})/J \le L\sqrt{T/J}$, choose $J_0$ sufficiently large such that $\beta_{a,\kappa} L\sqrt{T_0/J_0} \le \lambda_\kappa/4$. Then on $\{T \le T_0, J \ge J_0\}$, we get
    \begin{equation*}
        \log \frac{Q_{\eta,\delta}V_{a,\kappa}}{V_{a,\kappa}} \le h_{\eta,\delta} \left[-\frac{a}{4} D - \frac{\lambda_\kappa}{4}\right] \le- c h_{\eta,\delta}(1+D),
    \end{equation*}
    where $c = \min(a/4, \lambda_\kappa/4)>0$ here.
    \item On $\{T \le T_0, J \le J_0\}$, we use the lower bound $D \ge c_2 E/J_0 - b_2 / J_0$ from \Cref{prop:DistantConvexity}. Since $\Tr(KC^{-1})/J \le LT_0/d$ on this region, choose $E_0$ sufficiently large such that on $\{E \ge E_0\}$, $aD/8 \ge \beta_{a,\kappa}LT_0/d$. Then on $\{T \le T_0, J \le J_0, E \ge E_0\}$,
    \begin{equation*}
        \log \frac{Q_{\eta,\delta}V_{a,\kappa}}{V_{a,\kappa}} \le h_{\eta,\delta} \left[-\frac{a}{8} D - \frac{\lambda_\kappa}{2}\right] \le -c h_{\eta,\delta}(1+D),
    \end{equation*}
    where $c = \min(a/8, \lambda_\kappa/2)$ here.
    \item The remaining region $\mathrm{K} = \{T \le T_0, J \le J_0, E \le E_0\} \subset \mathrm{M}$ is compact. Note that it does not depend on $\delta$.
\end{enumerate}

Combining \labelcref{eq:tamedTgeT0} and the above partitions, we obtain the desired Foster--Lyapunov condition: for any $\kappa \in (0,1)$, there exists $a,b,c>0$, a compact $\mathrm{K} \subset \mathrm{M}$, and $\delta_*>0$ such that for any $\delta \le \delta_*$, there exists $\eta_* \le \delta$ such that for any $\eta \le \eta_*$,
\begin{align*}
    \log \frac{Q_{\eta,\delta}V_{a,\kappa}}{V_{a,\kappa}} \le -ch_{\eta,\delta}(1+D) + b h_{\eta,\delta} \mathbf{1}_\mathrm{K} \le -c \eta + b \eta \mathbf{1}_\mathrm{K}.
\end{align*}
Taking exponentials and using the continuity of $V_{a,\kappa}$ yields the desired Foster--Lyapunov condition. The update has a continuous, strictly positive Gaussian transition density on $\mathrm{M}\times\mathrm{M}$, so compact sets are small and the chain is irreducible and aperiodic. Geometric ergodicity therefore follows from \cite[Thm. 6.3]{meyn1992stability1}.

\subsubsection{Weak convergence of stationary distributions}
Similarly to \Cref{sssec:weakconv1}, the invariant distributions $\Pi_{\eta,\delta}$ are tight as $\eta \rightarrow 0$. In particular, we have that $\Pi_{\eta,\delta} V_{a,\kappa} \le b/c$ uniformly as $\eta \rightarrow 0$. From \Cref{prop:DistantConvexity}, we have
\begin{equation*}
    E \ge \frac{c_1}{N} \|\rmX\|_F^2 - b_1,\quad J \ge \frac{d^2}{T} \ge \frac{Nd^2}{1+\|\rmX\|^2_F}.
\end{equation*}
In particular, for any $p \ge 1$, there exists a constant $A_p>0$ such that 
\begin{equation*}
     1 + \|\rmX\|_F^p\le A_p V_{a,\kappa}(\rmX).
\end{equation*}
Therefore, for any $p \ge 1$,
\begin{equation} \label{eq:uniformPBoundPi}
    \sup_{\eta \le \eta_*(\delta)} \int_\gM  (1+\|\rmX\|_F^p)\dd{\Pi_{\eta,\delta}(\rmX)} < \infty
\end{equation}

Suppose that $\Pi_{\eta,\delta} \rightharpoonup \nu$ for some probability distribution $\nu$. Denoting $\gL$ the generator of \labelcref{eq:aldiDynamics}, we have for any test function $\phi \in \gC_c^\infty(\mathrm{M})$,

\begin{equation}
    0 = \int \frac{Q_{\eta,\delta} \phi - \phi}{\eta}\dd{\Pi_{\eta,\delta}} = \int \left(\frac{Q_{\eta,\delta} \phi - \phi}{\eta} -  \gL \phi\right)\dd{\Pi_{\eta,\delta}} + \int \gL \phi \dd{\Pi_{\eta,\delta}}. \label{eq:zeroPhiWeakBound}
\end{equation}
It now suffices to show that for any test function $\phi \in \gC_c^\infty(\mathrm{M})$,
\begin{equation}
    \int \left|\frac{Q_{\eta,\delta} \phi - \phi}{\eta} -  \gL \phi\right|\dd{\Pi_{\eta,\delta}} \rightarrow 0 \qquad \qquad \text{as}\  \eta \rightarrow 0.\label{eq:ConvergenceToShow}
\end{equation}
Let $\phi$ be such a test function. Recall that the update is 
\begin{gather*}
    \rmX_+ = \rmX + \delta \rmX, \quad \delta \rmX = h_{\eta, \delta} \rmH + \Xi \sqrt{2 h_{\eta, \delta} C},\\
    \rmH \coloneqq -\nabla F C + \frac{d+1}{N}P_c \rmX
\end{gather*}
where $\Xi$ is an $N \times d$ i.i.d. Gaussian matrix. Extending $\phi$ to a $\gC_c^\infty(\R^{N \times d})$ function, Taylor expanding, we obtain for every $\rmX \in \R^{N \times d}$,
\begin{equation*}
    \phi(\rmX + \delta \rmX) - \phi(\rmX) =  \delta \rmX \!:\! \nabla \phi(\rmX) + \frac{1}{2} \nabla^2 \phi(\rmX)(\delta \rmX, \delta \rmX)+ \mathcal{O}(\|\delta \rmX\|_F^3),
\end{equation*}
where the remainder term is uniform depending on $\|D^3\phi\|_\infty$. Taking expectations,
\begin{multline}
    (Q_{\eta,\delta} \phi - \phi)(\rmX) = h_{\eta,\delta} \rmH \!:\! \nabla \phi(\rmX) +\frac{h_{\eta,\delta}^2}{2} \nabla^2 \phi(\rmX)(\rmH, \rmH) + h_{\eta,\delta}\sum_{i=1}^N C \!:\! \nabla_i^2 \phi(\rmX)  \\ + \mathbb{E}_{\rmX}\mathcal{O}(\|\delta \rmX\|_F^3). \label{eq:Qphiminusphi}
\end{multline}
The first order $h_{\eta,\delta}$ terms are precisely $h_{\eta,\delta} \gL \phi(\rmX)$. It remains to bound the higher order terms in terms of powers of $\|\rmX\|_F$, which will be uniformly bounded in expectation. Using the inequality $(u+v)^3 \le 4(u^3+v^3)$ for $u,v \ge 0$, we get
\begin{align*}
    \mathbb{E}_{\rmX} \|\delta \rmX\|_F^3 &\le 4h^3_{\eta,\delta} \|\rmH\|_F^3 + 4(2h_{\eta,\delta})^{3/2} \mathbb{E} \|\Xi C^{1/2}\|_F^3 \\
    &\lesssim h_{\eta,\delta}^3 \|\rmH\|_F^3 +  h^{3/2}_{\eta,\delta} T^{3/2}.
\end{align*}
Bounding the second order term as well in \labelcref{eq:Qphiminusphi}, we obtain
\begin{equation*}
    \left|Q_{\eta,\delta} \phi - \phi - h_{\eta,\delta} \gL \phi\right| \le C_{\phi, N, d} \left(h_{\eta,\delta}^{3/2} T^{3/2} + h_{\eta,\delta}^2 \|\rmH\|_F^2 + h_{\eta,\delta}^3 \|\rmH\|_F^3\right).
\end{equation*}
Using $h \le \eta$, divide by $\eta$ to obtain
\begin{equation*}
    \left|\frac{Q_{\eta,\delta} \phi - \phi}{\eta} - \frac{h_{\eta,\delta}}{\eta} \gL \phi\right| \le C_{\phi, N, d} \left(\sqrt{\eta} T^{3/2} + \eta \|\rmH\|_F^2 + \eta^2 \|\rmH\|_F^3\right). 
\end{equation*}
Since $\supp \phi$ is compact and $\delta$ is fixed, we get $h_{\eta,\delta} = \eta$ on $\supp \phi$ for sufficiently small $\eta$. Since $\supp \gL \phi \subseteq \supp \phi$, we get for sufficiently small $\eta$,
\begin{equation*}
    \left|\frac{Q_{\eta,\delta} \phi - \phi}{\eta} -  \gL \phi\right| \le C_{\phi, N, d} \left(\sqrt{\eta} T^{3/2} + \eta \|\rmH\|_F^2 + \eta^2 \|\rmH\|_F^3\right). 
\end{equation*}
Since $\nabla f$ is globally Lipschitz, it grows at most linearly, and therefore $\|\rmH\|_F \lesssim 1 + \|\rmX\|_F^3$. We also have $T(\rmX) \lesssim \|\rmX\|_F^2$. Inserting these estimates into the preceding inequality yields the uniform pointwise estimate
\begin{equation*}
     \left|\frac{Q_{\eta,\delta} \phi - \phi}{\eta} -  \gL \phi\right| \le C_{\phi, N, d} \sqrt{\eta}\left(1 + \|\rmX\|_F^9\right),
\end{equation*}
where the constant is independent of $\rmX$. Integrating with respect to $\Pi_{\eta,\delta}$ and using \labelcref{eq:uniformPBoundPi} with $p=9$ proves the desired \labelcref{eq:ConvergenceToShow}. Using \labelcref{eq:zeroPhiWeakBound} shows that $\Pi_{\eta,\delta} (\gL \phi) \rightarrow 0$ as $\eta \rightarrow 0$, therefore $\nu(\gL \phi) = 0$ by weak convergence. Since this is true for any test function $\phi$, we get $\nu = \Pi_*$ by \cite[Thm. 4.9.17]{echeverria1982criterion,ethier2009markov}. Since any subsequence of $\Pi_{\eta, \delta}$ has a subsequence converging to $\Pi_*$, we get that the entire sequence $\Pi_{\eta,\delta} \rightharpoonup \Pi_*$ weakly.
\color{black}

\section{Discussion}
At the continuous finite-particle level, we prove geometric ergodicity of the affine invariant ensemble Langevin dynamics using a Lyapunov function that controls both escape to infinity and covariance collapse. This is done under either a quadratic or super-quadratic tail assumption. At the discrete level, we show that when using the full ensemble covariance or leave-one-out covariance, the direct Euler--Maruyama scheme can diverge with positive probability. A covariance-trace time regularization restores geometric ergodicity in continuous and discrete time, retaining invariance under orthogonal transformations but sacrificing scale invariance. We further propose an explicit affine-invariant discretization based on clamping the maximum step size, and we show that it is geometrically ergodic.

The threshold $N\ge d+2$ comes from the inverse-covariance term in the Lyapunov estimates. For $N>d+2$, its drift has a negative constant coefficient, which disappears for $N= d+2$. On the threshold, the correction from the power $\Tr(C^{-1})^\kappa$ supplies the missing negative drift, as shown in \Cref{appsec:dplus2}. As noted in \cite{garbuno2020affine}, the admissible state space is path-disconnected when $N=d+1$, ruling out convergence to the full product target in continuous time. Whether a discrete affine-invariant scheme can be geometrically ergodic in this regime remains open.

Our Foster--Lyapunov argument proves the existence of exponential rates, but does not quantify their dependence on $N$ and $d$. The formal limits of the Lyapunov functions as $N\to\infty$ suggest one direction for studying this dependence.

More broadly, the techniques developed here are not specific to the dynamics \labelcref{eq:aldiDynamics}. Covariance degeneracy is the common obstacle to quantitative rates across affine-invariant ensemble methods, and we expect these tools to be useful for other affine invariant methods in sampling and data assimilation, for which geometric ergodicity in the finite-particle regime is not yet available.

\section*{Acknowledgments}
We acknowledge helpful conversations with Jonathan Goodman, Andrew Stuart and Jonathan Weare. This work is supported by National Science Foundation grant DMS-2608264. Generative AI tools, in particular GPT Sol 5.6, were used to test candidate Lyapunov functions for the continuous-time and regularized diffusions, to help formulate \cref{lem:uniformBoundUpdates}, and to polish the writing. The authors retain full responsibility for the mathematical correctness of this manuscript.

\begin{appendix}
\section{Supporting definitions} \label{appsec:defs}
\begin{definition}
    Let $\mathsf{X}$ be a locally compact separable metric space. A function $V: \mathsf{X} \rightarrow \R_+$ is \emph{norm-like} if the level sets $\{x \mid V(x) \le B\}$ are precompact for each $B>0$.

    Let $\{O_n \mid n \in \mathbb{N}\}$ be a fixed set of open precompact (compact closure) sets with $O_n \uparrow \mathsf{X}$ as $n \rightarrow \infty$. A process $\mathbf{\Phi}$ is \emph{non-explosive} if for all $x \in \mathsf{X}$, the exit times $\tau_n = \inf \{ t \mid \mathbf{\Phi} \text{ leaves } O_n\}$ satisfy $\mathbb{P}_x(\lim_{n \rightarrow \infty} \tau_n = \infty) = 1$ for all $x \in \mathsf{X}$.

    A process $\mathbf{\Phi}$ is \emph{non-evanescent} if $\mathbb{P}_x (\mathbf{\Phi} \rightarrow \infty) = 0$ for all $x \in \mathsf{X}$. Here $\mathbf{\Phi} \rightarrow \infty$ means that $\{\Phi_t \notin \mathrm{K}\}$ for any compact $\mathrm{K} \subset \mathsf{X}$ and all sufficiently large $t$.
\end{definition}

\section{Supporting proofs}
\subsection{The diffusions are Harris recurrent}\label{appsec:HarrisDiffusion}
In order to transfer positive recurrence (and Lebesgue-irreducibility) to ergodicity, one requires the additional requirement that the process is Harris. This requires a short additional argument since \cite[Thm. 2.1]{meyn1993stability} shows only non-explosiveness, while the stronger concept of non-evanescence is the equivalent condition for Harris recurrence. 
\begin{enumerate}
    \item Since the diffusion is hypoelliptic, it is Feller. Moreover, the path-connectedness shows that it is Lebesgue-irreducible \cite{kliemann1987recurrence}.
    \item This implies that for some skeleton chain, we have that all compact sets are petite \cite{meyn1992stability1}.
    \item Since we have shown that the affine invariant ensemble Langevin diffusion is \textit{positive recurrent}, it is in particular recurrent. Therefore, there exists a compact $\mathrm{K}' \subset \mathrm{M}$ such that starting from any $\rmX \in \mathrm{M}$, it will hit $\mathrm{K}'$ with probability 1. \cite[Thm. 3.1(iv)]{kliemann1987recurrence}
    \item Since $\mathrm{K}'$ is compact, it is a petite set with a.s. finite hitting time from any $\rmX$. This implies that the process is Harris recurrent \cite[Thm. 4.3]{meyn1993stability2}.
\end{enumerate}

\subsection{Proof of Proposition \ref{prop:leaveoneout}} \label{appsec:leaveOneOutIneqs}

First observe the equivalent representation
\begin{equation*}
    C_{-i}(\rmX) = \frac{N}{N-1} C(\rmX) - \frac{N}{(N-1)^2} (x_i - \bar x)^2.
\end{equation*}
Some manipulation yields the following two inequalities, detailed in the following subsection:
\begin{equation}\label{eq:leaveOneOutIneqs}
    C_{-i}(\rmX) \le \frac{N}{N-1} C(\rmX),\quad C_{-i} + C_{-j} \ge 2 \kappa C(\rmX), \quad i \ne j,\, \kappa \coloneqq \frac{N(N-2)}{2(N-1)^2}.
\end{equation}
In particular, at least $N-1$ of the covariances $C_{-i}$ are at least $\kappa C(\rmX)$. Let us condition on some fixed $\rmX$ and let $r = C^{1/2}(\rmX)$. Let $\rmY, \rmZ$ be the next two updates under \labelcref{eq:leaveoneout}. We will show that the following two-step event holds with high probability: for some constants $K>0$ depending only on $N, \eta$ that may change between lines,
\begin{equation}\label{eq:HighProbabilityLOO}
    \mathbb{P}(C^{1/2}(\rmY) \ge \sqrt{r},\, C^{1/2}(\rmZ) \ge 2r \mid \rmX) \ge 1 - K r^{-1/4}.
\end{equation}

Since $N \ge 3$, let $i \ne j$ index two particles satisfying $C_{-i}(\rmX), C_{-j}(\rmX) \ge \kappa r^2$. Then $y_i, y_j$ are two independent Gaussians with variance at least $2 \eta \kappa r^2$. From triangle inequality, we can bound the covariance in terms of the squared difference as $|y_i - y_j|^2 \le 2(y_i - \bar{y})^2 + 2(y_j - \bar y)^2 \le 2NC(\rmY)$. Therefore, using the variance lower bound on $y_i - y_j$,
\begin{align}
    \mathbb{P}(C^{1/2}(\rmY) \le \sqrt{r}) &\le \mathbb{P}(|y_i - y_j|^2 \le 2Nr) \notag\\
    &\le\frac{2\sqrt{2Nr}}{\sqrt{2\pi} \sqrt{4\eta \kappa r^2}}\le Kr^{-1/2}. \label{eq:c12sqrtrbound}
\end{align}

We now use the following elementary bound: if $U \sim \gN (m, \sigma^2)$ is a one-dimensional Gaussian, then a quadratic form of $U$ is bounded with probability
\begin{equation*}
    \mathbb{P}(|aU^2 + bU + c| \le u) \le \frac{2}{\sqrt{\pi |a| \sigma^2}} \sqrt{u}.
\end{equation*} 
A short proof of this bound is given in \Cref{appsec:leaveOneOutIneqs}, based on upper bounding the length of the admissible interval and the density of $U$. From the lower bound of the variance of particle $j$, we have
\begin{equation}
    \mathbb{P}(|y_j|^2 < r \mid \rmX) \le \frac{2}{\sqrt{\pi \eta \kappa r^2}}\sqrt{r} \le K r^{-1/2}.
\end{equation}

Now consider the complement event $|y_j|^2 \ge  r$. The difference in the means of the second update is
\begin{multline*}
    M_{ij}(\rmY) \coloneqq \mathbb{E}[z_i - z_j\mid \rmY] = \left(1 - \frac{\eta}{N-1}\sum_{k \ne i} (y_k - m_{-i}(\rmY))^2\right) y_i \\- \left(1 - \frac{\eta}{N-1}\sum_{k \ne j} (y_k - m_{-j}(\rmY))^2\right) y_j,
\end{multline*}
where $m_{-i}$ is the mean of all particles except the $i$'th. This is a quadratic in $y_i$, and the leading coefficient of $y_i^2$ is 
\begin{equation*}
    \frac{\eta}{N-1} \left[\frac{(N-2)}{(N-1)^2} + \left(1 - \frac{1}{N-1}\right)^2 \right]y_j = \eta\frac{N-2}{(N-1)^2} y_j.
\end{equation*}
Fixing the initial ensemble $\rmX$, let us now condition on $\rmY_{-i}$, the positions of all particles after one step except the $i$th. Conditional on $\rmX$, the first step updates are all independent. Therefore, conditioning on $\rmY_{-i}$ leaves the law of $y_i$ unchanged, i.e.
\begin{equation*}
    y_i \mid \rmX, \rmY_{-i} \sim \gN\left((1 - \eta C_{-i}(\rmX))x_i, 2 \eta C_{-i}(\rmX)\right).
\end{equation*}
Under this conditioning and on the event $|y_j|^2 \ge  r$, $M_{ij}(\rmY)$ is a quadratic in $y_i$, and the absolute value of the leading coefficient is at least $\eta(N-2)(N-1)^{-2} \sqrt{r}$. Since the variance of $y_i$ is $\sigma_i^2 \ge 2 \eta \kappa r^2$, applying the quadratic bound conditional on $\rmX, \rmY_{-i}$, then taking conditional expectations conditioned on $\rmX$ and this event gives
\begin{equation}\label{eq:r14}
    \mathbb{P}(|M_{ij}(Y)| < 2r^2 \mid  \rmX,\, |y_j|^2 \ge r) \le K r^{-1/4}.
\end{equation}

We now split into two different cases that prevent the second step $\rmZ$ from having a small variance. If the variance of $\rmY$ is large, then the noise of the second step is sufficiently large to give the desired small ball bound directly. If the variance is small, then the noise is unlikely to cancel out the large conditional mean $M_{ij}(\rmY)$.

Consider the empirical standard deviation $s = C^{1/2}(\rmY)$. Conditioned on $\rmY$, we have the representation
\begin{equation*}
    z_i - z_j = M_{ij}(\rmY) + W,\quad W\mid \rmY \sim \gN(0, 2 \eta (C_{-i}(\rmY) + C_{-j}(\rmY))).
\end{equation*}
Using the second inequality in \labelcref{eq:leaveOneOutIneqs}, conditioned on $\rmY$, $z_i - z_j$ has standard deviation at least $2\sqrt{\eta\kappa }s$. As it is Gaussian, if $s > r^{3/2}$, a uniform bound on the density gives
\begin{align}
    \mathbb{P}(C^{1/2}(\rmZ) < 2r  \mid \rmY, s > r^{3/2}) &\le \mathbb{P}(|z_i - z_j|^2 < 8N^2r^2  \mid \rmY, s > r^{3/2}) \notag\\
    &\le K \frac{r}{s} \le Kr^{-1/2}. \label{eq:c12r32ubound}
\end{align}
If $s \le r^{3/2}$, then from the first inequality in \labelcref{eq:leaveOneOutIneqs}, we get
\begin{equation*}
    \Var(W\mid \rmY) \le 4 \eta \frac{N}{N-1}r^3.
\end{equation*}
Therefore, $W$ is bounded with high probability:
\begin{equation}\label{eq:WBound}
    \mathbb{P}(|W| \le r^2 \mid \rmY) \ge 1-2e^{-cr}.
\end{equation}
On the event $\{|W| \le r^2\} \cap \{|M_{ij}(\rmY)| \ge 2r^2\}$, we have for sufficiently large $r$,
\begin{equation*}
    C^{1/2}(\rmZ) \ge \frac{|z_i - z_j|}{\sqrt{2N}} \ge \frac{r^2}{\sqrt{2N}} \ge 2r,
\end{equation*}
and therefore
\begin{equation*}
    \mathbb{P}(C^{1/2}(\rmZ) \ge 2r \mid s\le r^{3/2}, |W| \le r^2,\,|M_{ij}(\rmY)| \ge 2r^2) = 1.
\end{equation*}
The complement of the condition has probability bounded by \labelcref{eq:WBound,eq:r14}. We now show \labelcref{eq:HighProbabilityLOO} using the following chain of events
\begin{align*}
    &\quad \mathbb{P}(C^{1/2}(\rmY) \ge \sqrt{r},\, C^{1/2}(\rmZ) \ge 2r) & \\
    &\ge 1 -  \mathbb{P}(C^{1/2}(\rmY) \le \sqrt{r}) - \mathbb{P}(C^{1/2}(\rmY) \ge \sqrt{r},\, C^{1/2}(\rmZ) \le 2r)\\
    &\ge 1 - Kr^{-1/2} - \mathbb{P}(C^{1/2}(\rmY) > r^{3/2},\, C^{1/2}(\rmZ) \le 2r) - \mathbb{P}(C^{1/2}(\rmY) \le r^{3/2},\, C^{1/2}(\rmZ) \le 2r)\\
    &\ge 1 - Kr^{-1/2} - Kr^{-1/2} - (2e^{-cr} + Kr^{-1/4})\\ 
    &\ge 1 - Kr^{-1/4}
\end{align*}
as desired. The second inequality uses \labelcref{eq:c12sqrtrbound}, and the third inequality uses \labelcref{eq:c12r32ubound} and the above discussion. We can now apply a union bound for sufficiently large initial $r$. Induction gives that the following recursion holds
\begin{equation*}
    C^{1/2}(\rmX^{(2k)}) \ge 2^k r_0,\quad C^{1/2}(\rmX^{(2k+1)}) \ge \sqrt{C^{1/2}(\rmX^{(2k)})} \ge \sqrt{2^k r_0}
\end{equation*}
with probability at least 
\begin{equation*}
    1 - \sum_{k=0}^\infty K(2^kr_0)^{-1/4} = 1 - \frac{Kr_0^{-1/4}}{1 - 2^{-1/4}}.
\end{equation*}
Choosing a sufficiently large $r_0 = C^{1/2}(\rmX^{(0)})$ yields a positive probability that the covariance tends to infinity. Thus the discretized leave-one-out scheme \labelcref{eq:leaveoneout} also diverges with positive probability.

\subsubsection{Supporting proofs}
\textit{Upper bound for quadratic forms of Gaussians}. We stated that if $U \sim \gN (m, \sigma^2)$ is a one-dimensional Gaussian, then
\begin{equation*}
    \mathbb{P}(|aU^2 + bU + c| \le u) \le \frac{2}{\sqrt{\pi |a| \sigma^2}} \sqrt{u}.
\end{equation*} 
We claim that for $a \ne 0$, the length of the sublevel set $\{t \in \R \mid |at^2 + bt + c| \le u\}$ is bounded by $2\sqrt{2u}/\sqrt{|a|}$. Completing the square, without loss of generality $b=0$ and $a>0$, and it will be convenient to replace $c$ with $-c$. The interval becomes
\begin{equation*}
    |at^2 - c| \le u \quad \Leftrightarrow \quad c-u \le at^2 \le c+u.
\end{equation*}

\textit{Case 1.} $-u < c < u$. The feasible set is $at^2 \le c+u$. The claim holds.

\textit{Case 2.} $c \ge u$. The desired intervals have total length $2(\sqrt{c+u} - \sqrt{c-u})/\sqrt{|a|} \le 2\sqrt{2u/|a|}$, taking supremum over $c \ge u$. 

\textit{Case 3.} For $c \le -u$, the bound holds trivially.

The desired inequality follows using the density upper bound $p_U \le 1/\sqrt{2\pi \sigma^2}$.

\textit{Representations of the leave-one-out covariances.} The first inequality in \labelcref{eq:leaveOneOutIneqs} follows directly from the preceding representation. For $i \ne j$, we have
\begin{equation*}
    C_{-i}(\rmX) + C_{-j}(\rmX) = \frac{2N}{N-1} C(\rmX) - \frac{N}{(N-1)^2} (\delta_i ^2 + \delta_j^2),
\end{equation*}
where $\delta_i = x_i - \bar x$ and similar for $\delta_j$. Since
\begin{equation*}
    \delta_i^2 + \delta_j^2 \le \sum_{k} \delta_k^2 = NC(\rmX),
\end{equation*}
we obtain the desired inequality
\begin{equation*}
    C_{-i}(\rmX) + C_{-j}(\rmX) \ge \left(\frac{2N}{N-1} - \frac{N^2}{(N-1)^2}\right) C(\rmX) = \frac{N(N-2)}{(N-1)^2} C(\rmX).
\end{equation*}

\subsection{Proof of Lemma \ref{lem:traceIneqs}}\label{appsec:traceIneqs}
Here and below, we restate the result before proving it. 
\begin{lemma*}
    Assume $\|\nabla^2f\| \le L$. The following trace inequalities hold.
    \begin{align*}
        |\Tr K| &\le LT,  &|\Tr (KC)| &\le LT^2, &|\Tr(KC^2)| &\le LT^3,\\
        |\Tr (KC^{-1})| &\le L\sqrt{TJ},&TJ&\ge d^2 , &J^{-1} |\Tr(KC^{-1})|&\le \frac{LT}{d}. 
    \end{align*}
\end{lemma*}
\begin{proof}
    \begin{align*}
        |\Tr K| &= \frac{1}{N} |\sum_i (\rvx_i - \bar \rvx)^\top (\rvg_i - \rvg(\bar \rvx))| \le \frac{1}{N}\sum_i \|\rvx_i - \bar \rvx\| \|\rvg_i - \rvg(\bar \rvx)\|\\
        &\le \frac{L}{N} \sum_i \|\rvx_i - \bar \rvx\|^2 = LT.
   \end{align*}
    The bounds on $\Tr(KC)$ and $\Tr(KC^2)$ follow from H\"older's inequality on matrix norms.

    To bound $\Tr(KC^{-1})$, let $\gamma_i:[0,1] \rightarrow \R^d$ be the straight line from $\bar \rvx$ to $\rvx_i$. Then
    \begin{align*}
        \rvg_i - \rvg(\bar \rvx) &= \int_0^1 \nabla^2 f(\gamma_i(s)) \dd{\gamma_i(s)}\\
        &= \bar H_i (\rvx_i - \bar \rvx),\quad  \bar H_i \coloneqq \int_0^1 \nabla^2 f(\gamma_i(s)) \dd{s}.
    \end{align*}
    Here $H_i$ are some symmetric matrices satisfying $\|H_i\| \le L$. Let $\boldsymbol{\delta}_i = \rvx_i - \bar \rvx$. The trace can thus be written as 
    \begin{align*}
        \Tr(KC^{-1}) = \frac{1}{N} \sum_i \boldsymbol{\delta}_i^\top \bar H_i C^{-1} \boldsymbol{\delta}_i.
    \end{align*}
    Taking absolute values,
    \begin{align*}
        \left|\Tr(KC^{-1})\right| &\le \frac{1}{N} \sum_i \|\bar H_i \boldsymbol{\delta}_i\| \|C^{-1} \boldsymbol{\delta}_i\|\\
        &\le \left(\frac{1}{N} \sum \|\bar H_i \boldsymbol{\delta}_i\|^2\right)^{1/2} \left(\frac{1}{N} \sum \| C^{-1} \boldsymbol{\delta}_i\|^2\right)^{1/2}\\
        &\le \left(L^2 \frac{1}{N}\sum_i \boldsymbol{\delta}_i^\top \boldsymbol{\delta}_i\right)^{1/2} \left(\frac{1}{N} \sum \boldsymbol{\delta}_i C^{-2} \boldsymbol{\delta}_i\right)^{1/2} = L \sqrt{\Tr(C) \Tr(C^{-1})},
    \end{align*}
    as desired. The fact that $TJ \ge d^2$ comes from using Cauchy--Schwarz on the Frobenius inner product $\langle A,B\rangle = \Tr(AB^\top)$,
    \begin{align*}
        d = \Tr I = \Tr C^{1/2} C^{-1/2} \le \|C^{1/2}\|_F \|C^{-1/2}\|_F = \sqrt{\Tr(C) \Tr(C^{-1})}.
    \end{align*}

    The final inequality comes from combining the previous two.
\end{proof}

\subsection{Proof of Proposition \ref{prop:DistantConvexity}}\label{appsec:propDistConvex}

\begin{proposition*}
    Suppose that $f$ satisfies \Cref{assmp:distantConvexityND}. There exist constants $c_1, b_1, c_2, b_2$, $C_T, B_T>0$ depending only on $f$ and $d$, such that for all $x \in \R^d$ or ensemble $\rmX \in \R^{N \times d}$,
    \begin{align*}
        f(x) &\ge c_1 \|x\|^2 - b_1,\\
        \|\nabla f(x) \|^2 &\ge c_2 f(x) - b_2,\\
        D &\ge C_T T^2 - B_T + \bar \rvg^\top C \bar \rvg.
    \end{align*}
    Moreover, for any $J_0>0$, the following holds on the set $\{J \le J_0\}$:
    \begin{align*}
        D \ge \frac{c_2}{J_0} E - \frac{b_2}{J_0}.
    \end{align*}
\end{proposition*}
\begin{proof}
Let $R>0$ be sufficiently large such that $\nabla^2 f \succeq \mu I_d$ on $\|x\| \ge R$.
\begin{enumerate}[i.]
    \item Given $\|x\| \ge R$, strong convexity between $x$ and $y = Rx/\|x\|$ yields
    \begin{align*}
        f(x) &\ge f(y)  + (x-y)^\top \nabla f(y) + \frac{\mu}{2} (\|x\|-R)^2\\
        &\ge \inf_{\|z\| = R} f(z) - (\|x\|+R)\sup_{\|z\| = R} \|\nabla f(z)\| + \frac{\mu}{2} (\|x\|-R)^2.
    \end{align*}
    Minor rearranging yields $f(x) \ge c_1 \|x\|^2 - b_1$ on $\|x\| \ge R$ for positive $c_1$; increasing $b_1$ makes this hold over all of $\R^d$.
    \item Given $x,y$ as above, FTC yields
    \begin{align*}
        y^\top \nabla f(x) - y^\top \nabla f(y) = \frac{1}{R}\int_{R}^{\|x\|} y^\top \nabla^2 f(tx/\|x\|)y\dd{t} \ge \mu(\|x\|-R) R.
    \end{align*}
    Since $\|y\|=R$, Cauchy--Schwarz gives
    \begin{align*}
        \|\nabla f(x)\| \ge \frac{1}{R}y^\top \nabla f(x) \ge \frac{1}{R}\inf_{\|z\|=R} z^\top \nabla f(z) + \mu (\|x\|-R)  \ge \mu \|x\| - c,
    \end{align*}
    for some constant $c$. Furthermore, the upper bound on $\|\nabla^2 f\|$ gives $f(x) \lesssim 1+\|x\|^2$. Combining the two inequalities concludes $\|\nabla f\|^2 \ge c_2 f(x) - b_2$ for some positive $c_2$.
    \item Applying Cauchy--Schwarz to $K$ yields
    \begin{align*}
        |\Tr K| &\le \left(\frac{1}{N}\sum_j (\rvx_j - \bar \rvx)^\top C^{-1} (\rvx_j - \bar \rvx)\right)^{1/2} \left(\frac{1}{N} \sum_j (\rvg_j - \bar \rvg)^\top C (\rvg_j - \bar \rvg)\right)^{1/2} \\ &= \sqrt{d} \sqrt{D - \bar \rvg^\top C \bar \rvg}.
    \end{align*}
    It suffices to show that $\Tr K \ge c T - b$ for some constants $c>0, b \in \R$. This follows from monotonicity: given $x \ne y \in \R^d$, set $v=x-y$. Define the interval
    \begin{equation*}
     I_{x,y}\coloneqq \{t\in[0,1] \mid \|y+tv\|<R\},
     \qquad \ell_{x,y} \coloneqq |I_{x,y}|.
    \end{equation*}
    The intersection of a line segment with the ball $B_R$ has length at
    most $2R$, so $\ell_{x,y}|v|\le2R$.  Since
    $\nabla^2f\succeq-LI_d$ everywhere and $\nabla^2f\succeq\mu I_d$ outside $B_R$,
    \begin{align*}
     v^\top (\nabla f(x)-\nabla f(y))
     &=\int_0^1 v^\top\nabla^2f(y+tv)v\dd{t}\\
     &\ge \mu|v|^2-(\mu+L)\ell_{x,y}|v|^2\\
     &\ge \mu|v|^2-2R(\mu+L)|v|\\
     &\ge \frac{\mu}{2}|v|^2-b
    \end{align*}
    for some finite $b$. We conclude $\Tr K \ge c T - b$ using
    \begin{equation*}
        \Tr K = \frac{1}{2N^2} \sum_{i,j} (\rvx_i - \rvx_j)^\top (\nabla f(\rvx_i) - \nabla f(\rvx_j)).
    \end{equation*}
    with $v = \rvx_i - \rvx_j$ and summing.
    
    \item This follows from summing (ii) and the lower bound $\rvg^\top C \rvg \ge J_0^{-1}\|\rvg\|^2$.
\end{enumerate}
\end{proof}

\subsection{Proof of Lemma \ref{lem:GeneratorJ}}\label{appsec:GeneratorJ}
This result is a routine computation.
\begin{lemma*}
    The generator applied to $J = \Tr(C^{-1})$ is 
    \begin{equation*}
        \gL J = 2 \Tr(KC^{-1}) + \left[-2 + \frac{2d+4}{N}\right] \Tr(C^{-1}).
    \end{equation*}
\end{lemma*}
\begin{proof}
    Define $\boldsymbol{\delta}_i = \rvx_i - \bar \rvx \in \R^d$ so that $C = \frac{1}{N} \sum_i \boldsymbol{\delta}_i \boldsymbol{\delta}_i^\top$. 
    We first note that the first order gradient is
    \begin{equation}
        \nabla_j \Tr(C^{-1}) = -\frac{2}{N} C^{-2} \boldsymbol{\delta}_j.
    \end{equation}
    The drift component of $\gL J$ is therefore
    \begin{align*}
        &\quad \sum_j \left[-C \rvg_j + \frac{d+1}{N} \boldsymbol{\delta}_j\right] \cdot \left(-\frac{2}{N} C^{-2} \boldsymbol{\delta}_j\right)\\
        &= \sum_j \frac{2}{N} \rvg_j^\top C C^{-2} 
        \boldsymbol{\delta}_{j} - \frac{2(d+1)}{N^2} \boldsymbol{\delta}_j^{\top} C^{-2} \boldsymbol{\delta}_{j}\\
        &= 2\Tr(KC^{-1}) - \frac{2(d+1)}{N} \Tr(C^{-1}).
    \end{align*}
    For the diffusion component, we need to compute $C \!:\!\nabla_j^2 \Tr(C^{-1})$. We have
    \begin{align*}
        \partial_{x_{j,p}}\partial_{x_{j,q}} \Tr(C^{-1}) &= \partial_{x_{j,p}} \left[-\frac{2}{N}\rve_q^\top C^{-2} \boldsymbol{\delta}_{j}\right]\\
        &= -\frac{2}{N}\rve_q^\top \left[\partial_{x_{j,p}}(C^{-2}) \boldsymbol{\delta}_{j} + C^{-2} \partial_{x_{j,p}} \boldsymbol{\delta}_{j}\right]\\
        &= -\frac{2}{N} \rve_q^\top \left[\partial_{x_{j,p}}(C^{-2}) \boldsymbol{\delta}_{j} + (1-\frac{1}{N})C^{-2} \rve_p\right].
    \end{align*}
    Here $\rve_p, \rve_q$ denote canonical basis vectors. Let us denote the  covariance derivative
    \begin{align*}
        \partial_{x_{j,q}} C =  \frac{1}{N}[\rve_q \boldsymbol{\delta}_{j}^\top + \boldsymbol{\delta}_{j} \rve_q^\top] \eqqcolon D_q^{(j)}.
    \end{align*}
    The inverse matrix derivative is thus 
    \begin{equation*}
        \partial_{x_{j,p}}(C^{-2})=-C^{-1}D_p^{(j)}C^{-2}-C^{-2}D_p^{(j)}C^{-1}.
    \end{equation*}
    For a fixed $j$,
    \begin{align*}
        C\!:\!\nabla_j^2 \Tr(C^{-1}) &= -\frac{2}{N} \sum_{pq} C_{p,q} \rve_q^\top
        \left[\partial_{x_{j,p}}(C^{-2}) \boldsymbol{\delta}_{j} + (1-\frac{1}{N})C^{-2} \rve_p\right]\\
        &= -\frac{2}{N}(1-\frac{1}{N})\sum C_{pq} (C^{-2})_{pq} - \frac{2}{N} \sum C_{pq} \rve_q^\top \left(-C^{-1} D_p^{(j)} C^{-2} - C^{-2} D_p^{(j)}C^{-1}\right)\boldsymbol{\delta}_{j}\\
        &= -\frac{2}{N}(1-\frac{1}{N}) \Tr(C^{-1}) + \frac{2}{N} \underbrace{\sum_q \rve_q^\top C \left(C^{-1}D_q^{(j)} C^{-2} + C^{-2} D_q^{(j)} C^{-1}\right) \boldsymbol{\delta}_{j}}_{{\eqqcolon \Xi}}
    \end{align*}
    The final term $\Xi$ can be computed using the identities $\sum_q \rve_q^\top \rve_q = d$ and $\sum_q \rve_q^\top \rva \rve_q^\top \rvb = \rva^\top \rvb$ for vectors $\rva,\rvb \in \R^d$,
    \begin{align*}
        \Xi &= \sum_q \rve_q^\top \left(D_q^{(j)} C^{-2} + C^{-1}D_q^{(j)}C^{-1}\right) \boldsymbol{\delta}_{j}\\
        &= \frac{1}{N}\sum_q \rve_q^\top\left([\rve_q \boldsymbol{\delta}_{j}^{\top} + \boldsymbol{\delta}_{j}\rve_q^\top]C^{-2} + C^{-1}[\rve_q \boldsymbol{\delta}_{j}^{\top} + \boldsymbol{\delta}_{j}\rve_q^\top]C^{-1}\right) \boldsymbol{\delta}_{j}\\
        &= \frac{d}{N} \boldsymbol{\delta}_{j}^{\top} C^{-2} \boldsymbol{\delta}_{j} + \frac{1}{N}\boldsymbol{\delta}_{j}^{\top} C^{-2} \boldsymbol{\delta}_{j} +  \frac{1}{N} \sum_q \rve_q^\top C^{-1} \rve_q \boldsymbol{\delta}_{j}^{\top } C^{-1} \boldsymbol{\delta}_{j} + \frac{1}{N}\boldsymbol{\delta}_{j}^{\top} C^{-2} \boldsymbol{\delta}_{j}\\
        &= \frac{d+2}{N} \boldsymbol{\delta}_{j}^{\top} C^{-2} \boldsymbol{\delta}_{j} + \Tr(C^{-1}) \frac{1}{N} \boldsymbol{\delta}_{j}^{\top} C^{-1} \boldsymbol{\delta}_{j}.
    \end{align*}
    Substituting back in and summing over $j$, the diffusion component gives
    \begin{align*}
        \sum_j C \!:\! \nabla_j^2 \Tr(C^{-1}) &= -2(1-\frac{1}{N}) \Tr(C^{-1}) + \frac{2(d+2)}{N^2} \sum_j \boldsymbol{\delta}_{j}^{\top} C^{-2} \boldsymbol{\delta}_{j} + \frac{2}{N^2}\Tr(C^{-1}) \sum_j  \boldsymbol{\delta}_{j}^{\top} C^{-1} \boldsymbol{\delta}_{j}\\
        &= -2(1-\frac{1}{N})\Tr(C^{-1}) + \frac{2(d+2)}{N}\Tr(C^{-1}) + \frac{2}{N}d\Tr(C^{-1})\\
        &= \left(-2 + \frac{4d+6}{N}\right)\Tr(C^{-1}).
    \end{align*}
    Therefore the generator applied to this function gives
    \begin{align}
        \gL J &= 2\Tr(KC^{-1}) - \frac{2(d+1)}{N}J  + \left(-2 + \frac{4d+6}{N}\right)J \notag\\
        &= 2\Tr(KC^{-1}) + \left(-2+\frac{2d+4}{N}\right)J. 
    \end{align}
\end{proof}

\subsection{Proof of Lemma \ref{lem:uniformBoundUpdates}}\label{app:uniformBoundUpdates}
We will prove a slightly generalized statement with an additional power, which will be used to prove the $N = d+2$ statement. \Cref{lem:uniformBoundUpdates} is a special case of the below statement with $\kappa = 1$.
\begin{lemma}\label{applem:inverseMatBound}
    Let $\kappa \in (0,1]$. Suppose that $n = N-1 \ge d+1$, and $n \ge d+2$ if $\kappa = 1$, and consider the compact subset
    \begin{equation*}
        \gZ = \left\{Z \in \R^{N \times d} \mid \mathbf{1}_N^\top Z = 0,\, Z^\top Z = N I_d\right\}.
    \end{equation*}
    Let $M < \infty$ , and define $ \gU \coloneqq \{U \in \R^{N \times d} \mid P_c U = U,\,\|U\|_F \le M\} $ as the set of bounded centered matrices, and further let $\gamma \in [0,1]$. Let $\gamma I_d \preceq R \preceq 2\gamma I_d$ be some symmetric positive semidefinite matrix, and $\Xi \in \gN(0, I_{N \times d})$ have independent standard Gaussian entries. Define the random matrices
    \begin{equation}
        \Lambda_\eta = Z + \eta U + \sqrt{2\eta} P_c \Xi R^{1/2},\quad H_\eta \coloneqq \frac{1}{N} \Lambda_\eta^\top \Lambda_\eta .
    \end{equation}
    Then the following first order expansion holds, where the remainder is uniformly bounded over $Z \in \gZ$, $U \in \gU$, $\gamma \in [0,1]$, $\gamma I_d\preceq R \preceq 2\gamma I_d$, and symmetric $A \succeq 0$ with $\Tr A = 1$ as $\eta \rightarrow 0$,
    \begin{align}
        \mathbb E\left[\Tr(AH_\eta^{-1})^\kappa\right]
    ={}&1+\eta\left[\kappa\Tr(AB_{-1})
    -\frac{4\kappa(1-\kappa)}N\Tr(RA^2)\right]
    +\mathcal O(\eta^{3/2}),\\
    B_{-1}
    ={}&-\frac1N(Z^\top U+U^\top Z)
    +\frac2N\left[(d+3-N)R+(\Tr R)I_d\right].
    \end{align}
\end{lemma}
\begin{proof}
    All the nonzero singular values of $Z$ are $\sqrt{N}$. Since $U$ is bounded, for all sufficiently small $\eta$ (depending on $M$), all singular values of $Z + \eta U$ are bounded in $(\sqrt{N}/2, 3\sqrt{N}/2)$. Without loss of generality, we work in the $n$-dimensional subspace given by $\Im P_c$, and assume that $R = \diag(r_1,...,r_d)$ is diagonal so that all columns of $\Lambda_\eta$ are independent. Letting $M_\eta = Z + \eta U$, we have
    \begin{equation}
        \Lambda_\eta = M_\eta + \sqrt{2\eta} \Xi R^{1/2}
    \end{equation}
    where $\Xi$ is an $n \times d$ standard Gaussian matrix. Since the singular values of $M_\eta$ are at least $\sqrt{N}/2$, we have that for any of its columns $m_k$,
    \begin{equation}
        \dist(m_k, \mathrm{span} \langle m_l \mid l \ne k\rangle) \ge \sqrt{N}/2.
    \end{equation}

    We first show that a higher moment of $H_\eta^{-1}$ is bounded in order to use H\"older's inequality to bound the remainder term. Consider the elementary identity for a nonnegative real-valued random variable $Z$
    \begin{align*}
        \mathbb{E}[Z] = \int_0^\infty \mathbb{P}(Z \ge s) \dd{s}.
    \end{align*}
    We will apply this to the minimum singular value $Z = \sigma_{\min}(\Lambda_{\eta})^{-2p}$ for some appropriate $p>0$, using the small ball estimate $\mathbb{P}(\sigma_{\min}(\Lambda_{\eta}) \le s) \lesssim s^{N-d}$.
    
    Observe that if $\xi \sim \gN(\mu, \sigma^2 I_m)$, then for $s>0$, letting $\kappa_m$ be the volume of the unit ball in $m$ dimensions divided by $(2\pi)^{m/2}$,
    \begin{align*}
        \mathbb{P}(\|\xi\| \le s) \le \kappa_m s^m \frac{1}{\sigma^m} \exp(-\frac{1}{2\sigma^2} \max(\|\mu\|-s,0)^2).
    \end{align*}
     Additionally note the elementary inequality
    \begin{equation}\label{eq:sigmaExpBound}
        \sup_{\sigma>0} \sigma^{-m} e^{-\alpha/\sigma^2} = \left(\frac{m}{2 e \alpha}\right)^{m/2}.
    \end{equation}
    
    Let $k \in \{1,...,d\}$ and condition on all other columns. Let $P_k$ be the projection onto the orthogonal complement of the column span $V_k \coloneqq \mathrm{span}\langle \lambda_l\mid l \ne k\rangle^\perp \subset \R^n$ of dimension $d_k \ge n - d + 1 = N-d$. Since $R$ is diagonal,
    \begin{equation*}
        \mathrm{dist}(\lambda_k,  \langle \lambda_l\mid l \ne k\rangle) = \|P_k \lambda_k\|,\quad P_k \lambda_k \sim \gN(P_k m_k, 2 \eta r_k I_{V_k} ).
    \end{equation*}

    To show that $P_k m_k$ is bounded from below, consider the event $F_k = \{r_l^{1/2} \|\xi_l\| \le \sqrt{N}/(4\sqrt{2d \eta}) \mid \forall l \ne k\}$. On this event, we have that for any $c \in \R^{d-1}$,
    \begin{align*}
        \|m_k - \sum_{l \ne k} c_l \lambda_l\| &\ge \|m_k - \sum_{l \ne k} c_l m_l \| - \sum_{l \ne k} \sqrt{N}/(4\sqrt{d}) |c_l| \\
        &\ge \frac{\sqrt N}{2} \sqrt{1 + \|c\|^2} -\frac{\sqrt N}{4} \|c\| \\
        &\ge \frac{\sqrt N}{4}
    \end{align*}
    where the second inequality uses $\sigma_{\min}(M_\eta) \ge \sqrt{N}/2$ and Cauchy--Schwarz. 

    \begin{enumerate}
        \item On $F_k$, we showed that $\|P_k m_k\| \ge \sqrt{N}/4$. Applying \labelcref{eq:sigmaExpBound} with $\alpha = N/128$, we have that for $s < \min(1,\sqrt{N}/8)$,
        \begin{align*}
            \mathbb{P}(\|P_k \lambda_k\| \le s \mid F_k) \le C s^{d_k} \le C s^{N-d},
        \end{align*}
        for some constant $C$ depending on $N,d$. 
        \item On $F_k^c$, which has probability bounded by
        \begin{align*}
            \mathbb{P}(F_k^c) \le \exp(-c/(\eta\gamma )),
        \end{align*}
        for some constant $c>0$, we have that for $s \in (0,1)$,
        \begin{align*}
             \mathbb{P}(\|P_k \lambda_k\| \le s , F_k^c) &\le \kappa_{d_k} \left(\frac{s}{\sqrt{2 \eta r_k}}\right)^{d_k} \mathbb{P}(F_k^c) \\
             &\le \kappa_{d_k} s^{d_k} (\sqrt{2 \eta \gamma})^{-d_k} \exp(-c/(\eta \gamma)) \le C s^{d_k}.
        \end{align*}
        Here we used the lower and upper bounds for $R$.
    \end{enumerate}
    Summing, we therefore have that for all $s < \min(1,\sqrt{N}/8)$,
    \begin{align*}
        \mathbb{P}(\|P_k \lambda_k\| \le s) \le Cs^{N-d}.
    \end{align*}
    To relate this to the minimum singular value, $\sigma_{\min}(\Lambda_{\eta}) \le s$ implies that for some column $k$,
    \begin{align*}
        \dist(\lambda_k, \mathrm{span}\langle \lambda_l \mid l \ne k\rangle)  = \|P_k \lambda_k\| \le \sqrt{d} s.
    \end{align*}
    This can be seen by taking any vector $v \in \R^d$ satisfying $\|v\| = 1$ and $\|\Lambda_\eta v\| \le s$, and considering the column corresponding to $|v_k| \ge d^{-1/2}$. Therefore, for any $s < \min(1,\sqrt{N}/8)$,
    \begin{align}
        \mathbb{P}(\sigma_{\min}(\Lambda_\eta) \le d^{-1/2} s) &\le \mathbb{P}(\|P_k \lambda_k\|\le s \text{ for some }k)\le C' s^{N-d} \label{eq:sigmaMinBound}
    \end{align}
    using a union bound over all $k \in [d]$. To conclude,
    \begin{align*}
        \mathbb{E}[\|H_\eta^{-1}\|^p] &= N^p \mathbb{E}[\sigma_{\min}(\Lambda_\eta)^{-2p}]\\
        &= N^p \int_0^\infty \mathbb{P}( \sigma_{\min}(\Lambda_\eta) \le s^{-\frac{1}{2p}})\dd{s}\\
        &\lesssim 1 + \int_{(\sqrt{d}/\min(1, \sqrt{N}/8))^{2p}}^\infty s^{-\frac{N-d}{2p}} \dd{s}
    \end{align*}
    using \labelcref{eq:sigmaMinBound}. Therefore, $\mathbb{E}[\|H_\eta^{-1}\|^p]$ is finite and uniformly bounded if $N-d > 2p$.

    To show the uniform remainder, let $\Delta_\eta = H_\eta - I$. Writing $\Gamma = P_c \Xi R^{1/2}$,

    \begin{align}
        N \Delta_\eta &= \sqrt{2\eta}(\Gamma^\top Z + Z^\top \Gamma) + \eta (Z^\top U + U^\top Z + 2\Gamma^\top \Gamma) \notag\\
        &\quad + \sqrt{2} \eta^{3/2} (U^\top \Gamma + \Gamma^\top U) + \eta^2 U^\top U. \label{appeq:DeltaEta}
    \end{align}
    For $\{\|\Delta_\eta\| \le 1/2\}$, the uniform expansion holds
    \begin{align*}
        H_\eta^{-1} = I_d - \Delta_\eta + \Delta^2_\eta + \mathcal{O}(\|\Delta_\eta\|^3).
    \end{align*}
    The expectation of the order $\sqrt{\eta}$ term is zero, and the expectation of the $\eta$ term is 
    \begin{align*}
       \mathbb{E}[-\Delta_\eta + \Delta_\eta^2] &= \eta\left[\frac{2}{N^2}\mathbb{E} \left[(\Gamma^\top Z + Z^\top \Gamma)^2\right] - \frac{1}{N}(Z^\top U + U^\top Z + 2 \mathbb{E}[\Gamma^\top \Gamma])\right] + \mathcal{O}(\eta^{3/2})\\
        &= - \frac{1}{N}(Z^\top U + U^\top Z)\eta + \frac{2}{N}\left[(d+2-n) R + (\Tr R) I_d\right]\eta \\
        &= \eta B_{-1} + \mathcal{O}(\eta^{3/2})
    \end{align*}
    leaving a $\mathcal{O}(\eta^{3/2})$ remainder. Furthermore, since $A\succeq0$ and $\Tr A=1$, the following Taylor expansion on $\{\|\Delta_\eta\| \le 1/2\}$ is uniform:
    \begin{align}
        \left(\Tr(AH_\eta^{-1})\right)^\kappa
        ={}&1+\kappa\Tr\!\left(A(-\Delta_\eta+\Delta_\eta^2)\right)
        +\frac{\kappa(\kappa-1)}2\left(\Tr(A\Delta_\eta)\right)^2
        +\mathcal O(\|\Delta_\eta\|^3).\label{eq:scalarFractionalTaylor}
    \end{align}
    The second term is given by the expansion of $H_{\eta}^{-1} - I_d$ up to the remainder order. To deal with the third term, we need only the order $\sqrt{\eta}$ term of $\Delta_{\eta}$, yielding
    \begin{align*}
        \mathbb{E} \left[\Tr(A\Delta_\eta)^2\right] &= 2\eta N^{-2} \mathbb{E} \left[\Tr(A\Gamma^\top Z + AZ^\top \Gamma)^2\right]+ \mathcal{O}(\eta^{3/2})\\
        &= 8\eta N^{-2} \mathbb{E} \left[\Tr((P_c \Xi)^\top  ZAR^{1/2})^2\right] + \mathcal{O}(\eta^{3/2})\\
        &= 8 \eta N^{-2} \|ZAR^{1/2}\|_F^2 + \mathcal{O}(\eta^{3/2}) \\
        &= 8 \eta N^{-1} \Tr(RA^2) + \mathcal{O}(\eta^{3/2}).
    \end{align*}
    Combining everything together, we get
    \begin{align*}
        \mathbb E\left[\left(\Tr(AH_\eta^{-1})\right)^\kappa\mathbf1_{\{\|\Delta_\eta\| \le 1/2\}}\right]
        =1+\eta\left[\kappa\Tr(AB_{-1})
        -\frac{4\kappa(1-\kappa)}N\Tr(RA^2)\right]
        +\mathcal O(\eta^{3/2}).
    \end{align*}
    
    On the small complement event $\{\|\Delta_\eta\| \ge 1/2\}$, we can use the moment bound for any admissible $q>1$ satisfying $\kappa q < \frac{N-d}{2}$, which exists by the first assumptions,
    \begin{align*}
        \mathbb{E}[\|H^{-1}_\eta\|^\kappa \mathbf{1}_{\{\|\Delta_\eta\| \ge 1/2\}}] &\le (\mathbb{E}\|H^{-1}_\eta\|^{\kappa q})^{1/q} \mathbb{P}(\|\Delta_\eta\| \ge 1/2)^{1-1/q}\\
        &\lesssim \exp(-c(1-1/q)/\eta).
    \end{align*}
    H\"older completes the upper bound on $\mathbb E\left[\left(\Tr(AH_\eta^{-1})\right)^\kappa\mathbf1_{\{\|\Delta_\eta\| \ge 1/2\}}\right]$.
\end{proof}

\section{Proof of Remark \ref{rem:polyTailsBound}}\label{appsec:polyTailsBound}
We now restate the assumptions. Note that $\ell=0$ gives back Assumption \ref{assmp:distantConvexityND} and the result reduces to \Cref{thm:ctsGeoErgo}.
\begin{theorem}
    Suppose that $f$ satisfies the following: $f \in \gC^2 \cap L^1(\pi)$, and further that there exists a compact set $\mathrm{K} \subset \R^d$ and constants $0 < c_1 < c_2,\, \ell \ge 0$ such that for all $x \in \R^d \setminus \mathrm{K}$,
    \begin{subequations}
        \begin{align}
            &c_1 \|x\|^{\ell+2} \le f(x) \le c_2 \|x\|^{\ell+2},\\
            &c_1 \|x\|^{\ell+1} \le \|\nabla f(x)\| \le c_2 \|x\|^{\ell+1},\label{appeq:dfbound}\\
            & c_1 \|x\|^\ell I_d \preceq \nabla^2 f(x) \preceq c_2 \|x\|^\ell I_d.
        \end{align}
    \end{subequations}
    Consider the affine invariant Langevin dynamics \labelcref{eq:aldiDynamics}. If $N \ge d+3$, then for sufficiently small $a>0$, the Lyapunov function \labelcref{eq:LyapunovWa} satisfies the Foster--Lyapunov condition
    \begin{equation*}
        \gL W_a \le -c W_a + b \mathbf{1}_\mathrm{K}
    \end{equation*}
    for some constants $c>0,\, b \in \R$ and compact $\mathrm{K} \subset \mathrm{M}$ depending on $a$. 
\end{theorem}
\begin{proof}
    Recall \labelcref{eq:LWaoverWa}: for $\lambda = 2-(2d+4)/N$ and $c_a \coloneqq a(1-a/N)$, and where $P$ is defined in \labelcref{eq:PDef}
    \begin{equation*}
        \frac{\gL W_a}{W_a} = -\lambda -c_a D + aP + \left(2-\frac{4a}{N}\right) \frac{\Tr(KC^{-1})}{J}.
    \end{equation*}
    We now show some modified bounds on $\Tr K,\, D$, and $\Tr(KC^{-1})$. From here onwards, $c,c'$ will denote some positive constants that may be different in each instance. Let us define a scale factor
    \begin{equation}
        S \coloneqq \frac{1}{N} \sum_i \|x_i\|^2 = T + \|\bar x\|^2.
    \end{equation}
    Furthermore, by compactness, let $b_1$ be some constant such that for all $x \in \R^d$,
    \begin{equation}
        (c_1 \|x\|^{\ell} - b_1) I_d \preceq \nabla^2 f (x) \preceq (c_2 \|x\|^{\ell} + b_1) I_d.
    \end{equation}
    We now lower bound $\Tr K$: we have for any $x,y \in \R^d$,
    \begin{align*}
        \int_0^1 \|(1-t)y + tx\|^\ell \dd{t} &\ge \int_0^1 |\|y\| - (\|x\|+\|y\|)t|^\ell \dd{t} \\
        &= \frac{\|x\|^{\ell+1}  + \|y\|^{\ell+1}}{(\ell+1)(\|x\|+\|y\|)}\\
        &\ge \frac{1}{(\ell+1)2^\ell} (\|x\|+\|y\|)^\ell,
    \end{align*}
    using reverse triangle inequality in the first step, and convexity of $t^{\ell+1}$ in the last step. We now utilize this in the integral form
    \begin{align*}
        (\rvx_i - \rvx_j)^\top (\rvg_i - \rvg_j) &= \int_0^1 (\rvx_i - \rvx_j)^\top \nabla^2 f(\rvx_j + t(\rvx_i - \rvx_j)) (\rvx_i - \rvx_j) \dd{t} \\
        &\ge \|\rvx_i - \rvx_j\|^2 \int_0^1 c_1\|\rvx_j + t(\rvx_i - \rvx_j)\|^\ell - b_1 \dd{t} \\
        &\ge  \left(c (\|\rvx_i\| + \|\rvx_j\|)^\ell - c'\right) \|\rvx_i - \rvx_j\|^2.
    \end{align*}
    Summing over $i,j$ and dividing by $2N^2$, we get
    \begin{align}
        \Tr K &\ge \frac{c}{2N^2} \sum_{i,j} \left(\|\rvx_i\| + \|\rvx_j\|\right)^\ell \|\rvx_i - \rvx_j\|^2 - c'T. \notag
        \intertext{Now choosing an index $\rvx_r$ satisfying $\|\rvx_r\|^2 \ge S,$}
        &\ge \frac{c}{2N^2} \sum_j S^{\ell/2} \|\rvx_r - \rvx_j\|^2- c'T \notag\\
        &\ge \frac{c}{2N^2}TS^{\ell/2} - c'T. \label{appeq:TrKLowerBound}
    \end{align}
    We can absorb the fixed factor $1/(2N)$ into $c$ in the last line, and write 
    $\Tr K \ge cTS^{\ell/2} - c'T$.
    
    We can similarly upper bound $\Tr K$: since $\|\rvx_i\|, \|\bar \rvx\|\le \sqrt{NS}$, we have for all $i$ and $t \in [0,1]$,
    \begin{equation*}
        \|\nabla^2 f(t\rvx_i + (1-t) \bar \rvx)\| \le c(1 + S^{\ell/2}).
    \end{equation*}
    Cauchy--Schwarz yields
    \begin{equation}
        |\Tr K| \le c T (1+S^{\ell/2}),\quad |P| \le cT(1+S^{\ell/2}).
    \end{equation}
    This gives lower bounds on $D$: since 
    \begin{equation*}
        (\Tr K)^2 \le \left(\frac{1}{N} \sum_i \boldsymbol{\delta}_i^\top C^{-1} \boldsymbol{\delta}_i\right)  \left(\frac{1}{N} \sum_i \rvg_i^\top C \rvg_i\right) = dD,
    \end{equation*}
    we obtain the following bound \textit{for sufficiently large $S$,}
    \begin{equation}
        D \ge cT^2 S^\ell. \label{appeq:DandPBound}
    \end{equation}
    We can also bound $P$ everywhere:
    \begin{align}
        |P| &\le c|\Tr K| +  \left|\frac{1}{N} \sum_i C\!:\! \nabla^2 f(\rvx_i) \right|\notag\\
        &\le c\sqrt{D} + c{T} (1+S^{\ell/2}) \le c(1+\sqrt{D}) .\label{appeq:Pbound}
    \end{align}
    
    The final needed bound is for $\Tr(KC^{-1})$. Since $\sum_i \boldsymbol\delta_i = 0$, and utilizing the Hessian bound between $\rvx_i$ and $\bar \rvx$,
    \begin{align}
        |\Tr(KC^{-1})| &= \left|\frac{1}{N} \sum_i (\rvg_i - \nabla f(\bar \rvx))^\top C^{-1} \boldsymbol \delta_i\right| \notag \\
        &\le \left(\frac{1}{N}\sum_i \boldsymbol\delta_i^\top C^{-2} \boldsymbol\delta_i\right)^{1/2} \left(\frac{1}{N} \sum_i \|\boldsymbol\delta_i\|^2 c(1+S^{\ell/2})^2\right)^{1/2}\notag \\
        &\le c\sqrt{TJ}(1+S^{\ell/2}). \label{appeq:trkcinv}
    \end{align}
    Since $C \succeq J^{-1}I_d$, we have the lower bound \textit{for sufficiently large $S$,}
    \begin{equation}
        D \ge \frac{1}{J} \frac{1}{N} \sum_i \|\nabla f(\rvx_i)\|^2 \ge \frac{cS^{\ell+1} - C}{J} \ge c\frac{S^{\ell+1}}{J}, \label{appeq:Dlowerbound}
    \end{equation}
    where the second inequality follows from extending the lower bound \labelcref{appeq:dfbound}. We can now proceed to bound $\gL W_a/W_a$. Let us define $\theta \coloneqq T/S \in [0,1]$, and first work where $S$ is sufficiently large such that the previous inequalities hold. Using \labelcref{appeq:trkcinv,appeq:Dlowerbound}, we have the bounds
    \begin{equation*}
        \frac{\Tr(KC^{-1})^2}{J^2} \le c\frac{TS^{\ell}}{J} = c\theta \frac{S^{\ell+1}}{J} \le c \theta D.
    \end{equation*}
    Furthermore, since $|P| \le c\sqrt{D}$,
    \begin{align*}
        \frac{\gL W_a}{W_a} &\le -\lambda - c_a D + aP + 2 \left|\frac{\Tr(KC^{-1})}{J}\right|\\
        &\le -\lambda -c_a D + ca(1+\sqrt{D}) + 2\left|\frac{\Tr(KC^{-1})}{J}\right| \\
        &\le -\lambda - \frac{c_a}{2} D + ca + 2 \left|\frac{\Tr(KC^{-1})}{J}\right|.
    \end{align*}
    Take $a \in (0, N/2)$ to be sufficiently small such that $ca \le \lambda/4$.
    
    \textit{Case 1.} $\theta \le a^2$. Since $c_a \ge a/2$, the following inequality holds
    \begin{align*}
        -\frac{c_a}{2} D + 2\left|\frac{\Tr(KC^{-1})}{J}\right| \le -\frac{a}{4}D + 2 ca\sqrt{D} \le c'a
    \end{align*}
    for some constant $c'$ independent of $a$. By further shrinking $a$ if necessary such that $c'a \le \lambda/4$, we have the uniform drift $\gL W_a/W_a \le -\lambda/2$ here.

    \textit{Case 2.} $\theta \ge a^2$. From \labelcref{appeq:DandPBound},
    \begin{equation*}
        D \ge cT^2 S^\ell = c\theta^2 S^{\ell+2} \ge ca^4 S^{\ell+2}.
    \end{equation*}
    In particular, $D \rightarrow \infty$ as $S \rightarrow \infty$. Then the remaining terms
    \begin{align*}
        -\frac{c_a}{2} D + 2 \left|\frac{\Tr(KC^{-1})}{J}\right| \le - \frac{c_a}{2} D + 2\sqrt{cD},
    \end{align*}
    since $\theta \in [0,1]$. This is bounded above by $\lambda/4$ for sufficiently large $S$. Therefore $\gL W_a/W_a \le -\lambda/2$.
    
    \textit{Case 3.} $S$ is bounded, say $S \le S_0$. From \labelcref{appeq:trkcinv}, we have the upper bound for some constant $C_{S_0}$ depending on $S_0$, 
    \begin{align*}
        \frac{\gL W_a}{W_a} \le -\frac{3\lambda}{4} + 2 C_{S_0} J^{-1/2}
    \end{align*}
    which is bounded by $-\lambda/2$ for sufficiently large $J \ge J_0$.
    
    The remaining set is of the form $\{S \le S_0, J \le J_0\}$, which is compact and bounded in $\mathrm{M}$. The norm-like condition on $W_a$ follows similarly to before from coercivity of $f$. A similar argument to \Cref{thm:ctsGeoErgo} shows the desired Foster--Lyapunov condition.
\end{proof}

\section{Relaxing the particle condition}\label{appsec:dplus2}
We now detail how to relax the particle conditions from $N \ge d+3 $ to $N \ge d+2$. This is based on a modified Lyapunov function: for any $\kappa \in (0,1)$ and $a>0$, we consider
\begin{equation}\label{eq:VakLyapunov}
    V_{a,\kappa} \coloneqq e^{\frac{a}{N} \sum f(\rvx_i)} \Tr(C^{-1})^\kappa.
\end{equation}

This is similar to the Lyapunov function \labelcref{eq:LyapunovWa} with an additional power on the inverse covariance term, which will give us a small additional drift in the critical regime $N = d+2$. We will now prove that the previous Foster--Lyapunov statements hold for any fixed $\kappa \in (0,1)$, with the sufficiently small parameter $a$ depending on $\kappa$, as well as $d,N$, and the regularity of $f$.

\subsection{Lemmas}
We first need the generator computations. Recall that $\nabla_i J = -2N^{-1} C^{-2} (\rvx_i - \bar \rvx)$, therefore $\Gamma(J,J) = 4N^{-1} \Tr(C^{-2})$.
\begin{lemma}
    The generator of \labelcref{eq:aldiDynamics} applied to $V_{a,\kappa}$ satisfies
    \begin{multline}
        \frac{\gL V_{a,\kappa}}{V_{a,\kappa}} = \left(-a + \frac{a^2}{N}\right)D+aP
        +\kappa\left(-2+\frac{2d+4}{N}\right)\\ 
        +2\kappa\left(1-\frac{2a}{N}\right)\frac{\Tr(KC^{-1})}{J}
        -\frac{4\kappa(1-\kappa)}{N}\frac{\Tr(C^{-2})}{J^2}.
    \end{multline}
\end{lemma}
\begin{proof}
    From \Cref{lem:generatorE,lem:GeneratorJ} and using the carr\'e-du-champ operator, we can compute
    \begin{align*}
        \gL V_{a,\kappa} &= J^\kappa \gL e^{aE}  + 2 \Gamma(e^{aE}, J^\kappa)  + e^{aE} \gL (J^\kappa)\\
        &= V_{a,\kappa} \left(a \gL E + \frac{a^2}{N^2} \sum_i \rvg_i^\top C \rvg_i \right) + 2 \sum_i \frac{a}{N} \rvg_i^\top e^{aE} C\kappa(\nabla_i J) J^{\kappa-1} \\
        &\quad + e^{aE} \left(\kappa J^{\kappa-1} \gL J + \kappa(\kappa-1) J^{\kappa-2} \sum_i (\nabla_i J)^\top C (\nabla_i J)\right) \\
        &= V_{a,\kappa}\left[a(-D+P) + \frac{a^2}{N} D\right] - \frac{4a\kappa}{N}V_{a,\kappa} \frac{\Tr(KC^{-1})}{J}\\
        &\quad + \kappa V_{a,\kappa} J^{-1} \left[2\Tr(KC^{-1}) + (-2+\frac{2d+4}{N})J\right] - \frac{4\kappa(1-\kappa)}{N} \frac{\Tr(C^{-2})}{J^2}V_{a,\kappa}.
    \end{align*}
    Rearranging gives the desired identity.
\end{proof}
The key is the additional last control term. From Cauchy--Schwarz, we have $\Tr(C^{-2}) J^{-2} \ge d^{-1}$. This term therefore gives an additional negative control when $N = d+2$. Recall the following constant definitions and inequalities for $N \ge d+2$ and $a \in (0, N/2)$:
\begin{align*}
    c_a &\coloneqq a\left(1-\frac{a}{N}\right)>0,\\
    \beta_{a,\kappa} &\coloneqq 2\kappa\left(1-\frac{2a}{N}\right)>0,\\
    \lambda_{\kappa} &\coloneqq 2\kappa\left(1-\frac{d+2}{N}\right)+\frac{4\kappa(1-\kappa)}{Nd}>0.
\end{align*}
We thus obtain
\begin{equation}\label{appeq:mainLyapunovVak}
    \frac{\gL V_{a,\kappa}}{V_{a,\kappa}} \le -\lambda_{\kappa} - c_a D + aP + \beta_{a,\kappa}\frac{\Tr(KC^{-1})}{J}.
\end{equation}

We also need to show that $V_{a,\kappa}$ is norm-like for any $a>0$. Similarly to \Cref{cor:geoErgodicWa}, we have that $V_{a,\kappa} \rightarrow \infty$ as $\rmX \rightarrow \partial \mathrm{M}$. Furthermore, 
\begin{equation*}
    V_{a,\kappa} = e^{\frac{a}{N}\sum_i f(\rvx_i)} \Tr(C^{-1})^\kappa  \ge \frac{d^{2\kappa} e^{\frac{a}{N}\sum_i f(\rvx_i)}}{\Tr(C)^\kappa } \ge \frac{d^{2\kappa} N^\kappa e^{\frac{a}{N}\sum_i f(\rvx_i)}}{\|\rmX\|_F^{2\kappa}} \gtrsim\frac{e^{\frac{a}{N}\sum_i f(\rvx_i)}}{\left[\sum_i f(\rvx_i) + b_1\right]^\kappa}.
\end{equation*}
Coercivity of $f$ gives that $V_{a,\kappa} \rightarrow \infty$ as $\rmX \rightarrow \infty$, as desired.

\subsection{Modifications for Theorem \ref{thm:ctsGeoErgo}}
Since $\lambda_\kappa>0$, directly replace \labelcref{eq:WaUnsimplifiedBound} with \labelcref{appeq:mainLyapunovVak}. All other arguments are identical for a fixed $\kappa$.

\subsection{Modifications for Theorem \ref{thm:timeScaleErgo}}
We now compute $\bar \gL V_{a,\kappa}$. The desired carr\'e-du-champ is 
\begin{align*}
    \Gamma(\Tr C, V_{a,\kappa}) &= \sum_i (\nabla_i \Tr C)^\top C \nabla_i (e^{aE} J^\kappa ) \\
    &= \frac{2}{N} \sum_i (\rvx_i - \bar \rvx)^\top C \left[\frac{a}{N} \rvg_i V_{a,\kappa} - \frac{2\kappa}{N} C^{-2} (\rvx_i - \bar \rvx)V_{a,\kappa} J^{-1}\right] \\
    &= V_{a,\kappa} \left[\frac{2a}{N} \Tr(KC) - \frac{4d\kappa}{N}J^{-1}\right].
\end{align*}

Using \Cref{lem:generatorALDITimescale}, we get
\begin{align}
    \frac{\bar \gL V_{a,\kappa}}{V_{a,\kappa}} &= \frac{\gamma \gL V_{a,\kappa} + \gamma' \Gamma(\Tr C, V_{a,\kappa})}{V_{a,\kappa}}\notag \\
    &\le \gamma \left\{-\lambda_{\kappa} - c_a D + aP + \beta_{a,\kappa}\frac{\Tr(KC^{-1})}{J}\right\}   + |\gamma'| \left[\frac{2aL}{N}T^2 + \frac{4d\kappa}{NJ}\right].
\end{align}
Compare this with \labelcref{eq:LBarWaBound}. Since $\lambda_{\kappa}>0$ for $N \ge d+2$, the same arguments hold.

\subsection{Modifications for Theorem \ref{thm:discreteTimeScaleErgo}}
The main change is that in \Cref{sssec:ineqConstruction}, the expression to bound is replaced with 
\begin{equation*}
    \frac{Q_\eta V_{a,\kappa}}{V_{a,\kappa}}(\rmX) = \mathbb{E} \left[e^{a(E_+ - E)} \left(\frac{J_+}{J}\right)^\kappa \right].
\end{equation*}
\Cref{sssec:normalizing} is identical. For \Cref{sssec:expectation}, we apply \Cref{applem:inverseMatBound} with $A = C^{-1}/J$ which is psd with unit trace, and the same $R_\eta, U_\eta$. Since $J_+ = \Tr(C^{-1}H_{\eta}^{-1})$, this gives
\begin{align*}
    \hat{\mathbb{E}} [(J_+/J)^\kappa] &= \hat{\mathbb{E}}[\Tr (C^{-1} H_\eta^{-1}/J)^\kappa]\\ 
    &= 1 + \eta \kappa \gamma \left(-2 + \frac{2d+4}{N}\right) + \eta \kappa \gamma \left(2 - \frac{4a}{N}\right) \frac{\Tr(KC^{-1})}{J}  \\
    &\qquad - \frac{4d\eta \kappa \gamma'}{NJ}- \frac{4 \eta \kappa(1-\kappa)}{N} \Tr(R_\eta C^{-2})J^{-2} + \mathcal{O}(\eta^{3/2}).
\end{align*}
Since $\gamma I_d \preceq R_\eta \preceq 2\gamma I_d$ for sufficiently small $\eta$, we have $\Tr(R_\eta C^{-2})J^{-2} \ge \gamma \Tr(C^{-2})J^{-2}  \ge \gamma/d$. Therefore, 

\begin{align}
    \frac{Q_\eta V_{a,\kappa}}{V_{a,\kappa}} &\le \exp \left(\eta \gE_{a,\kappa} + \kappa_1 \eta^2 (1+\gamma D)\right) (1 + \eta\kappa \gJ_{a,\kappa} + \kappa_2 \eta^{3/2}),\\
    \gE_{a.\kappa} &\coloneqq -c_a \gamma D + \frac{a}{N}(d+1+N) \gamma LT + \frac{2aL}{N} |\gamma'|T^2,\\
    \gJ_{a,\kappa} &\coloneqq \left(-2 + \frac{2d+4}{N} - \frac{4(1-\kappa)}{Nd}\right) \gamma + \left(2 - \frac{4a}{N}\right) L \sqrt{\frac{T}{J}} \gamma + \frac{4 d|\gamma'|}{NJ}. 
\end{align}
In particular, the coefficient of $\gamma$ in $\gJ_{a,\kappa}$ is $-\kappa^{-1}\lambda_\kappa$, negative for $N \ge d+2$. A similar argument in \Cref{sssec:discreteFosterLyapunov} concludes the desired Foster--Lyapunov condition, replacing $2 - (2d+4)/N$ with $\lambda_\kappa$.

\section{Spectral gap and exponential convergence}\label{appsec:spectralgap}
The Foster--Lyapunov estimates also give exponential convergence in $L^2(\Pi_*)$. This technique can be found for example in \cite{bakry2008simple,bakry2008rate}.

\begin{proposition}\label{prop:covariancePoincare}
    Fix $N\ge d+2$ and assume \Cref{assmp:distantConvexityND}, or the tail assumptions of \Cref{rem:polyTailsBound}. Then the dynamics \labelcref{eq:aldiDynamics} have an $L^2(\Pi_*)$ spectral gap: there exists $\lambda>0$ such that
    \begin{equation}\label{eq:diffusionLtwoGap}
        \|P^t g-\Pi_*g\|_{L^2(\Pi_*)}
        \le e^{-\lambda t}\|g-\Pi_*g\|_{L^2(\Pi_*)},
        \qquad g\in L^2(\Pi_*).
    \end{equation}
    In particular, if the initial law has density $h_0\in L^2(\Pi_*)$ relative to $\Pi_*$, its density $h_t$ at time $t$ satisfies
    \[
        \|h_t-1\|_{L^2(\Pi_*)}
        \le e^{-\lambda t}\|h_0-1\|_{L^2(\Pi_*)}.
    \]
\end{proposition}
\begin{proof}
    Recall that $\gL$ has the representation $\gL V = \Pi_*^{-1} \sum_i \nabla_i \cdot (\Pi_* C \nabla_i V)$. Denote the energy by
    \[
        \mathcal E(g)=\int\Gamma(g,g) \dd{\Pi_*} = \int\sum_i(\nabla_i g)^\top C\nabla_i g\dd\Pi_*.
    \]
    From the Lyapunov estimates, including \Cref{appsec:dplus2}, we have a Lyapunov function $V>0$ and a compact $\mathrm{K}\subset\mathrm{M}$ satisfying
    \[
        -\frac{\gL V}{V}\ge c-b\mathbf1_{\mathrm{K}},
        \qquad c>0,\quad b\ge0.
    \]
    For $u\in\gC_c^\infty(\mathrm{M})$, integration by parts gives

    \begin{align*}
        \int-\frac{\gL V}{V}u^2\,\dd\Pi_*
        &=\int\sum_i(\nabla_i V)^\top C\nabla_i\left(\frac{u^2}{V}\right)\,\dd\Pi_*\\
        &=\int\sum_i\left[
            \frac{2u}{V}(\nabla_i u)^\top C\nabla_i V
            -\frac{u^2}{V^2}(\nabla_i V)^\top C\nabla_i V
        \right]\,\dd\Pi_*\\
        &= \mathcal E(u)-\int\sum_i
        \left\|C^{1/2}\left(\nabla_i u-\frac{u}{V}\nabla_i V\right)\right\|^2\,\dd\Pi_*\\
        &\le\mathcal E(u).
    \end{align*}

    Combining this estimate with the integrated drift inequality gives
    \begin{equation}\label{eq:LyapunovFormBound}
        c\int u^2\,\dd\Pi_*
        \le\mathcal E(u)+b\int_{\mathrm K}u^2\,\dd\Pi_*.
    \end{equation}

    By density of $\gC_c^\infty(\mathrm M)$ in the weighted Sobolev space associated with $\mathcal E$ (for example see \cite{grigoryan2009heat,grigoryan2013parabolicity}), \labelcref{eq:LyapunovFormBound} extends to every $u\in L^2(\Pi_*)$ with $\mathcal E(u)<\infty$.
    
    Now control the integral over $\mathrm K$. Fix a bounded connected open $D$ with $\mathrm{K}\subset D\subset\mathrm{M}$. On $\overline D$ and Lipschitz boundary. On $D$ the density of $\Pi_*$ is bounded above and below by positive constants, and $C$ is uniformly positive definite. Combining this with the standard Poincar\'e-Wirtinger inequality on $D$, it holds for any $g \in W^{1,2}(D)$,
    \[
        \int_D(g-\bar g_D)^2\,\dd\Pi_*\le C_D\mathcal E(g),
        \qquad \bar g_D=\frac1{\Pi_*(D)}\int_Dg\,\dd\Pi_*,
    \]
    for some $C_D<\infty$. Setting $u=g-\bar g_D$ in \labelcref{eq:LyapunovFormBound}, using $\mathrm{K}\subset D$ and $\mathcal E(u)=\mathcal E(g)$, we get the Poincar\'e inequality
    \begin{equation}\label{eq:covariancePoincare}
        \operatorname{Var}_{\Pi_*}(g)
        \le\int (g-\bar g_D)^2\,\dd\Pi_*
        \le\frac{1+bC_D}{c}\mathcal E(g).
    \end{equation}
    This is the Poincar\'e inequality for the diffusion, which implies a spectral gap with $\lambda=c/(1+bC_D)$. More precisely, for $g\in L^2(\Pi_*)$ and $t>0$, applying it to $P^t g$ gives
    \begin{align*}
        \frac{\dd}{\dd t}\|P^t g-\Pi_*g\|_2^2
        &=-2\mathcal E(P^t g)\\
        &\le-2\lambda\|P^t g-\Pi_*g\|_2^2.
    \end{align*}
    Integrating this gives \labelcref{eq:diffusionLtwoGap}. Finally, reversibility gives $h_t=(P^t)^*h_0=P^t h_0$, so the same estimate applies to densities.
\end{proof}
Under \Cref{assmp:distantConvexityND}, the same conclusion holds for the trace-regularized diffusion, using its respective Foster--Lyapunov estimate.

\end{appendix}

\bibliographystyle{plain} 
\bibliography{refs}       

\end{document}